\documentclass[11pt,reqno]{amsart} 
\usepackage[utf8]{inputenc}
\usepackage{mathptmx}
\usepackage{amsmath}
\usepackage{amscd}
\usepackage{dsfont}
\usepackage{amssymb, bm}
\usepackage{amsthm}
\usepackage{eufrak}
\usepackage{yfonts}
\usepackage{mathrsfs}
\usepackage{xspace}
\usepackage[all,tips]{xy}
\usepackage[dvips]{graphicx}
\usepackage{verbatim}
\usepackage{syntonly}
\usepackage{hyperref}
\usepackage{graphics, fullpage,color, epsfig,url}
\usepackage{indentfirst}
\usepackage{esint}
\usepackage{enumitem}
\usepackage[dvipsnames]{xcolor}
\usepackage{tensor}
\usepackage{amsrefs, eucal}
\usepackage{bbm}
\usepackage{ulem}
\numberwithin{equation}{section}

\providecommand{\MR}{\relax\ifhmode\unskip\space\fi MR }

\theoremstyle{plain}
\newtheorem{thm}{Theorem}
\newtheorem{lem}[thm]{Lemma}

\newtheorem{prop}[thm]{Proposition}
\newtheorem{defn}[thm]{Definition}
\newtheorem{cor}[thm]{Corollary}

\theoremstyle{remark}
\newtheorem{rem}[thm]{Remark}

\numberwithin{thm}{section}
\newenvironment{pf}
{\begin{proof}} {\end{proof}}

\newcommand{\disp}{\displaystyle}

\DeclareMathOperator{\supp}{supp}
\DeclareMathOperator{\di}{div}

\newcommand{\eps}{\varepsilon}
\newcommand{\vp}{\varphi}

\newcommand{\al}{\alpha}
\newcommand{\be}{\beta}
\newcommand{\ga}{\gamma}
\newcommand{\de}{\delta}
\newcommand{\De}{\Delta}
\newcommand{\Ga}{\Gamma}
\newcommand{\te}{\theta}

\newcommand{\om}{\omega}
\newcommand{\Om}{\Omega}
\newcommand{\si}{\sigma}

\newcommand{\nid}{\noindent}

\newcommand{\iny}{\infty}
\newcommand{\del}{\partial}
\newcommand{\su}{\subset}
\newcommand{\LP}{\Delta}
\newcommand{\gr}{\nabla}

\newcommand{\dGnt}{x_0^a \, \mathcal{G}_n(X,t) dX}

\newcommand{\dGnta}{x_0^a \, \mathcal{G}_n(X,\tau) dX}
\newcommand{\dGt}{x_0^a \, \mathcal{G}(X,t) dX}
\newcommand{\dGta}{x_0^a \, \mathcal{G}(X,\tau) dX}

\newcommand{\norm}[1]{\left\| #1\right\|}

\newcommand{\abs}[1]{\left\vert #1 \right\vert}

\newcommand{\set}[1]{\left\{#1\right\}}
\newcommand{\brac}[1]{\left[#1\right]}
\newcommand{\pr}[1]{\left( #1 \right) }
\newcommand{\pb}[1]{\left( #1 \right] }

\newcommand{\BB}[1]{\ensuremath{\mathbb{#1}}}

\newcommand{\Sp}[1]{\ensuremath{\BB{S}^{#1}}}
\newcommand{\aSp}[1]{\ensuremath{\abs{\BB{S}^{#1}}}}

\newcommand{\N}{\ensuremath{\BB{N}}}

\newcommand{\R}{\ensuremath{\BB{R}}}

\newcommand{\D}{\ensuremath{\BB{D}}}

\def\div{\mathop{\operatorname{div}}}

\title[Fractional elliptic to parabolic]{Fractional Parabolic Theory as a High-Dimensional Limit of  \\ Fractional Elliptic Theory}

\author{Blair Davey}
\address[B. Davey]{Department of Mathematical Sciences, Montana State University, Bozeman, MT 59717}
\email{\textcolor{blue}{\href{mailto:}{blairdavey@montana.edu}}}

\author{Mariana Smit Vega Garcia}
\address[M. Smit Vega Garcia]{Department of Mathematics, Western Washington University, Bellingham, WA 98225}
\email{\textcolor{blue}{\href{mailto:}{smitvem@wwu.edu}}}

\date{}

\begin{document}

\begin{abstract}
This paper continues the program that was initiated in \cite{Dav18} and continued in \cite{DSVG24}, where a high-dimensional limiting technique was developed and used to prove certain parabolic theorems from their elliptic counterparts.
The articles \cite{Dav18} and \cite{DSVG24} address the constant-coefficient and variable-coefficient settings, respectively. 
Here, we focus on fractional operators. 
As shown in \cite{CS07}, \cite{NS16}, \cite{ST17}, fractional operators may be associated with certain degenerate operators via extension problems, so we study the corresponding class of degenerate operators. 
Our high-dimensional limiting technique is demonstrated through new proofs of three theorems for degenerate parabolic equations. 
Specifically, we establish the monotonicity of Almgren-type, Weiss-type, and Alt-Caffarelli-Friedman-type functionals in the degenerate parabolic setting. 
Each new parabolic proof in this article is based on a (new) related elliptic theorem and a careful limiting argument that is reminiscent of those from \cite{Dav18} and \cite{DSVG24}. 
Our proof of the degenerate parabolic Weiss-type monotonicity formula additionally uses an epiperimetric inequality for weakly $a$-harmonic functions, which we also prove.
To the best of our knowledge, our Alt-Caffarelli-Friedman monotonicity result is new.

\end{abstract}

\maketitle

\section{Introduction}

The goal of this paper is to generalize the ideas developed in \cite{Dav18} and \cite{DSVG24} to fractional operators. 
Fractional powers of the Laplacian were first introduced by Riesz in \cite{Rie38} and \cite{Rie49}. 
Since then, fractional operators such as $(-\Delta)^s$ and $(\partial_t-\Delta)^s$ have played a significant role in many applied problems in fields such as fluid dynamics, elasticity, and quantum mechanics, for example. 
In particular, the celebrated Signorini problem in elasticity is equivalent to the obstacle problem for $(-\Delta)^{1/2}$; see \cite{ACS08}, \cite{AC04}, \cite{CSS08}, \cite{CS07}, \cite{GP09}, \cite{GPSVG16}. 
Another important example is the obstacle problem for the fractional heat equation $(\partial_t - \Delta)^{1/2}$, which models osmosis; see \cite{DGPT17} and \cite{DL76}. 
We refer the interested reader to \cite{Gar19} and the references therein for a discussion of many applications of fractional operators. 

In this paper, we focus on (backward) fractional parabolic operators of the form $(-\del_t-\Delta)^s$.
The Caffarelli-Silvestre extension \cite{CS07} shows that the fractional Laplacian, $(-\LP)^s$, is related to a certain degenerate elliptic equation.
In the parabolic setting, this extension procedure, developed independently by Nystr\"{o}m-Sande \cite{NS16} and Stinga-Torrea \cite{ST17}, allows us to study parabolic fractional operators through their degenerate counterparts.
Throughout this article, we view fractional operators through their extension problems, and in doing so, we focus our study on the related degenerate equations. 
In particular, we prove theorems that hold for solutions to degenerate parabolic equations as the high-dimensional limit of the corresponding theorems for related degenerate elliptic equations. 

The idea that parabolic theory can be realized as a high-dimensional limit of elliptic theory has appeared in many other contexts.
Perelman in \cite{Per02} used this philosophy in the development of his reduced volume. 
This idea was also examined by Tao \cite{Tao08}, modified in the course notes of Sverak \cite{Sve11}, then developed and applied in \cite{Dav18} by the first-named author. 
Similar ideas have appeared in the literature; see \cite{CS05}*{Section 12.2}, for example.
Very recently, this technique has been used by other authors in \cite{BR25}, \cite{Bus25}, and \cite{Sta25}, while a similar philosophy appears in \cite{DED24}.

We now give a general overview of what we mean by a ``high-dimensional limiting argument."
Let $L$ denote an elliptic operator defined in space $\Om$, and for each $n \in \N$, let $L_n$ denote a related elliptic operator defined in some high-dimensional space $\Om_n$.
Our goal is to prove a parabolic theorem of the following form: 
$$\text{If $U(X, t)$ satisfies the parabolic equation $(\del_t + L) U = 0$ in $\Om \times (0, T)$, then $P(U)$ holds.}$$
To start, we construct transformation maps of the form $F_n : \Om_n \to \Om \times \R_+$ that send points $Y \in \Om_n$, the high-dimensional elliptic setting, to points $(X, t) \in \Om \times \R_+$, the space-time (parabolic) setting.
For each $n \in \N$, we define functions $V_n(Y)$ in $\Om_n$ by setting $V_n(Y) = U(F_n(Y)) = U(X,t)$.
If we choose $F_n$ appropriately, then $L_n V_n = K_n$, where $K_n$ is an error term.
That is, each $V_n$ solves a nonhomogeneous elliptic equation.
Then we prove elliptic theorems of the following form:
$$\text{If $V(Y)$ satisfies the nonhomogeneous elliptic equation $L_n V = K$ in $\Om_n$, then $E_n(V, K)$ holds.}$$
Using certain pushforward computations, we can transform each elliptic statement, $E_n(V_n, K_n)$, into a parabolic statement of the form $P_n(U)$.
Finally, we take limits and show that $P_n(U) \to P(U)$ as $n \to \iny$.

Our high-dimensional limiting argument has previously appeared in \cite{Dav18} and \cite{DSVG24}.
In \cite{Dav18}, the first-named author proved parabolic theorems for solutions to $\pr{\del_t + \Delta} u = 0$ by taking high-dimensional limits of the corresponding results for solutions to elliptic equations of the form $\Delta v=k$. 
In the language of the previous paragraph, $\Om = \R^d$, $\Om_n = \R^{dn}$, and $L = L_n = \De$ (in the appropriate dimensions).
In \cite{DSVG24}, we generalized this technique to $L = \di (A \gr)$ using $L_n = \di \pr{\kappa_n \gr}$, where $A$ is a matrix function and each $\kappa_n$ is a related scalar function.
More specifically, we proved theorems for solutions to variable-coefficient parabolic equations of the form $\partial_t u + \text{div}(A\nabla u) = 0$ by applying the limiting argument to the related results for elliptic equations of the form $\text{div}(\kappa \nabla v)=\kappa l$. 
In the current paper, we study degenerate parabolic equations through high-dimensional limits of their nonhomogeneous degenerate elliptic counterparts.
To do this, we build the framework described above with $\Om = \R^{d+1}_+$, $\Om_n = \R^{dn+1}_+$, and $L = L_n = L_a$ (in the appropriate dimension), where $L_a := \di \pr{x_0^a \gr }$, $a \in (-1,1)$, denotes our degenerate operator.
Compared to the results of \cite{Dav18} or \cite{DSVG24}, new ideas are needed to build this framework, and the techniques used here are substantially modified to account for the degeneracy of the operator.
As described above, the extension techniques of \cite{CS07}, \cite{NS16}, \cite{ST17} relate our results for the degenerate setting to their fractional counterparts. 

Our first main results are Almgren-type monotonicity formulas. 
To start, we prove an Almgren-type monotonicity formula for solutions to nonhomogeneous elliptic equations arising from the extension problem, see Theorem \ref{monoThm}. 
This formula was first proved by Caffarelli-Silvestre in the homogeneous elliptic setting in the groundbreaking work \cite{CS07}. 
In the context of the Signorini problem, the monotonicity of a time-independent Almgren-type frequency function was initially proved by Garofalo-Petrosyan in \cite{GP09} for the case $a=0$, and then extended to the whole range $a\in(-1,1)$ in \cite{GR19}.  
In general, in the elliptic setting, Almgren-type monotonicity formulas have been extensively used to study a wide variety of free boundary problems, see for example \cite{ACS08},  \cite{BBG22},  \cite{DT23}, \cite{GPSVG16}, \cite{GPSVG18},  \cite{GSVG14}, \cite{Gui09}, \cite{JPSVG24}. 

When $a=0$, the parabolic Almgren monotonicity formula is attributed to Poon \cite{Poon96} who used it to prove unique continuation results for caloric functions. 
An analogous Almgren monotonicity formula was generalized by Stinga-Torrea in \cite{ST17} for general $a\in(-1,1)$. 
A version of this result was proved in the setting of the parabolic Signorini problem for $a=0$ in \cite{DGPT17}, where the authors used monotonicity to establish the optimal regularity of solutions and study the free boundary. 
More recently, the authors of \cite{BDGP21} and \cite{BDGP20} studied the free boundary in the thin obstacle problem for degenerate parabolic equations. 
One of the key ingredients used in those two papers is an Almgren-Poon monotonicity formula for solutions to degenerate parabolic equations. 
In \cite{AT24}, the authors also used an Almgren-type monotonicity formula to study nodal sets of solutions to nonlocal parabolic equations. 
We also mention \cite{BG18}, where the authors used an Almgren-type monotonicity formula to study unique continuation for degenerate parabolic equations, \cite{Zil14}, where the author proved an Almgren monotonicity formula for degenerate, elliptic systems, \cite{FPS}, where the authors use an Almgren-type monotonicity formula to obtain strong unique continuation for a class of fractional heat equations, and \cite{ABDG23}, where the authors use a Poon-type monotonicity to study fractional parabolic equations.

In this paper, we use Theorem \ref{monoThm} and a high-dimensional limiting argument to prove an Almgren-type monotonicity formula for solutions to degenerate parabolic equations, see Theorem \ref{pMono}.

Our next main results are a degenerate parabolic Weiss monotonicity formula (Theorem \ref{T:parabolicweiss}) and a new epiperimetric inequality for weakly $a$-harmonic functions (Theorem \ref{T:epi}). 
Both of these results were inspired by those originally obtained by Weiss in \cite{Wei98} and \cite{Wei99} in the context of the classical obstacle problem, but our strategies differ from the original ones. 
We first prove a Weiss monotonicity formula for elliptic (spatial) functions that are not necessarily solutions to any equation, see Theorem \ref{monoWeissThm}. 
Then we use our elliptic Weiss monotonicity formula, together with our epiperimetric inequality for weakly $a$-harmonic functions, to prove a version of Weiss' monotonicity formula for degenerate parabolic equations, see Theorem \ref{T:parabolicweiss}.
Compared to our two other parabolic monotonicity formulas, the techniques here are far more subtle, as explained in more detail at the beginning of Subsection \ref{SS:paraWeiss}.

Like Almgren's monotonicity formula, Weiss' monotonicity formula has also played a fundamental role in the study of a multitude of free boundary problems. 
In the context of the Signorini problem, Garofalo and Petrosyan initially proved the monotonicity of an elliptic Weiss-type frequency function when $a=0$ in \cite{GP09}. 
This result was extended to the whole range $a\in(-1,1)$ in \cite{GR19} by Garofalo and Ros-Oton. 
In \cite{All12}, Allen proved a fractional elliptic version of a Weiss monotonicity formula and used it to study the separation of a lower-dimensional free boundary problem. 
In \cite{DGPT17}, to study the parabolic Signorini problem, Danielli, Garofalo, Petrosyan, and To proved a family of Weiss monotonicity formulas.
In \cite{ALP15}, Allen, Lindgren and Petrosyan also relied on such a formula to study an elliptic two-phase fractional obstacle problem. 
Furthermore, in \cite{BDGP21} and \cite{BDGP20}, the authors proved a family of elliptic and parabolic (respectively) Weiss monotonicity formulas and used them to study the free boundary in the thin obstacle problem for degenerate parabolic equations.

Historically, Reifenberg first proved an epiperimetric-type inequality in \cite{Rei64} for minimal surfaces. 
Similar result were subsequently proved by Taylor \cite{Tay76} for soap-films and soap-bubbles minimal surfaces, by White \cite{Whi83} for tangent cones, by Chang \cite{Cha88} while studying branch points in two-dimensional area minimizing integral currents, and by De Lellis and Spadaro \cite{DLS11} for two-dimensional multiple-valued functions minimizing the generalized Dirichlet energy. 
Since Weiss' epiperimetric inequality for the classical obstacle problem, epiperimetric inequalities have been used extensively in the study of free boundary problems. 
In the thin obstacle problem, these inequalities played a crucial role in \cite{FS16} and \cite{GPSVG16} in the study of the regular set for variable coefficient problems.
In \cite{GPPSVG17}, an epiperimetric inequality was used in the context of the obstacle problem for the fractional Laplacian. 
In \cite{SV19}, the authors used a direct approach to prove a 2-dimensional epiperimetric inequality for the Alt-Caffarelli functional, and in \cite{ESV20}, a (log-)epiperimetric inequality for the Weiss energy was used to prove the uniqueness of blow-ups at isolated singularities. 
In \cite{CSV18}, the authors proved a logarithmic epiperimetric inequality for obstacle-type problems, and in \cite{CSV20}, the same authors proved direct epiperimetric inequalities for the thin obstacle problem. 
Our approach is inspired by that of \cite{BET20}, where the authors proved an epiperimetric inequality for harmonic functions. 
Recently, using techniques that are similar to ours, Carducci \cite{Car24} obtained epiperimetric inequalities in the obstacle problem for the fractional Laplacian. 

Our last main results are Alt-Caffarelli-Friedman-type monotonicity formulas. 
In the original work of Alt-Caffarelli-Friedman \cite{ACF84}, the authors studied two-phase elliptic free boundary problems. 
A key ingredient in their study of the regularity of the free boundary, and in their proof of the Lipschitz continuity of minimizers, is the so-called ACF monotonicity formula.
Parabolic versions of this formula were proved by Caffarelli \cite{Caf93} and Caffarelli-Kenig \cite{CK98}, and then used to show the regularity of solutions to parabolic problems. 
In \cite{CJK02}, Caffarelli, Jerison, and Kenig further extended these ideas and obtained a powerful uniform bound on the monotonicity functional, instead of a monotonicity result. 
In \cite{EP08}, Edquist and Petrosyan proved a parabolic counterpart of the almost monotonicity formula of Caffarelli, Jerison and Kenig. 
In the setting of fractional elliptic problems, Terracini, Verzini and Zilio considered a class of competition diffusion nonlinear systems involving the square root of the Laplacian in \cite{TVZ2}. 
Their work relied on an ACF monotonicity formula for certain subsolutions of the one- and two-phase problems.  
We use a similar notion of subsolution in \eqref{strangecondition}, although our setting requires us to consider nonhomogeneous problems, while \cite{TVZ2} addresses homogeneous ones. 
In \cite{TVZ14}, the same authors studied another class of non-linear competition-diffusion systems for the fractional Laplacian with $s\in (0,1)$. 
As in \cite{TVZ2}, a key ingredient in their study was an ACF monotonicity formula, once again for subsolutions in a certain sense. 
In his Ph.D. thesis, \cite{Zil14}, Zilio also proved an ACF monotonicity formula for subsolutions to the extended, degenerate elliptic problem arising from the fractional Laplacian. 
In \cite{TTS18}, Terracini and Vita once again used an ACF monotonicity formula while studying s-harmonic functions on cones. 

Our final result is a parabolic counterpart of the fractional ACF-monotonicity formulas of \cite{TVZ14}, \cite{TVZ2}, \cite{Zil14}, stated below in Theorem \ref{T:parabolicACF}. 
This result is also proved as a high-dimensional limit of a family of ACF-type formulas for solutions to nonhomogeneous degenerate elliptic problems, see Theorem \ref{C:ACFellipticnonzero}.

\subsection*{Organization of the article}

In Section \ref{S:fractional}, we introduce fractional operators, the Caffarelli-Silvestre extension, and its parabolic counterpart(s). 
In Section \ref{S:elliptictopara}, we introduce the transformations that connect the elliptic and parabolic settings. 
That is, we define maps $F_{n}: \R^{dn+1}_+\rightarrow \R^{d+1}_+\times\R_+$ that take points in the high-dimensional (elliptic) space to space-time (parabolic) points. 
We use random walks to motivate our definitions of $F_n$, then examine, via the chain rule, how these maps relate degenerate elliptic and parabolic operators. 
Section \ref{S:PushfR} examines the transformation maps from a measure-theoretic perspective.
By defining an appropriate probability measure in the elliptic setting, we perform pushforward computations to arrive at parabolic probability measures that approximate weighted Gaussian distributions.
We refer to these results as the ``Bridge Lemmas."
In Section \ref{S:IntRel}, we collect a number of applications of the Bridge Lemmas. 

The Almgren-type monotonicity formulas appear in Section \ref{S:almgren}; Subsection \ref{SS:EAlmgren} presents an Almgren-type monotonicity formula for solutions to nonhomogeneous degenerate elliptic equations, Subsection \ref{SS:PAlmgren} presents an Almgren-type monotonicity formula for solutions to degenerate parabolic equations. 
In Section \ref{S:Weiss}, we discuss Weiss monotonicity formulas and the epiperimetric inequality for weakly $a$-harmonic functions. 
Subsection \ref{SS:ellipticWeiss} presents the elliptic theory: a Weiss monotonicity formula in the degenerate elliptic setting, some general elliptic theory, and the epiperimetric inequality for weakly $a$-harmonic functions. 
The results of Subsection \ref{SS:ellipticWeiss} are used in Subsection \ref{SS:paraWeiss} to prove a Weiss monotonicity formula for solutions to degenerate parabolic equations.
Section \ref{S:ACF} contains our ACF monotonicity formulas; the result for solutions to nonhomogeneous degenerate elliptic equations is in Subsection \ref{SS:EllipACF} while the result for solutions to degenerate parabolic equations is in Subsection \ref{SS:ParaACF}.  
Finally, in Appendix \ref{AppA}, we present some technical limit lemmas that are used in our elliptic-to-parabolic arguments.
\subsection*{Acknowledgments}

This material is based upon work supported by the National Science Foundation under Grant No. DMS-1928930 and the National Security Agency under Grant No. H98230-23-1-0004 while the authors participated in a program hosted by the Simons Laufer Mathematical Sciences Institute (formerly MSRI) in Berkeley, California, during the Summer of 2023.
B. D. was also partially supported by the NSF CAREER DMS-2236491, NSF LEAPS-MPS grant DMS-2137743, and a Simons Foundation Collaboration Grant 430198.
M.S.V.G was also partially supported by the NSF grants DMS-2054282, DMS-2348739 and a Karen EDGE fellowship. 
The authors thank the anonymous referees who provided useful comments on an earlier version of the manuscript.

\section{Fractional and Degenerate Operators}
\label{S:fractional}

To introduce our main operators and the notation that we use throughout, we start by discussing fractional elliptic and fractional parabolic operators.
Via the Caffarelli-Silvestre extension \cite{CS07} and its parabolic analogues in \cite{NS16} and \cite{ST17}, we also introduce the associated degenerate operators.

\subsection{Elliptic operators}
\label{SS:fractional}

For $s \in (0,1)$, the fractional Laplacian, denoted by $\pr{-\LP}^s$, acts on spatial functions $v(y)$, where $y \in \R^N$.
There are a number of different interpretations of this operator, some of which we now discuss.

The Fourier approach shows that the fractional Laplacian can be viewed as a symmetric differentiation operator of order $2s$.
In particular, denoting the Fourier transform with $\mathcal{F}$, we see that for $\eta \in \R^N$,
\begin{equation}
\label{EFourier}
\mathcal{F}(\pr{-\LP}^s v)\pr{\eta} = \abs{\eta}^{2s} \mathcal{F}(v) \pr{\eta}.
\end{equation}
As a singular integral operator,
\begin{equation}
\label{ESIO}
\pr{-\LP}^s v\pr{y} = C_{N,s} \int_{\R^N} \frac{v\pr{y} - v\pr{y'}}{\abs{y-y'}^{N+2s}} dy',
\end{equation}
where the integral is understood as a principal value. 

Another way to understand the fractional Laplacian is through the Caffarelli-Silvestre extension, \cite{CS07}.
To denote the half-space in $\R^{N+1}$, we write $\R^{N+1}_+ = \set{Y = \pr{y_0, y} : y_0 > 0,  \ y \in \R^N}$.
With $a \in (-1,1)$, let $V(Y) = V\pr{y_0,y}$ be a solution to the elliptic problem
\begin{equation*}
\left\{\begin{array}{ll}  
\LP_y V + \frac{a}{y_0} \del_{y_0} V + \del_{y_0}^2 V = 0 & \text{ for } \, Y \in \R^{N+1}_+ \\
\disp \lim_{y_0 \to 0^+} V\pr{y_0,y} = v\pr{y} & \text{ for } \, y \in \R^{N}.
\end{array}  \right.
\end{equation*}
We use the notation $\disp \del_{y_0} = \frac{\del }{\del y_0}$, $\disp \gr_y = \pr{\frac{\del }{\del y_1}, \ldots, \frac{\del }{\del y_N}}$, and $\LP_y = \gr_y \cdot \gr_y$.
The solution function $V(Y)$ is given explicitly using the Poisson formula
\begin{equation*}
V\pr{y_0, y} = \int_{\R^N} P_a\pr{y_0, y'} v\pr{y - y'} dy',
\end{equation*}
where the Poisson kernel is defined as $P_a\pr{Y} = C_{N,a} y_0^{1-a} \abs{Y}^{-(N+1-a)}$.
It follows that if $a = 1 - 2s$, then
\begin{equation}
\label{EExt}
c_{N,s} \pr{-\LP}^s v\pr{y} = - \partial_{\nu}^a V(0, y),
\end{equation}
where  
we introduced the notation
\begin{equation}
\label{flatDerivV}
    \partial_{\nu}^a V(0, y) 
    := \lim\limits_{y_0 \rightarrow 0+}y_0^a\partial_{y_0} V(y_0, y).
\end{equation}

Here and throughout, the open ball of radius $r$ in $\R^{N+1}$ is denoted by
$\BB{B}^{N+1}_r = \set{Y \in \R^{N+1} : \abs{Y} < r}$, while the open half-ball of radius $r$ in the half-space $\R^{N+1}_+$ is given by $\BB{D}^{N+1}_{r} = \set{Y \in \R^{N+1}_+ : \abs{Y} < r}$.
We identify $\BB{B}^N_r$ with $\BB{B}^{N+1}_r \cap \set{y_0 = 0}$.

In many of our results, we often assume certain boundary conditions on the thin set, $\R^{N+1} \cap \set{y_0 = 0}$.
We introduce language for those conditions here.

\begin{defn}[Elliptic boundary conditions]
    \label{eBC}
    For $V = V(Y)$ defined on $\BB{D}_R^{N+1}$, we say that $V$ satisfies the \textbf{elliptic boundary conditions} if there exists $\eps_0 > 0$ so that $\disp y_0^{a} \del_{y_0} V(y_0, y)$ is bounded whenever $(y_0, y) \in \mathbb{D}_R^{N+1} \cap \set{y_0 \in(0, \eps_0)}$ and for every $y \in \mathbb{B}_R^N$, either 
    \begin{enumerate}
        \item[(I)] $\del_\nu^a V(0,y) = 0$ as defined in \eqref{flatDerivV}, or 
        \item[(II)] $\disp \lim_{y_0 \to 0^+} V(y_0, y) = 0$ and $\disp \lim_{y_0 \to 0^+} \gr_y V(y_0, y) = 0$. 
    \end{enumerate}
    \nid If $(I)$ holds for every point $y \in \mathbb{B}_R^N$, then we say that $V$ has \textbf{elliptic boundary conditions of type I}.
    If $(II)$ holds for every point $y \in \mathbb{B}_R^N$, then we say that $V$ has \textbf{elliptic boundary conditions of type II}.
\end{defn}

\subsection{Parabolic operators}
\label{SS:backwardheat}
We consider fractional parabolic operators of the form $\pr{\del_t - \LP}^s$, where $s \in (0,1)$.
These operators act on space-time functions $u(x,t)$, where $x \in \R^d$ and $t \in \R$.
By analogy with \eqref{EFourier}, we see via the Fourier transform that for $(\xi, \rho) \in \R^d \times \R_+$,
\begin{equation}
\label{PFourier}
\mathcal{F}(\pr{\del_t - \LP}^s u)\pr{\xi, \rho} = \pr{\abs{\xi}^2 + i \rho}^s \mathcal{F}( u)\pr{\xi, \rho}.
\end{equation}
In keeping with the singular integral operator representation from \eqref{ESIO}, as shown in \cite{NS16},
\begin{equation}
\label{PSIO}
\pr{\del_t - \LP}^su\pr{x,t} = \int_{-\iny}^t \int_{\R^d}  \brac{u\pr{x, t} - u\pr{x', t'}} K_s\pr{x -x', t- t'}  dx' dt',
\end{equation}
where
\begin{equation*}
K_s\pr{x, t} = \frac{e^{-\abs{x}^2/\pr{4t}}}{\pr{4\pi t}^{d/2} } \frac{t^{-\pr{1 + s}}}{\abs{\Ga\pr{-s}}} 
\end{equation*}
and $\Ga$ denotes the Gamma function.
See also \cite{ST17}*{Theorem 1.1} for an equivalent pointwise representation.

As in the elliptic setting, we may also interpret the fractional heat operator using an extension problem.
To denote the half-space in $\R^{d+1}$, we write $\R^{d+1}_+ = \set{X = \pr{x_0, x} : x_0 > 0,  \ x \in \R^d}$.
Consider the parabolic problem
\begin{equation}
\label{PExtProb}
\left\{\begin{array}{ll}  
\LP_x U + \frac{a}{x_0} \del_{x_0} U + \del_{x_0}^2 U = \del_t U & \text{ for } \pr{X, t} \in \R^{d+1}_+ \times \R \\
\disp \lim_{x_0 \to 0^+} U\pr{x_0, x, t} = u\pr{x, t} & \text{ for } \, \pr{x, t} \in \R^d \times \R,
\end{array}  \right.
\end{equation}
where $\disp \del_{x_0} = \frac{\del }{\del x_0}$, $\disp \gr_x = \pr{\frac{\del }{\del x_1}, \ldots, \frac{\del }{\del x_d}}$, and $\LP_x = \gr_x \cdot \gr_x$.
The solution function $U\pr{X, t}$ may be defined through the Poisson formula by 
\begin{equation}
\label{PPoisson}
U\pr{{x_0}, x, t} = \int_0^\iny \int_{\R^d} P_a\pr{{x_0}, x', t'} u\pr{x - x', t - t'} dx' dt',
\end{equation}
where the parabolic Poisson kernel is $\disp P_a\pr{X, t} = \frac{\pi^{-\frac d 2}2^{-(d+1-a)}}{ \Ga\pr{\frac{1-a}{2}}  } \frac{x_0^{1-a}}{t^{\frac{d+3-a}{2}}} e^{- \frac{\abs{X}^2}{4t}}$. 
Then, in analogy with \eqref{EExt}, it can be shown that
\begin{equation}
\label{PExt}
\begin{aligned}
\frac{\abs{\Ga\pr{-s}}}{4^s \Ga\pr{s}} \pr{\del_t - \LP}^s u\pr{x, t} 
&= - \lim_{x_0 \to 0^+} \frac{U\pr{x_0, x, t} - U\pr{0, x, t}}{x_0^{1-a}} 
= - \frac{1}{1-a} \lim_{x_0 \to 0^+} x_0^{a} \del_{x_0} U\pr{x_0, x, t}.
\end{aligned}
\end{equation}
In other words, the derivative of $U$ with respect to $x_0$ gives us information about $\pr{\del_t - \LP}^s u$.

In our applications below, we use the backward fractional heat operator and its associated extension problem.
As in \eqref{PFourier}, we see via the Fourier transform that
\begin{equation*}
\mathcal{F}(\pr{-\del_t - \LP}^s u)\pr{\xi, \rho} = \pr{\abs{\xi}^2 - i \rho}^s \mathcal{F}(u)\pr{\xi, \rho}.
\end{equation*}
Applying a time reversal to \eqref{PSIO} with $\bar{u}(x,t) = u(x, - t)$ shows that $\pr{-\del_t - \LP}^s$ also has a singular integral representation,
\begin{align*}
\pr{-\del_t - \LP}^s u\pr{x,t} 
&= \pr{\del_t - \LP}^s \bar{u}\pr{x,-t} 
= \int_{-\iny}^{-t} \int_{\R^d}  \brac{\bar{u}\pr{x, -t} - \bar{u}\pr{x', t'}} K_s\pr{x -x', -t- t'}  dx' dt' \\
&= \int^{\iny}_t \int_{\R^d}  \brac{u\pr{x, t} - u\pr{x', \tau}} \overline{K_s}\pr{x -x', t - \tau}  dx' d\tau.
\end{align*}
Similarly, from \eqref{PPoisson} we get that
\begin{align*}
U\pr{{x_0}, x, t}
&= \overline{U}\pr{{x_0}, x, -t}
= \int_0^\iny \int_{\R^d} P_a\pr{{x_0}, x', t'} \bar{u}\pr{x - x', -t - t'} dx' dt' \\
&= \int_{-\iny}^0 \int_{\R^d} \overline{P_a}\pr{{x_0}, x', \tau} u\pr{x - x', t - \tau} dx' d\tau
\end{align*}
is a solution to the backward extension problem
\begin{equation}
\label{BPExtProb}
\left\{\begin{array}{ll}  
\LP_x U + \frac{a}{x_0} \del_{x_0} U + \del_{x_0}^2 U + \del_t U = 0 & \text{ for } \, \pr{X, t} \in \R^{d+1}_+ \times \R \\
\disp \lim_{x_0 \to 0^+} U\pr{x_0, x, t} = u\pr{x, t} & \text{ for } \, \pr{x, t} \in \R^d \times \R.
\end{array}  \right.
\end{equation}
By analogy with \eqref{PExt}, we get
\begin{equation}
\label{BPExt}
\frac{2s \abs{\Ga\pr{-s}}}{4^s \Ga\pr{s}} \pr{-\del_t - \LP}^s u\pr{x, t} 
= - \partial_{\nu}^a U(0, x, t),
\end{equation}
where we introduce the notation
\begin{equation}
\label{flatDerivU}
    \partial_{\nu}^a U(0, x, t) 
    := \lim\limits_{x_0 \rightarrow 0+} x_0^a\partial_{x_0} U(x_0, x, t).
\end{equation}
As with the elliptic setting, we often assume certain boundary conditions on the thin set, described with the following language.

\begin{defn}[Parabolic boundary conditions]
    \label{pBC}
    For $U = U(X,t)$ defined on $\R^{d+1}_+ \times \pr{0, T}$, we say that $U$ satisfies the \textbf{parabolic boundary conditions} if there exists $\eps_0 > 0$ so that $x_0^a \del_{x_0} U(x_0, x, t)$ and $x_0^{1+a} \del_{t} U(x_0, x, t)$ are bounded whenever $\pr{x_0, x, t} \in \pr{0, \eps_0} \times \R^d \times (0, T)$ and for every $(x, t) \in \R^d \times \pr{0, T}$ either
    \begin{itemize}
        \item[(I)] $\del_\nu^a U(0,x, t) = 0$ as defined in \eqref{flatDerivU}, or 
        \item[(II)] $\disp \lim_{x_0 \to 0^+} U(x_0, x, t) = 0$ and $\disp \lim_{x_0 \to 0^+} \gr_{(x,t)} U(x_0, x,t) = 0$.
    \end{itemize}
    \nid If $(I)$ holds for every point $(x, t) \in \R^d \times \pr{0, T}$, then we say that $U$ has \textbf{parabolic boundary conditions of type I}.
    If $(II)$ holds for every point $(x, t) \in \R^d \times \pr{0, T}$, then we say that $U$ has \textbf{parabolic boundary conditions of type II}.
\end{defn}

\section{Elliptic-to-Parabolic Transformations}
\label{S:elliptictopara}

In this section, we construct the transformations that connect space-time functions $U = U(x_0, x,t)$ defined on $\R^{d+1}_+ \times \pr{0, T}$ to a sequence of spatial functions $V_n = V_n(y_0, y)$ defined in $\R^{dn+1}_+$ for all $n \in \N$.
More specifically, fixing $d \in \N$ and $a \in (-1,1)$, for each $n \in \N$, we construct a mapping of the form
\begin{align*}
F_{n} :  \R^{dn+1}_+ &\to \R^{d+1}_+ \times \R_+ \\
\pr{y_0, y} &\mapsto (x_0, x,t)
\end{align*}
that takes elements $Y = \pr{y_0, y}$ in (high-dimensional) space $\R^{dn+1}_+$ to elements $(X,t)$ in space-time $\R^{d+1}_+ \times \R_+$.
Given a function $U = U(X,t)$ defined on a space-time domain, 
we use $F_{n}$ to define a function $V_n = V_n(Y)$ on the space $\R^{dn+1}_+$ by setting $V_n(Y) = U(F_{n}(Y))$.
As we show below via the chain rule, if $U$ is a solution to a degenerate backward parabolic equation, then each $V_n$ is a solution to some related nonhomogeneous degenerate elliptic equation.
This observation explains why we think of $U$ as a ``parabolic function" and of each $V_n$ as an ``elliptic function."
Moreover, as $n$ becomes large, the function $V_n$ behaves (heuristically) more and more like a solution to a homogeneous degenerate elliptic equation.
As such, the transformation $F_{n}$ becomes more useful to our purposes as $n \to \iny$, thereby illuminating why the notion of a high-dimensional limit is relevant here.

Using the extension technique, we are able to find relationships between functions $u = u(x, t)$ which solve fractional parabolic equations in $\R^d \times \R_+$, and associated functions $v_n = v_n(y)$ that solve nonhomogeneous fractional elliptic equations in $\R^{dn}$.

From another perspective, when we use $F_{n}$ to pushforward measures on half-spheres and half-balls in $\R^{dn+1}_+$, we produce measures in space-time that are weighted by approximations to generalized Gaussians.
This perspective is explored in Section \ref{S:PushfR}.

We can think of the transformations $F_{n} : \R^{dn+1}_+ \to \R^{d+1}_+ \times \R_+$ in a number of ways.
As mentioned above (and as we show in this section), these maps are constructed so that each $V_n$ solves an elliptic equation whenever $U$ solves a parabolic equation.
This viewpoint is purely computational as these relationships are illuminated by the chain rule.
This perspective is perhaps the most (easily) checkable, but it is somewhat mysterious.
Because random walks are the underlying mechanism used to define these mappings, we begin there.

We recall some classical ideas concerning random walks, going back to Wiener \cite{Wie23}.
An explanation of these ideas is also available in Sverak's notes \cite{Sve11}.
Consider $d+1$ particles, where $d$ of them move randomly in one spatial dimension, and $1$ of them moves randomly, but only in the positive direction.  
Let $x_1, x_2, \ldots, x_d \in \R$ and $x_0 > 0$ denote the coordinates of these particles.
Rather than imposing a condition on the step size, we instead require the following more universal condition: 
If $x_0$ makes one positive random step denoted by $y_0$, and each other $x_i$ makes $n$ random steps, denoted $y_{i,1}, y_{i,2}, \ldots, y_{i,n}$, then for some fixed $t > 0$, we have
\begin{equation}
\label{yNorm}
y_0^2 + \abs{y}^2 = y_0^2 + \sum_{i=1}^d \brac{y_{i,1}^2 + \ldots + y_{i,n}^2}= M t,
\end{equation}
where
\begin{equation}
\label{MDef}
M = M(d, n, a) 
:= \frac{2 \pr{d n + 1 + a}}{n} 
= 2d + \frac{2(1+a)}{n}.
\end{equation}
This choice for $M$ will be explained below.
Assuming that each $x_i$ starts at the origin, after these $n$ steps, the new positions will be
\begin{equation}
\label{xiDef}
x_i = y_{i,1} + y_{i,2} + \ldots + y_{i,n}.
\end{equation}
On the other hand, we make the rescaling
\begin{equation}
x_0 = \sqrt n \, y_0.
\label{x_0Def}
\end{equation}
We define the function $f_{n} : \R^{dn+1}_+ \to \R^{d+1}_+$ by
\begin{align}
f_{n}\pr{y_0, y_{1,1}, \ldots, y_{1,n}, \ldots, y_{d,1}, \ldots, y_{d,n}}
&= \pr{\sqrt n \, y_0, y_{1,1} + \ldots + y_{1,n}, \ldots, y_{d,1} + \ldots + y_{d,n}}.
\label{fnDef}
\end{align}
In other words, equation \eqref{x_0Def} holds and equation \eqref{xiDef} holds for each $i = 1, \ldots, d$.
Going further, define $F_{n} : \R^{dn+1}_+ \to \R^{d+1}_+ \times \R_+$ by
\begin{align}
\label{FnDef}
F_{n}\pr{y_0, y_{1,1}, \ldots, y_{1,n}, \ldots, y_{d,1}, \ldots, y_{d,n}}
&= \pr{f_{n}\pr{y_0, y_{1,1}, \ldots, y_{1,n}, \ldots, y_{d,1}, \ldots, y_{d,n}}, \frac{y_0^2 + \abs{y}^2}{M}}.
\end{align}
In other words, $F_{n}$ captures \eqref{yNorm}, as well as \eqref{x_0Def} and each \eqref{xiDef}.

We may sometimes abuse notation and apply $f_{n}$ and $F_{n}$ to elements in $\R^{dn}$ or $\R^{dn+1}$ instead of $\R^{dn+1}_+$.
That is, we interpret $f_{n} : \R^{dn} \to \R^d$ and $F_{n} : \R^{dn} \to \R^d \times \R_+$ to mean $f_{n}\pr{y} 
= \pr{y_{1,1} + \ldots + y_{1,n}, \ldots, y_{d,1} + \ldots + y_{d,n}}$ and $F_{n}\pr{y} 
= \pr{f_{n}(y), \frac{\abs{y}^2}{M}}$.

\begin{defn}[Half-balls and half-spheres]
In $\R^{N+1}$, the open half-ball of radius $r > 0$ is denoted by $\BB{D}^{N+1}_{r} = \set{Y \in \R^{N+1}_+ : \abs{Y} < r}$, 
$\widetilde{\BB{D}^{N+1}_{r}} = \overline{\BB{D}}^{N+1}_{r} \cap \set{y_0 > 0}$, where 
$\overline{\BB{D}}^{N+1}_{r}$ denotes the closed half-ball of radius $r$, and the half-sphere of radius $r$ is given by $\BB{H}^{N}_r = \set{Y \in \R^{N+1}_{+} : \abs{Y} = r}$.    
In all cases, the superscript denotes the dimension of the set, while the subscript denotes the radius.
If the underlying dimension is understood from the context, we may write $\BB{D}_r$, $\overline{\BB{D}}_r$, $\widetilde{\BB{D}}_r$, or $\mathbb{H}_r$.
\end{defn}

\begin{rem} 
\label{mappingObs}
Notice that by \eqref{xiDef}, \eqref{x_0Def}, standard inequalities, and \eqref{yNorm},
\begin{align*}
x_0^2 + \abs{x}^2
= x_0^2  + \sum_{i=1}^d x_i^2
= \pr{\sqrt n y_0}^2  + \sum_{i=1}^d \pr{y_{i,1} + \ldots + y_{i,n}}^2
\le n y_0^2 + n \sum_{i=1}^d \pr{y_{i,1}^2 + \ldots + y_{i,n}^2} 
= n \cdot M t.
\end{align*}
In fact, the above equality and that $y_0 > 0$ shows that $f_{n}\pr{\BB{H}^{dn}_{\sqrt{Mt}}} \subseteq \widetilde{\BB{D}}^{d+1}_{\sqrt{Mnt}}$.

Conversely, given $X \in \widetilde{\BB{D}}^{d+1}_{\sqrt{Mnt}}$,let $P(X) = f_n^{-1}(X) = \set{Y \in \R^{dn+1}_+ : f_n(Y) = X}$ and note that because $P(X)$ is determined by $d+1$ linear equations, then $P(X)$ is a $(dn - d)$-dimensional hyperplane.
As shown above, if $Y \in P(X)$, then $\abs{X}^2 \le n \abs{Y}^2$.
Since $\abs{X}^2 \le M nt$, then $P(X) \cap \BB{H}^{dn}_{\sqrt{Mt}} \ne \emptyset$.
In particular, we may choose $Y \in \BB{H}^{dn}_{\sqrt{Mt}}$ so that $f_n(Y) = X$ and it follows that $\widetilde{\BB{D}}^{d+1}_{\sqrt{Mnt}}  \subseteq f_{n}\pr{\BB{H}^{dn}_{\sqrt{Mt}}}$.

Therefore, 
\begin{equation}
\label{fnRem}
\begin{aligned}
 & f_{n}\pr{\BB{H}^{dn}_{\sqrt{Mt}}} = \widetilde{\BB{D}}^{d+1}_{\sqrt{Mnt}} \\
 & F_{n}\pr{\BB{H}^{dn}_{\sqrt{Mt}}} = \widetilde{\BB{D}}^{d+1}_{\sqrt{Mnt}} \times \set{t} .
\end{aligned}    
\end{equation}
Moreover, as $\disp \widetilde{\BB{D}}^{dn+1}_{\sqrt{M T}} = \bigcup_{0 < t \le T} \BB{H}^{dn}_{\sqrt{Mt}}$, then
\begin{equation}
\label{FnImage}    F_{n}\pr{\widetilde{\BB{D}}^{dn+1}_{\sqrt{M T}}} 
    = \bigcup_{0 < t \le T} F_{n}\pr{\BB{H}^{dn}_{\sqrt{Mt}}} 
    = \bigcup_{0 < t \le T} \widetilde{\BB{D}}^{d+1}_{\sqrt{Mnt}} \times \set{t}.
\end{equation}
\end{rem}

The next set of observations shows how the derivatives of space-time functions $U : \R^{d+1}_+ \times \pr{0, T} \to \R$ can be related to those of high-dimensional space functions $V_n: \R^{dn+1}_+ \to \R$ through composition with $F_n$.

\begin{lem}[Extension Chain Rule Lemma]
\label{ChainRuleLem}
Given $n \in \N$ and a function $U : \R^{d+1}_+ \times \pr{0, T} \to \R$, the function $V_n: \BB{D}^{dn+1}_{\sqrt{M T}} \to \R$ defined by $V_n(Y) = U(F_{n}(Y)) = U(X, t)$ satisfies the following:
\begin{align}
&\del_{y_0} V_n
= \sqrt n \, \del_{x_0} U
+ \frac{2 y_{0}}{M} \del_t U
= \sqrt n \, \del_{x_0} U
+ \frac{2 x_{0}}{M \sqrt n} \del_t U
\label{Vy0Deriv} \\
&\del_{y_{i,j}} V_n
= \del_{x_i} U
+ \frac{2 y_{i,j}}{M} \del_t U
\label{VyijDeriv}  \\
&Y \cdot \gr V_n
= (X, 2t) \cdot \gr_{(X,t)} U 
\label{YdotGrExp} \\
&\abs{\gr V_n}^2
= n \abs{\gr U}^2
+ \frac {4t} {M} \abs{\del_t U}^2
+ \frac 4 M \pr{X \cdot \gr U} \del_t U.
\label{gradExp}
\end{align}
Moreover, if $U$ is a solution to $\LP_x U + \frac{a}{x_0} \del_{x_0} U + \del^2_{x_0} U + \del_t U= 0$ in $\R^{d+1}_+ \times \pr{0, T}$, then $V_n$ solves
\begin{align}
\label{VnKnEllipEqn}
 \LP_y V_n 
+ \frac{a}{y_0}\del_{y_0} V_n 
+ \del^2_{y_0} V_n
&= K_n,
\end{align}
in $\BB{D}^{dn+1}_{\sqrt{M T}}$, where $K_n: \BB{D}^{dn+1}_{\sqrt{M T}} \to \R$ is defined by 
\begin{equation}
\label{KnDef}
K_n\pr{Y} =  \frac 4 M J\pr{F_{n}\pr{Y}}
\end{equation}
with
\begin{equation}
\label{JDefn}
J\pr{X, t} 
= \pr{X, t} \cdot \gr_{(X,t)} \del_t U.
\end{equation}
\end{lem}

\begin{rem}
The relationship between elliptic and parabolic solutions that is described by \eqref{VnKnEllipEqn} holds in a far more general setting, see Lemma \ref{weaktosuperweak}.
\end{rem}

\begin{proof}
Applications of the chain rule lead to the first four expressions.
Differentiating \eqref{Vy0Deriv} and \eqref{VyijDeriv} again, we get
\begin{align*}
&\del^2_{y_0} V_n
= n \, \del^2_{x_0} U
+ \frac{2}{M} \del_t U 
+ \frac{4 x_0}{M}\del_t \del_{x_0} U
+ \frac{4 x_{0}^2}{M^2 n} \del_t^2 U, \\
&\del^2_{y_{i,j}} V_n
= \del^2_{x_i} U
+ \frac{2}{M}\del_t U 
+ \frac{4 y_{i,j}}{M} \del_t \del_{x_i} U
+ \frac{4 y_{i,j}^2}{M^2} \del^2_{t} U ,
\end{align*}
and then
\begin{align*}
\del^2_{y_0} V_n
+ \frac{a}{y_0} \del_{y_0} V_n
+ \sum_{i=1}^d\sum_{j=1}^n \del^2_{y_{i,j}} V_n
&= n \brac{\LP_x U
+ \frac{a}{x_0} \del_{x_0} U
+ \del^2_{x_0} U
+ \frac{2\pr{ dn + 1 + a}}{Mn} \del_t U } \\
&+ \frac 4 M \brac{x_0 \del_t \del_{x_0} U + \sum_{i=1}^d x_{i} \del_t \del_{x_i} U}
+ \frac 4 {M^2}\brac{y_{0}^2 + \sum_{i=1}^d \sum_{j=1}^n y_{i,j}^2} \del_t^2 U \\
&= n \pr{\LP_x U 
 + \frac{a}{x_0} \del_{x_0} U
 + \del_{x_0}^2 U
 + \del_t U} 
 +\frac 4 M \pr{X, t} \cdot \gr_{(X,t)} \del_t U,
\end{align*}
where we used \eqref{yNorm} and \eqref{MDef} to simplify the expression.
The final claim is a consequence of the previous equation.
\end{proof}

In addition to the derivative relationships described above, the mapping $F_n$ also preserves the boundary conditions that are described in Definitions \ref{eBC} and \ref{pBC}.

\begin{lem}[Inheritance of Boundary Conditions]
\label{BCLemma}
    If $U : \R_{+}^{d+1} \times \pr{0, T} \to \R$ satisfies the parabolic boundary conditions on $\R^{d+1}_+ \times \pr{0, T}$, then each $V_n : \BB{D}^{dn+1}_{\sqrt{MT}} \to \R$ defined by $V_n \pr{Y} = U\pr{F_{n}\pr{Y}} = U\pr{X, t}$ satisfies the elliptic boundary conditions on $\BB{D}^{dn+1}_{\sqrt{MT}}$ and the type is preserved.
\end{lem}

\begin{proof}
From \eqref{Vy0Deriv} with \eqref{MDef} and \eqref{x_0Def}, we see that
\begin{equation*}
    \begin{aligned}
    y_0^{a}\del_{y_0} V_n\pr{y_0, y}
    &= \pr{\frac {x_0} {\sqrt n}}^{a} \brac{\sqrt n \del_{x_0} U\pr{x_0, x,t} + \frac{2 x_{0}}{M \sqrt n} \del_t U\pr{x_0, x,t} } 
    = n^{\frac{1-a}{2}} \brac{x_0^{a} \del_{x_0} U
    + \frac { x_{0}^{1 + a}}{\pr{dn + 1+a}} \del_t U}.    
    \end{aligned}
\end{equation*}
Therefore, if $x_0^a \del_{x_0} U(x_0, x, t)$ and $x_0^{1+a} \del_{t} U(x_0, x, t)$ are bounded whenever $\pr{x_0, x, t} \in \pr{0, \eps_0} \times \R^d \times (0, T)$, 
then $y_0^a \del_{y_0} V_n(y_0, y)$ is bounded whenever $\pr{y_0, y} \in  \mathbb{D}_{\sqrt{MT}}\cap \set{x_0 \in \pr{0, \frac{\eps_0}{\sqrt n }}}$.
Moreover, $\del^a_\nu U(0, x, t) = 0$ implies that $\del^a_\nu V_n(0, y) = 0$, showing that if $U$ has parabolic boundary conditions of type I, then each $V_n$ has elliptic boundary conditions of type I.
An analogous statement for boundary conditions of type II follows from the definition of $V_n$ along with \eqref{Vy0Deriv} and \eqref{VyijDeriv}.
\end{proof}

Combining the previous two results leads to statements about the fractional operators.

\begin{cor}[Fractional Chain Rule Result]
Given $n \in \N$ and a function $u : \R^d \times \pr{0, T} \to \R$, there exists  
$V_n : \BB{D}^{dn+1}_{\sqrt{M T}} \to \R$ 
that solves the nonhomogeneous extension problem \eqref{VnExtProp} below and satisfies the following:
\begin{itemize}
    \item[(I)] If $\pr{-\del_t - \LP}^s u(x,t) = 0$, then $\disp \del_\nu^a V_n(0,y) = 0$.
    \item[(II)] If $u(x,t) = 0$ and $\gr_{(x,t)} u(x,t) = 0$, then $\disp \lim_{y_0 \to 0^+} V_n(y_0, y) = 0$ and $\disp \lim_{y_0 \to 0^+} \gr_y V_n(y_0, y) = 0$.
\end{itemize}
\end{cor}

\begin{proof}
Given $u: \R^d \times \pr{0, T} \to \R$, let $U : \R^{d+1}_+ \times \pr{0, T} \to \R$ be the solution to \eqref{BPExtProb}. 
Define $V_n : \BB{D}^{dn+1}_{\sqrt{M T}} \to \R$ so that $V_n \pr{Y} = U\pr{F_{n}\pr{Y}} = U\pr{X, t}$, then set $\disp v_n\pr{y} = \lim_{y_0 \to 0^+} V_n\pr{y_0, y}$.
As $U$ is a solution to \eqref{BPExtProb}, then it follows from \eqref{VnKnEllipEqn} in Lemma \ref{ChainRuleLem} that each $V_n$ satisfies the following nonhomogeneous elliptic extension problem:
\begin{equation}
\left\{\begin{array}{ll}  
\LP_y V_n + \frac{a}{y_0} \del_{y_0}{V_n} + \del^2_{y_0}{V_n} = K_n & Y \in \BB{D}^{dn+1}_{\sqrt{M T}} \\ \\
\disp \lim_{y_0 \to 0^+} V_n\pr{y_0, y} = v_n\pr{y} & y \in \BB{B}^{dn}_{\sqrt{M T}},
\end{array}  \right.
\label{VnExtProp}
\end{equation}
where $K_n$ is defined in \eqref{KnDef}.

If $\pr{-\del_t - \LP}^s u(x,t) = 0$, then \eqref{BPExt} shows that $\del_\nu^a U(0,x, t) = 0$, so the conclusion of $(I)$ follows from Lemma \ref{BCLemma}.

If $u(x,t) = 0$ and $\gr_{(x,t)} u(x,t) = 0$, then by definition $\disp \lim_{x_0 \to 0^+} U(x_0, x, t) = 0$ and $\disp \lim_{x_0 \to 0^+} \gr_{(x,t)} U(x_0, x,t) = 0$, so the conclusion of $(II)$ follows from Lemma \ref{BCLemma} as well.
\end{proof}

In summary, we have constructed sequences of maps $f_{n}$ and $F_{n}$, given in \eqref{fnDef} and \eqref{FnDef}, respectively, which serve as the connection between the elliptic and parabolic settings.
The following table describes these relationships and the notation that we use to describe our elliptic and parabolic settings.
Note that $F_{n}$ is the connection between the elliptic column and the parabolic column.

$$
\begin{array}{l|ll}
    & \textrm{Elliptic} & \textrm{Parabolic} \\
    \hline \hline 
    \textrm{Space} & \textrm{high-dimensional space} & \textrm{space-time} \\
    \textrm{Elements} & y \in \R^{dn} & (x,t)\in\R^{d} \times \R_+ \\
     \textrm{Functions} & v_n = v_n(y) & u = u\pr{x,t} \\
     \textrm{Extended elements} & Y=(y_0,y) \in \R^{d n+1}_+ & \pr{X, t} =\pr{x_0,x, t} \in \R^{d+1}_+ \times \R_+ \\
    \textrm{Extended functions} & V_n=V_n(y_0,y) & U=U(x_0,x,t)\\
    \textrm{Nonlocal operators} & (- \LP)^s &  (- \del_t - \LP)^s \\
    \textrm{Degenerate PDEs} & \LP_y V_n + \frac{a}{y_0} \del_{y_0}{V_n} + \del^2_{y_0}{V_n} = K_n & \LP_x U + \frac{a}{x_0} \del_{x_0} U + \del_{x_0}^2 U + \del_t U = 0\\
    \textrm{Type I boundary conditions} & \del_\nu^a V_n(0, y) = 0 & \del_\nu^a U(0, x, t) = 0 \\
    \textrm{Type II boundary conditions} & V_n(0, y) = 0, \gr_y V_n(0, y) = 0 & U(0, x, t) = 0, \gr_{(x,t)} U(0, x, t) = 0 \\
     \textrm{Sets} 
& \textrm{Half-spheres: } \BB{H}^{dn}_{\sqrt{Mt}} & \textrm{Half-balls in a time slice: } \widetilde{\BB{D}}^{d+1}_{\sqrt{Mnt}} \times \set{t} \\
& \textrm{Half-balls: } \widetilde{\BB{D}}^{dn+1}_{\sqrt{M T}}
    & \textrm{Space-time half-cones: }  \disp \bigcup_{0 \le t \le T} \widetilde{\BB{D}}^{d+1}_{\sqrt{Mnt}} \times \set{t}
\end{array}
$$

\section{Pushforward Results}
\label{S:PushfR}

Here we examine how integrals of functions $\phi(X,t)$ defined in space-time $\R^{d+1}_+ \times \R_+$ can be related to integrals of $\phi(F_{n}(Y))$ in the high-dimensional space $\R^{dn+1}_+$.
We start from a probabilistic viewpoint.
That is, if $Y \in \R^{dn+1}_+$ is distributed over a fixed half-sphere, we seek to determine the probability distribution of $X \in \R^{d+1}_+$, where $(X,t) = F_{n}(Y)$.
As we show below, understanding this probability distribution on the parabolic side reduces to a pushforward computation which we carry out explicitly.
We then collect the consequences that will be used in our elliptic-to-parabolic proofs.

Recall that $a \in (-1,1)$.
In contrast to our previous work, \cites{Dav18, DSVG24}, we no longer assume that the elements $Y$ are uniformly distributed over the (half)-sphere.
As in \cite{CS07}, for example, we adapt the viewpoint that we are working in $dn + 1 + a$ dimensions, even though $1+a$ may not be an integer.
More specifically, $y_0$ and $x_0$ are $1$-dimensional objects that ``act like" $\pr{1+a}$-dimensional objects in the measures that we put on our spheres.

Let $\si^{N-1}_r$ denote the canonical surface measure of the sphere $\BB{S}^{N-1}_r$. 
Note that $\abs{\BB{S}^{N-1}_r} = \aSp{N-1}r^{N-1}$, where $\Sp{N-1} = \BB{S}^{N-1}_1$ denotes the unit sphere in $\R^{N}$ and
\begin{equation}
\label{sphereSize}
\aSp{N-1} = \frac{2 \pi^{\frac N 2}}{\Ga\pr{\frac N 2}}.
\end{equation}
For the half-sphere $\BB{H}^{N}_{r}$ of radius $r$ in $\R^{N+1}_+$,
 let $\eta^N_r$ denote its canonical surface measure and note that $\abs{\BB{H}^{N}_{r}} = \frac 1 2 \aSp{N} r^N$.
Define the function $\rho_r(y_0) = \sqrt{r^2 - y_0^2}$ and observe that
\begin{align*}
\eta^N_r(y_0, y)
&= \si^{N-1}_{\rho_r(y_0)}(y) \sqrt{1 + \brac{\rho_r'(y_0)}^2} \chi_{\brac{0, r}}(y_0) 
= \si^{N-1}_{\rho_r(y_0)}(y) \frac{r}{\rho_r(y_0)} \chi_{\brac{0, r}}(y_0).
\end{align*}
By integrating in $y_0$, we see that 
\begin{align*}
\abs{\BB{H}^{N}_{r}}
&= \int_0^r \sigma^{N-1}_{\rho_r(y_0)} \frac{r}{\rho_r(y_0)}  dy_0
= r \aSp{N-1} \int_0^r \rho_r(y_0)^{N-2} dy_0
= \frac 1 2 \frac{2 \pi^{\frac{N+1} 2}}{\Ga\pr{\frac{N+1}{2}}} r^N
= \frac{1}{2} \aSp{N} r^N,
\end{align*}
as expected.
It follows that the normalized measure $\disp 
\frac{2}{\aSp{N} r^N} \eta^N_r$ defines the canonical probability measure on $\BB{H}^{N}_{r}$.

For $a \in (-1,1)$, we define a modified, non-uniform measure on $\BB{H}^{N}_{r}$ in which $y_0$ acts like a $(1+a)$-dimensional object by introducing a weight and setting
\begin{equation}
\label{etaNsDef}
\begin{aligned}
\eta^{N, a}_r(y_0, y)
&= \si^{N-1}_{\rho_r(y_0)}(y) \sqrt{1 + \brac{\rho_r'(y_0)}^2} y_0^{a} \, \chi_{\brac{0, r}}(y_0) 
= \si^{N-1}_{\rho_r(y_0)}(y) \frac{r}{\rho_r(y_0)} y_0^{a} \, \chi_{\brac{0, r}}(y_0).
\end{aligned}    
\end{equation}
This measure is a reweighting of the canonical surface measure on $\BB{H}^{N}_{r}$. 
In particular, when $a = 0$, this measure reduces to the canonical surface measure on $\BB{H}^{N}_{r}$, i.e. $\eta_r^{N, 0} = \eta_r^N$.
Integrating $\eta^{N, a}_r$ in $y_0$ produces a weighted measure for $\BB{H}^{N}_{r}$; that is, 
\begin{align*}
\abs{\BB{H}^{N}_{r}}_a
&:= \int_0^r \sigma_{\rho_r(y_0)}^{N-1} \frac{r}{\rho_r(y_0)} y_0^{a} \, dy_0
= \frac 1 {2\ga_a} \aSp{N+a} r^{N+a},
\end{align*}
where we have introduced the constant
\begin{equation}
\label{gasDefn}
  \ga_a = \frac{\pi^{\frac {1+a} 2}}{\Ga\pr{\frac{1+a}{2}}}.  
\end{equation}

Fix $d \in \N$ and $a \in (-1,1)$.
We use the measures $\eta^{dn, a}_r$ to understand the probability law for the random walks described by the points $Y = (y_0, y_{1,1}, \ldots, y_{1,n}, \ldots, y_{d,1}, \ldots, y_{d,n})$.
Assume that the vectors $Y$ are distributed over the $dn$-dimensional half-sphere of radius $\sqrt{M t}$, where for each fixed $y_0$, the measure is uniform in the $y$-variables, but the measure is weighted in $y_0$ to account for its $(1+a)$-dimensional type behavior.
That is, when the surface measure is normalized to have total measure equal to $1$, then this surface measure, $\mu^{n}_t$, is given by
\begin{equation}
\label{muDefn}
\mu^{n}_t
:= \frac{2\ga_{a} }{\aSp{dn+a} \pr{Mt}^{\frac{dn +a}{2}} } \eta^{dn, a}_{\sqrt{Mt}},
\end{equation}
where $\eta^{N, a}_{r}$, as defined in \eqref{etaNsDef}, denotes the weighted surface measure on the half-sphere in $\R^{N+1}_+$ of radius $r$.
If $a = 0$, then we are reduced to the normalized canonical surface measure for a half-sphere in $\R^{dn + 1}_+$.

Recall the mapping $f_{n}: \R^{dn+1}_+ \to \R^{d+1}_+$ from \eqref{fnDef}.
We compute the push-forward of $\mu^{n}_{t}$ by $f_{n}$ to get a measure on $\R^{d+1}_+ \times \set{t}$ denoted by $\nu^{n}_{t} = f_{n} \# \mu^{n}_{t}$.  
To start, we introduce a related map $\tilde f_n : \R^{dn+1}_+ \to \R^{d+1}_+$ that is given by
\begin{align*}
\tilde f_{n}\pr{y_0, y_{1,1}, y_{1,2}, \ldots, y_{1,n}, \ldots, y_{d,1}, y_{d,2}, \ldots, y_{d,n}}
&= \pr{\sqrt n \, y_0, \sqrt{n}\, y_{1,1}, \sqrt{n} \, y_{2,1}, \ldots, \sqrt{n} \, y_{d,1}}.
\end{align*}
Let $B$ be an $n \times n$ orthogonal matrix with the property that $b_{j, 1} = b_{n, \ell} = \frac 1 {\sqrt n}$ for all $j, \ell \in \set{1, \ldots, n}$.
Orthogonality shows that $\disp \sum_{j=1}^n b_{j, \ell} = \begin{cases} \sqrt n & \ell = 1 \\ 0 & \text{otherwise}\end{cases}$.
Writing $y^i = \pr{y_{i, 1}, y_{i, 2}, \ldots, y_{i,n}} \in \R^n$ for each $i \in \set{1, \ldots, d}$, observe that
\begin{align*}
\sum_{j = 1}^n (B y^i)_j
=\sum_{j = 1}^n \sum_{\ell = 1}^n b_{j, \ell} y_{i,\ell}
&= \sum_{\ell = 1}^n y_{i,\ell} \sum_{j = 1}^n b_{j, \ell} 
= \sum_{\ell = 1}^n y_{i,\ell} \sqrt n \de_{\ell, 1}
= \sqrt n y_{i,1} 
\end{align*}
from which it follows that
\begin{align*}
f_n\pr{y_0, B y^1, B y^2, \ldots, B y^d}
&= \pr{\sqrt n y_0, \sum_{j = 1}^n (B y^1)_j, \sum_{j = 1}^n (B y^2)_j, \ldots, \sum_{j = 1}^n (B y^d)_j}
= \tilde f_n(y_0, y^1, y^2, \ldots, y^d).
\end{align*}
In particular,

since the two maps $f_n$ and $\tilde f_n$ are related by an orthogonal transformation that leaves the measure unchanged, we choose to work with the simpler map, $\tilde f_n$, in our pushforward computation. 
Therefore, we write $x_i = \sqrt{n} \, y_{i,1}$ and $x_0 = \sqrt n \, y_0$ in what follows.  
The push-forward is computed in two steps.  
First, we push-forward the measure $\mu^{n}_{t}$ by the projection defined by 
$$\pi^0_d\pr{y_0, y_{1,1}, y_{1,2}, \ldots, y_{1,n}, \ldots, y_{d,1}, y_{d,2}, \ldots, y_{d,n}} =  \pr{y_0, y_{1,1}, y_{2,1}, \ldots, y_{d,1}},$$ 
then we dilate by a factor of $\sqrt{n}$.

As an intermediate step, consider the pushforward of the measure $\si^{dn-1}_r$ under the projection $\pi_d : \R^{dn} \to \R^d$ defined by 
$$\pi_d\pr{y_{1,1}, y_{1,2}, \ldots, y_{1,n}, \ldots, y_{d,1}, y_{d,2}, \ldots, y_{d,n}} = y^1 := \pr{ y_{1,1}, y_{2,1}, \ldots, y_{d,1}}.$$
Generalizing the notation from above, set $\rho_r(y^1) = \sqrt{r^2 - \abs{y^1}^2}$ so that 
$\disp \sqrt{1 + \abs{\gr \rho_r(y^1)}^2} = \frac{r}{\rho_r(y^1)}$.
The pushforward with this projection produces a measure with support in a ball $\BB{B}_r^d$ which is defined as
$$\pr{\pi_{d} \# \si^{dn-1}_r}(y^1) 
= \si^{dn - d-1}_{\rho_r(y^1)} \frac{r}{\rho_r(y^1)} \chi_{\BB{B}_r}(y^1) 
= r \aSp{dn-d-1} \pr{r^2 - \abs{y^1}^2}^{\frac{dn-d-2} 2} \chi_{\BB{B}_r}(y^1).$$
Recalling \eqref{muDefn} and \eqref{etaNsDef}, it follows that
\begin{align*}
\pi_d^0\# \mu^{n}_{t}(y_0,y)
&= \pi_d^0\# \pr{\frac{2\ga_{a} }{\aSp{dn+a} \pr{Mt}^{\frac{dn +a}{2}} }\eta^{dn, a}_{\sqrt{Mt}}} 
= \pi_d^0\# \pr{\frac{2\ga_{a} \sqrt{Mt} }{\aSp{dn+a} \pr{Mt}^{\frac{dn+a}{2}} }\frac{\si^{dn-1}_{\rho_{\sqrt{Mt}}(y_0)}(y)}{\rho_{\sqrt{Mt}}(y_0)} y_0^{a} \, \chi_{\brac{0, \sqrt{Mt}}}(y_0)} \\
&= \frac{2\ga_{a} \aSp{dn-d-1}}{\aSp{dn+a} \pr{Mt}^{\frac{d +1+ a}{2}} } \pr{1 - \frac{y_0^2 + \abs{y^1}^2}{Mt}}^{\frac{dn-d-2} 2} y_0^{a} \, \chi_{\brac{0, \sqrt{Mt}}}(y_0)  \chi_{\BB{B}^d_{\sqrt{Mt - y_0^2}}}(y^1),
\end{align*}
which is also a probability measure.

Using that $x_0 = \sqrt n y_0$ and $x_i = \sqrt{n} \; y_{i,1}$, the resulting measure $\nu^{n}_{t} = f_{n} \# \mu^{n}_{t}$ is given by
\begin{align}
d \nu^{n}_{t} (x_0, x)
&=\frac{2\ga_{a} \aSp{dn-d-1}}{\aSp{dn+a} \pr{Mnt}^{\frac{d + 1+a}{2}}} \pr{1 - \frac{x_0^2 + \abs{x}^2}{M n t}}^{\frac{dn-d-2} 2} x_0^{a} \, \chi_{ \BB{D}^{d+1}_{\sqrt{Mnt}}}(x_0, x) \, dx_0 \, dx,
\label{nundDefn}
\end{align}
where $d{x} = d{x}_1 d{x}_2 \ldots d{x}_d$.

The following lemma shows that the function appearing in the measure $\nu^{n}_{t}$ described by \eqref{nundDefn} is a weighted approximation to the fundamental solution for the degenerate parabolic equation.

\begin{lem}[Gaussian approximation]
\label{GaussianLimit}
For some $d \in \N$ and $a \in (-1,1)$, define $\mathcal{G}_n : \R^{d+1}_+ \times \R_+ \to \R$ and $\mathcal{G} : \R^{d+1}_+ \times \R_+\to \R$ by
\begin{equation}
\label{GsdnDefn} 
\begin{aligned}
  & \mathcal{G}_n\pr{X, t} 
  := \mathcal{C}_n t^{-\frac{d + 1+a}{2}} \pr{1 - \frac{\abs{X}^2}{M n t}}^{\frac{dn-d-2} 2}  \chi_{\BB{D}^{d+1}_{\sqrt{Mnt}}}(X)  \\
  & \mathcal{G}\pr{X, t} 
  := \mathcal{C} t^{-\frac{d + 1+a}{2}} \exp\pr{ - \frac{\abs{X}^2}{4 t}},
\end{aligned}
\end{equation}
where 
\begin{equation*}
\begin{aligned}
\mathcal{C}_n 
&= \frac{2 \ga_a \aSp{dn-d-1}}{\aSp{dn+a} \pr{Mn}^{\frac{d + 1+a}{2}}}, \quad
\mathcal{C}
= \frac{\pr{4 \pi}^{- \frac d 2}}{2^{a} \Ga\pr{\frac{1+a}{2}}}.
\end{aligned}
\end{equation*}
Given any $t_0 > 0$, the sequence $\set{\mathcal{G}_n}_{n=1}^\iny$ converges uniformly to $\mathcal{G}$ on $\R^{d+1}_+ \times \set{t \ge t_0}$.
Moreover, there exists a constant $C_0 = C_0\pr{d,a} > 0$ so that for every $(X, t) \in \R^{d+1}_+ \times \R_+$, we have
\begin{align}
\label{pointwiseGBound}
    \mathcal{G}_n(X,t)
    &\le C_0 \mathcal{G}(X,t).
\end{align}
\end{lem}

\begin{rem}
The function $\mathcal{G}\pr{X, t}$ is the fundamental solution to the equation in \eqref{PExtProb} subject to the boundary conditions of type I. 
In particular, a computation shows that $\del_\nu^a \mathcal{G}(0, x, t) = 0$ for any $(x,t) \in \R^d \times \R_+$ and $\mathcal{G}(X,t)$ satisfies $\disp \del_t \mathcal{G}(X, t) = x_0^{-a} \di \pr{x_0^a \gr \mathcal{G}(X,t)}$ away from the singularity.
Moreover, for any $t > 0$,
\begin{align}
\label{GnInt}
    \int_{\R^{d+1}_+} \mathcal{G}(X,t) x_0^a dX
    &= \int_{\R^{d+1}_+} \frac{\pr{4 \pi}^{- \frac d 2}}{2^a \Ga\pr{\frac{1+a}{2}}} t^{-\frac{d + 1+a}{2}} \exp\pr{ - \frac{\abs{X}^2}{4 t}} x_0^a dX 
    = 1.
\end{align}
\end{rem}

\begin{proof}
We first shows that $\disp \lim_{n \to \iny} \mathcal{C}_n = \mathcal{C}$.
Using \eqref{MDef}, \eqref{sphereSize}, and \eqref{gasDefn}, we simplify to get that
$$\mathcal{C}_n 
= \frac{\pr{4\pi}^{-\frac {d} 2}}{2^a \Ga\pr{\frac{1+a}{2}}} \frac{\Ga\pr{\frac {dn+1+a} 2}}{\Ga\pr{\frac {dn-d} 2}\pr{\frac{d n + 1 + a}2}^{\frac{d + 1+a}{2}}}
= \mathcal{C} \frac{\Ga\pr{\frac{dn-d} 2 + \frac{d+1+a} 2}}{\Ga\pr{\frac {dn-d} 2} \pr{\frac{d n -d}2}^{\frac{d + 1+a}{2}}} \pr{\frac{d n -d}{d n + 1 + a}}^{\frac{d + 1+a}{2}}.$$
Since $\disp \lim_{m \to \iny} \frac{\Ga\pr{m+\al}}{\Ga\pr{m} m^\al} = 1$ and $\disp \lim_{n \to \iny} \pr{\frac{1 - \frac 1 n}{1 + \frac {1+a} {d n}}}^{\frac{d+1+a} 2} = 1$, then $\disp \lim_{n \to \iny} \frac{\mathcal{C}_n}{\mathcal{C}} = 1$.
Following the arguments in \cite{DSVG24}*{Lemma 3.1}, we deduce that 
\begin{align*}
\lim_{n \to \iny} t^{-\frac{d + 1 + a}{2}} \pr{1 - \frac{\abs{X}^2}{4 t\pr{\frac{d n + 1 + a}2} }}^{\frac{dn-d-2} 2}  \chi_{\BB{D}^{d+1}_{\sqrt{Mnt}}}(X) 
= t^{-\frac{d + 1 + a}{2}} \exp\pr{- \frac{\abs{X}^2}{4t}}
\end{align*}
uniformly on $\R^{d+1}_+ \times \set{t \ge t_0}$.
The bound in \eqref{pointwiseGBound} also follows from the arguments used to reach \cite{DSVG24}*{Lemma 3.1}, so the conclusion follows. 
\end{proof}

We used the maps $f_{n}$ from \eqref{fnDef} to compute push-forwards and generate the measures $\nu^{n}_{t}$. 
The following statement shows each $f_{n}$ may be used to connect functions in a high-dimensional (elliptic) setting to those in fixed time-slices of the space-time (parabolic) setting.

\begin{lem}[Fixed $t$ bridge]
\label{PFTS}
For $d\in \N$ and $a \in (-1,1)$, let $\mathcal{G}_n(X,t)$ and $\mathcal{G}(X,t)$ be as defined in \eqref{GsdnDefn}.
If for some $t > 0$, $\vp: \R^{d+1}_+ \to \R$ is integrable with respect to $x_0^{a} \, \mathcal{G}(X,t) d{X}$, then for every $n \in \N$, $\vp$ is integrable with respect to $x_0^{a} \, \mathcal{G}_n(X,t) d{X}$ and it holds that
\begin{align*}
\int_{\BB{H}^{dn}_{\sqrt{Mt}}} \vp\pr{f_{n}(Y)} y_0^{a} \, d\eta(Y)
&= \bar{C} t^{\frac{dn +a}{2}} \int_{\R^{d+1}_+} \vp\pr{X}  \dGnt,
\end{align*}
where the constant is given by
\begin{equation}
\label{CbarDefn}
   \bar{C} =  \frac{\aSp{dn+a} M^{\frac{dn+a}{2}} }{2\ga_a},
\end{equation}
$\ga_a$ is defined in \eqref{gasDefn},  
and we write $d\eta$ to denote the canonical surface measure on the half-sphere $\BB{H}^{dn}_{\sqrt{Mt}}$. 
\end{lem}

\begin{proof}
That $\vp$ is integrable with respect to each $x_0^{a} \, \mathcal{G}_n(X,t) d{X}$ whenever $\vp$ is integrable with respect to $x_0^{a} \, \mathcal{G}(X,t) d{X}$ follows from the bound \eqref{pointwiseGBound}.

Since $\vp$ is integrable with respect to each $x_0^{a} \, \mathcal{G}_n(X,t) d{X}$, by comparing \eqref{GsdnDefn} with \eqref{nundDefn}, we see that
\begin{align*}
    \int_{\R^{d+1}_+} \vp\pr{X} \dGnt
    &= \int_{\BB{D}^{d+1}_{\sqrt{Mnt}}} \vp\pr{X} d \nu_{t}^n(X)
    = \int_{f_{n}\pr{\BB{H}^{dn}_{\sqrt{Mt}}}} \vp\pr{X} d \pr{f_{n} \# \mu^n_t}
    = \int_{\BB{H}^{dn}_{\sqrt{Mt}}} \vp\pr{f_{n}(Y)} d \mu^n_t(Y),
\end{align*}
where we have used that $\nu_{t}^n = f_{n} \# \mu^n_t$, the change-of-variables formula for pushforward measures, and \eqref{fnRem} in Remark \ref{mappingObs}.
Plugging in the definition from \eqref{muDefn} shows that
\begin{align*}
    \frac{2\ga_a}{\aSp{dn+a} \pr{Mt}^{\frac{dn +a}{2}} }  \int_{\BB{H}^{dn}_{\sqrt{Mt}}} \vp\pr{f_{n}(Y)} d\eta^{dn, a}_{\sqrt{Mt}}(Y)
    &= \int_{\R^{d+1}_+} \vp\pr{X}  \dGnt.
\end{align*}
From \eqref{etaNsDef}, we see that
\begin{align*}
    d\eta^{dn, a}_{\sqrt{Mt}}(Y)
    &= d\si^{dn-1}_{\rho_{\sqrt{Mt}}(y_0)}(y) \frac{{\sqrt{Mt}}}{\rho_{\sqrt{Mt}}(y_0)} y_0^{a} \, \chi_{\brac{0, {\sqrt{Mt}}}}(y_0) dy_0
    = y_0^{a} \, d\eta^{dn}_{\sqrt{Mt}}(Y),
\end{align*}
where $\eta^{N}_{r}$ is the canonical surface measure on the $N$-dimensional half-sphere of radius $r$.
The conclusion follows.
\end{proof}

This lemma relates spheres in the high-dimensional (elliptic) setting to $t$-slices in the (parabolic) setting.
To broaden this viewpoint, as in \cite{Dav18}, \cite{DSVG24}, and \cite{Sve11}, we consider $t$ to be a parameter instead of a fixed constant.  
We think of the measures $\nu_t^{n}$ as time slices of a space-time object that comes from projections of some global measure $\mu^n$ in the full space $\R^{dn+1}_+$ (not just the half-spheres) onto the space-time $\R^{d+1}_+ \times \R_+$ (not just the time slices).
To do this, recall the function $F_{n} : \R^{dn+1}_+ \to \R^{d+1}_+ \times \R_+$ defined in \eqref{FnDef} and that $F_{n}\pr{\BB{H}^{dn}_{\sqrt{Mt}}} = \widetilde{\BB{D}}^{d+1}_{\sqrt{Mnt}} \times \set{t}$, see \eqref{fnRem}.

The global measure $\mu^n$ on $\R^{dn+1}_+$ has the property that 
\begin{align}
\label{FdnIntRel}
F_{n} \# \mu^n
&= \int_0^\iny f_{n} \# \mu^n_t d{t} 
= \int_0^\iny \nu^n_{t} d{t}.
\end{align}
As $\disp \mu^n_t = \frac{2\ga_{a} }{\aSp{dn+a} \pr{Mt}^{\frac{dn +a}{2}} } \eta^{dn, a}_{\sqrt{Mt}}$,
and with $r = \abs{Y}$, $t = \frac{r^2}{M}$ so that $d t = \frac{2r}{M} dr$, then we claim that
\begin{equation}
\label{mundDefn}
d\mu^n 
= \frac{4 \ga_{a} \, y_0^{a}}{M \aSp{dn+a} \pr{y_0^2 + \abs{y}^2}^{\frac{dn -1+a}{2}}} d y_0 \, d{y}
= \frac{4 \ga_{a} }{M \aSp{dn+a}} \abs{Y}^{-(dn-1+a)} y_0^{a} \, d Y.
\end{equation}
To check this, set $\chi_\tau = \chi_{\BB{D}^{dn+1}_{\sqrt{M \tau}}}$. 
By the change-of-variables formula,
\eqref{FdnIntRel}, \eqref{FnImage}, and that $\nu^n_{t}(\BB{D}^{d+1}_{\sqrt{Mnt}}) = 1$,
we get
\begin{align*}
    \int_{\R^{dn+1}_+} \chi_\tau d \mu^n
    &= \int_{\BB{D}^{dn+1}_{\sqrt{M \tau}}} d \mu^n
    =\int_{F_{n}\pr{\BB{D}^{dn+1}_{\sqrt{M \tau}}}} d\pr{F_{n} \# \mu^n}
    = \pr{F_{n} \# \mu^n}\pr{F_{n}\pr{\BB{D}^{dn+1}_{\sqrt{M \tau}}}} 
    = \int_0^{\tau} \nu^n_{t}\pr{\BB{D}^{d+1}_{\sqrt{Mnt}}}  dt
    = \tau.
\end{align*}
On the other hand, a direct computation with \eqref{mundDefn} shows that
\begin{align*}
    \int_{\R^{dn+1}_+} \chi_\tau d \mu^n
    &= \int_0^{\sqrt{M\tau}} \int_{\set{\abs{y}^2 \le M \tau - y_0^2}} \frac{4 \ga_{a} \, y_0^{a}}{M \aSp{dn+a} \pr{y_0^2 + \abs{y}^2}^{\frac{dn-1+a}{2}}} d{y} \, d y_0 
    = \tau,
\end{align*}
where we use \eqref{sphereSize} and \eqref{gasDefn} to simplify.
As both computations give the same result, and $\mu^n$ is rotation-invariant in $y$, we may conclude that \eqref{mundDefn} is valid.

By analogy with Lemma \ref{PFTS}, we arrive at the following result.

\begin{lem}[Variable $t$ bridge]
\label{PFT}
For $d\in \N$ and $a \in (-1,1)$, let $\mathcal{G}_n(X,t)$ and $\mathcal{G}(X,t)$ be as defined in \eqref{GsdnDefn}.
If $\phi: \R^{d+1}_+ \times \pr{0, T_0} \to \R$ is integrable with respect to $x_0^{a} \, \mathcal{G}(X,t) d{X} dt$, then for every $n \in \N$, $\phi$ is integrable with respect to $x_0^{a} \, \mathcal{G}_n(X,t) d{X} dt$ and for every $T \le T_0$, it holds that
\begin{align*}
\int_{\BB{D}^{dn+1}_{\sqrt{MT}}} \phi\pr{F_{n}(Y)} y_0^{a} \, dY
&= \frac{\bar{C} \sqrt M}{2} \int_0^T \int_{\R^{d+1}_+} \phi\pr{X, t} t^{\frac{dn-1+a}{2}} \dGnt dt,
\end{align*}
where we define $\bar{C}$ in 
\eqref{CbarDefn}. 
If $T = \iny$, then
\begin{align*}
\int_{\R^{dn + 1}_+}\phi\pr{F_{n}(Y)} y_0^{a} \, dY
&= \frac{\bar{C} \sqrt M}{2}\int_0^\iny \int_{\R^{d+1}_+} \phi(X,t) t^{\frac{dn-1+a}{2}} \dGnt d{t}.
\end{align*}
\end{lem}

\begin{proof}
By the argument from the previous proof, if $\phi$ is integrable with respect to $x_0^{a} \,  \mathcal{G}(X,t) d{X} \, d{t}$, then $\phi$ is also integrable with respect to each $x_0^{a} \, \mathcal{G}_n(X,t) d{X} \, d{t}$.

Since $\phi$ is integrable with respect to each $x_0^{a} \, \mathcal{G}_n(X,t) d{X} dt$, then by comparing \eqref{GsdnDefn} with \eqref{nundDefn}, recalling \eqref{FnImage} and \eqref{FdnIntRel}, we see that

\begin{align*}
    \int_0^T \int_{\R^{d+1}_+} \phi(X,t) \dGnt dt
    &= \int_0^T \int_{\R^{d+1}_+} \phi(X,t) d \nu^n_t(X) dt
    = \int_0^T \int_{\BB{D}^{d+1}_{\sqrt{M nt}}} \phi(X,t) d \nu^n_t(X) dt \\
    &    = \int_{F_{n}\pr{\BB{D}^{dn+1}_{\sqrt{ MT}}}} \phi(X,t) d\pr{F_{n} \# \mu^n}(X, t)
    = \int_{\BB{D}^{dn+1}_{\sqrt{ MT}}} \phi\pr{F_{n}(Y)} d\mu^n(Y),
\end{align*}
where the last step follows from the change-of-variables formula.
An application of \eqref{mundDefn} shows that
\begin{align}
\label{bridgeResultwithY}
    \int_{\BB{D}^{dn+1}_{\sqrt{ MT}}} \phi\pr{F_{n}(Y)} \abs{Y}^{-(dn-1+a)} y_0^{a} \, d Y
    &= \frac{M \aSp{dn+a}}{4 \ga_{a}} \int_0^T \int_{\R^{d+1}_+} \phi(X,t) \, \dGnt dt.
\end{align}
We use that $\abs{Y} = \sqrt{Mt}$ to reach the conclusion.
\end{proof}

\section{Integral Relationships}
\label{S:IntRel}

In this section, we apply the results of the previous section, namely the ``Bridge Lemmas," to establish a number of tools that will be used repeatedly.
To simplify our notation, we introduce the functionals that appear often.

\begin{defn}[Common functionals]
\label{functionalDefs}
For $d\in \N$ and $a \in (-1,1)$, let $\mathcal{G}(X,t)$ be as in \eqref{GsdnDefn}.
We define the functionals
\begin{equation}
\label{12functionals}
\begin{aligned}
\mathcal{H}(t)
&= \mathcal{H}(t; U) 
= \int_{\R^{d+1}_+} \abs{U(X,t)}^2 \dGt \\
\mathcal{D}(t)
&= \mathcal{D}(t; U) 
= \int_{\R^{d+1}_+} \abs{\gr U(X,t)}^2 \dGt \\
\mathcal{T}(t)
&=\mathcal{T}(t;U)
=\int_{\R^{d+1}_+} \abs{\del_t U(X,t)}^2 \dGt \\
\mathcal{I}(t)
&= \mathcal{I}(t; U) 
= \int_{\R^{d+1}_+} \abs{(X,2t) \cdot \gr_{(X,t)} U(X,t)}^2 \dGt \\
\mathcal{J}(t)
&=\mathcal{J}(t;U)
=\int_{\R^{d+1}_+} \abs{J(X,t)}^2 \dGt
\\
\mathcal{M}(t)
&=\mathcal{M}(t;U)
=\int_{\R^{d+1}_+} U(X,t) J(X,t) \dGt,
\end{aligned}
\end{equation} 
where $J(X,t)$ is given in \eqref{JDefn}.
We write $\mathcal{J}^-$ and $\mathcal{M}^-$ when $J$ is replaced by $J^-$.
Let $\mathcal{H}_n(t)$, $\mathcal{D}_n(t)$, $\mathcal{T}_n(t)$, $\mathcal{I}_n(t)$, $\mathcal{J}_n(t)$, and $\mathcal{M}_n(t)$ be defined analogously with $\mathcal{G}_n(X,t)$ in place of $\mathcal{G}(X,t)$.
That is, for example,
\begin{equation*}
  \mathcal{H}_n(t)
= \mathcal{H}_n(t; U) 
= \int_{\R^{d+1}_+} \abs{U(X,t)}^2 \, \dGnt.
\end{equation*}
\end{defn}

\begin{rem}
    An application of Lemma \ref{GaussianLimit} shows that $\disp \lim_{n \to \iny} \mathcal{H}_n(t; U) = \mathcal{H}(t; U)$ for each $t\in (0, T)$.
    Analogous statements hold for $\mathcal{D}_n(t)$, $\mathcal{T}_n(t)$, $\mathcal{I}_n(t)$, $\mathcal{J}_n(t)$, and $\mathcal{M}_n(t)$.
\end{rem}

We collect applications of Lemmas \ref{PFTS} and \ref{PFT} that we use frequently.

\begin{lem}[Applications of the bridge lemmas]
\label{bridgeLemmaResults}
For $U(X,t)$ defined on $\R^{d+1}_+ \times (0,T)$, let $V_n : \BB{D}^{dn+1}_{\sqrt{MT}} \to \R$ be given by $V_n \pr{Y} = U\pr{F_{n}\pr{Y}} = U(X,t)$.

With $J$ as defined in \eqref{JDefn}, let $K_n$ be as defined in \eqref{KnDef}.
Assume that $\mathcal{H}(t)$, $\mathcal{D}(t)$, $\mathcal{T}(t)$, $\mathcal{I}(t)$, and $\mathcal{M}(t)$ are all well-defined.
For any $t \in (0, T)$, it holds that
\begin{align}
&\frac{1}{\bar{C} t^{\frac{dn+a}{2}}}\int_{\BB{H}^{dn}_{\sqrt{M t}}} \abs{V_n\pr{Y}}^2 y_0^a  
= \int_{\R^{d+1}_+}  \abs{U(X,t)}^2 \dGnt
=  \mathcal{H}_n(t)
\label{HnExpr} \\
&\frac{1}{\bar{C} t^{\frac{dn+a}{2}}} \int_{\BB{H}^{dn}_{\sqrt{Mt}}} \, V_n \pr{Y \cdot \gr V_n} y_0^a 
= \int_{\R^{d+1}_+} U(X,t) \brac{(X,2t) \cdot \gr_{(X,t)} U(X,t)} \dGnt
\le  \brac{\mathcal{H}_n(t) \mathcal{I}_n(t)}^{\frac 1 2}
\label{sphereMixedTerm} \\
&\frac{1}{\bar{C} t^{\frac{dn +a}{2}} } \int_{\BB{H}^{dn}_{\sqrt{Mt}}} \, \abs{\gr V_n(Y)}^2 y_0^a
\le 2 n \mathcal{D}_n(t)
+ \frac{8t}{M} \mathcal{T}_n(t)
\label{sphereGradientSquared} 
\end{align}

\begin{align}
&\frac{\sqrt M}{2\bar{C}} \int_{\BB{D}^{dn+1}_{\sqrt{Mt}}} V_n(Y) K_n(Y) y_0^a 
= \int_0^t \tau^{\frac{dn-1+a}2} \mathcal{M}_n(\tau) d\tau 
\le \pr{\int_0^t \tau^{\frac{dn-1+a}2} \mathcal{I}_n(\tau)  d\tau }^{\frac 1 2} \pr{ \int_0^t \tau^{\frac{dn-1+a}2} \mathcal{J}_n(\tau) d\tau }^{\frac 1 2}
\label{ballKnVn} \\
&\frac{\sqrt M}{2\bar{C}}\int_{\BB{D}^{dn+1}_{\sqrt{Mt}}}  (\, Y \cdot \gr V_n) K_n y_0^a
\le \pr{\int_0^t \tau^{\frac{dn-1+a}2} \mathcal{I}_n(\tau)  d\tau }^{\frac 1 2} \pr{ \int_0^t \tau^{\frac{dn-1+a}2} \mathcal{J}_n(\tau) d\tau }^{\frac 1 2}
\label{ballKngradBound} \\
&\frac{\sqrt M}{2\bar{C}} \int_{\BB{D}^{dn+1}_{\sqrt{Mt}}} \abs{\gr V_n\pr{Y}}^2 y_0^a
\le \frac{Mn}{2} \int_0^t \tau^{\frac{dn-1+a}{2}} \mathcal{D}_n(\tau) d\tau
+ 2 \int_0^t \tau^{\frac{dn+1+a}{2}} \mathcal{T}_n(\tau) d\tau
\label{grVnBound1}
\end{align}

\begin{align}
\frac{\ga_{a}}{\aSp{dn+a}} \int_{\BB{D}^{dn+1}_{\sqrt{Mt}}} V_n(Y) K_n^-(Y) \Ga(Y) & y_0^a 
= \frac M 4 \int_0^t \mathcal{M}_n^-(\tau) d\tau
\label{ballKnVnWeighted} \\
\frac{\ga_{a}}{\aSp{dn+a}} \int_{\BB{D}^{dn+1}_{\sqrt{Mt}}} \abs{\gr V_n(Y)}^2 \Ga(Y) y_0^a 
&= \int_0^t \int_{\R^{d+1}_+} \brac{\frac{M n}4 \abs{\gr U}^2 + \tau \abs{\del_\tau U}^2 + \pr{X \cdot \gr U} \del_\tau U} \dGnta d\tau \nonumber \\
&\le \frac{Mn}{2} \int_0^t \mathcal{D}_n(\tau) d\tau
+ 2 \int_0^t \tau \mathcal{T}_n(\tau) d\tau,
\label{ballgradVnWeighted}
\end{align}
where $\Ga(Y) = \abs{Y}^{-(dn-1+a)}$.
\end{lem}

\begin{proof}
Equations \eqref{HnExpr} and \eqref{sphereMixedTerm} follow from applications of Lemma \ref{PFTS} combined with the expression \eqref{YdotGrExp}; H\"older's inequality leads to the bound in \eqref{sphereMixedTerm}. 

Lemma \ref{PFTS} and \eqref{gradExp} show that
\begin{align*}
\int_{\BB{H}^{dn}_{\sqrt{Mt}}} \, \abs{\gr V_n(Y)}^2 y_0^a d\eta(Y)
&= \bar{C} t^{\frac{dn +a}{2}} \int_{\R^{d+1}_+} \pr{n \abs{\gr U}^2 + \frac {4t} {M} \abs{\del_t U}^2 + \frac 4 M \pr{X \cdot \gr U} \del_t U} \dGnt 
\end{align*}
while Young's inequality gives
\begin{align*}
\pr{n \abs{\gr U}^2 + \frac {4t} {M} \abs{\del_t U}^2 + \frac 4 M \pr{X \cdot \gr U} \del_t U}
&\le \pr{n + \frac {\abs{X}^2}{Mt}} \abs{\gr U}^2
+ \frac {8t} {M} \abs{\del_t U}^2.
\end{align*}

Since $\abs{X}^2 \le M n t$ on the support of $\mathcal{G}_n(X, t)$, then we reach \eqref{sphereGradientSquared}.
Analogous arguments with Lemma \ref{PFT} and \eqref{bridgeResultwithY} in place of Lemma \ref{PFTS} lead to \eqref{grVnBound1} and \eqref{ballgradVnWeighted}, respectively.

Equations \eqref{ballKnVn} (and \eqref{ballKnVnWeighted}) are a consequence of Lemma \ref{PFT} (see \eqref{bridgeResultwithY} in the proof),  \eqref{JDefn}, and \eqref{KnDef}; H\"older's inequality leads to the upper bound in \eqref{ballKnVn}.

Lemma \ref{PFT}, \eqref{YdotGrExp}, and \eqref{KnDef} show that
\begin{align*}
&\int_{\BB{D}^{dn+1}_{\sqrt{Mt}}}  (\, Y \cdot \gr V_n) K_n y_0^a dY
= \frac{2\bar{C}}{\sqrt M } \int_0^t \tau^{\frac{dn-1+a}2}\int_{\R^{d+1}_+}  \brac{\pr{X, 2 \tau}  \cdot \gr_{(X, \tau)} U(X, \tau)} J(X, \tau) \, \dGnta d\tau,
\end{align*}
which leads to \eqref{ballKngradBound} after an application of H\"older's inequality.
\end{proof}

We now introduce relevant elliptic and parabolic function spaces and show that there is a natural relationship between them.
We begin with the weighted Sobolev spaces.

\begin{defn}[Weighted Sobolev spaces]
    For $\Om \su \R^{N+1}$ open and bounded, the \textbf{weighted Sobolev space} $H^1\pr{\Om, \abs{y_0}^a}$ is defined as the closure of $C^\iny\pr{\Om}$ under the norm
    $$\norm{\vp}^2_{H^1\pr{\Om, \abs{y_0}^a}}
    := \int_{\Om} \abs{\vp(Y)}^2 \abs{y_0}^a dY + \int_{\Om} \abs{\gr \vp(Y)}^2 \abs{y_0}^a dY.$$
    The zero trace weighted Sobolev space $H^1_0\pr{\Om, \abs{y_0}^a}$ is defined as the closure of $C^\iny_0\pr{\Om}$ under the $H^1(\Om, \abs{y_0}^a)$-norm.
    If $\Om \su \R^{N+1}_+$, then we write $y_0^a$ in place of $\abs{y_0}^a$.
\end{defn}

Going forward, we focus on $\Om = \BB{B}_r$ and $\BB{D}_r$ since this is how we apply these results.
We begin with a couple of observations about weighted traces.

\begin{lem}[Trace lemma]
\label{traceLemma}
If $V \in H^1\pr{\BB{D}_R, y_0^a}$, then for any $r \in (0, R)$,
\begin{equation}
\label{trace1}
\lim_{\eps \to 0^+}\int_{\BB{B}_r \cap \set{y_ 0 = \eps}} \abs{V(Y)}^2 y_0^{1+a} dy = 0.
\end{equation}
If we additionally assume that $y_0^{-a}\di\pr{y_0^{a} \gr V} \in L^2\pr{\BB{D}_R, y_0^{2+a}}$, then for any $r \in (0, R)$,
\begin{equation}
\label{trace2}
\lim_{\eps \to 0^+}\int_{\BB{B}_r \cap \set{y_ 0 = \eps}} \brac{\abs{\gr V}^2 - 2 \pr{\del_{y_0} V}^2} y_0^{1+a} dy = 0.
\end{equation}
\end{lem}

\begin{pf}
For $r \in (0, R)$ and $\eps \in \pr{0, \frac{R- r}2}$, let $\zeta_\eps \in C^\iny_0(\BB{B}_R)$ be a nonnegative cutoff function that is supported on the $2\eps$-neighborhood of $\BB{B}_r \cap \set{y_ 0 = \eps}$ with $\zeta_\eps \equiv 1$ on the $\eps$-neighborhood of $\BB{B}_r \cap \set{y_ 0 = \eps}$.
Since $\abs{\gr \zeta_\eps} \lesssim \eps^{-1}$, then $y_0 \abs{\gr \zeta_\eps} \lesssim 1$.
An integration by parts gives
\begin{align*}
\int_{\BB{B}_r \cap \set{y_ 0 = \eps}} V^2 y_0^{1+a} dy
&\le 
\int_{\BB{B}_R \cap \set{y_ 0 = \eps}} \zeta_\eps V^2 y_0^{1+a} dy
= - \int_{\BB{B}_R \cap \set{y_ 0 > \eps}} \del_{y_0} \pr{\zeta_\eps V^2 y_0^{1+a}} dY \\
&= - \int_{\BB{B}_R \cap \set{y_ 0 > \eps}} \brac{y_0 \del_{y_0} \zeta_\eps + \zeta_\eps \pr{1 + a}}  V^2 y_0^{a} dY
- \int_{\BB{B}_R \cap \set{y_ 0 > \eps}}  2 \zeta_\eps V \del_{y_0} V  y_0^{1+a} dY \\
&\lesssim \int_{\BB{B}_R \cap \set{\eps < y_ 0 < 3 \eps}} \pr{V^2 + \abs{\gr V}^2} y_0^{a} dY,
\end{align*}
where we have used Young's inequality and the properties of $\zeta_\eps$.
Since $V \in H^1\pr{\BB{D}_R, y_0^a}$, then 
$$\lim_{\eps \to 0^+}\int_{\BB{B}_R \cap \set{\eps < y_ 0 < 3 \eps}} \pr{V^2 + \abs{\gr V}^2 }y_0^{a} dY = 0$$
and the first conclusion follows.

We analyze the gradient term in a similar way:
\begin{align*}
\int_{\BB{B}_R \cap \set{y_ 0 = \eps}} \zeta_\eps \brac{\abs{\gr_y V}^2 - \pr{\del_{y_0}V}^2} & y_0^{1+a} dy
= - \int_{\BB{B}_R \cap \set{y_ 0 > \eps}} \del_{y_0} \brac{\zeta_\eps \pr{\abs{\gr_y V}^2 - \pr{\del_{y_0}V}^2} y_0^{1+a}} dY \\
&= \int_{\BB{B}_R \cap \set{y_ 0 > \eps}} \brac{y_0 \del_{y_0} \zeta_\eps + \zeta_\eps \pr{1 + a}} \pr{-\abs{\gr_y V}^2 + \pr{\del_{y_0}V}^2} y_0^{a} dY \\ 
&+ 2 \int_{\BB{B}_R \cap \set{y_ 0 > \eps}} \zeta_\eps  \pr{-\gr_y V \cdot \gr_y \del_{y_0} V + \del_{y_0}V \del_{y_0}^2 V} y_0^{1+a} dY.
\end{align*}
Integrating the last term by parts shows that
\begin{align*}
&\int_{\BB{B}_R \cap \set{y_ 0 > \eps}} \zeta_\eps  \pr{-\gr_y V \cdot \gr_y \del_{y_0} V + \del_{y_0}V \del_{y_0}^2 V} y_0^{1+a} dY \\
&= \int_{\BB{B}_R \cap \set{y_ 0 > \eps}} \zeta_\eps  \del_{y_0}V \pr{\LP_y V +  \del_{y_0}^2 V} y_0^{1+a} dY
+ \int_{\BB{B}_R \cap \set{y_ 0 > \eps}} y_0 \gr_y \zeta_\eps \cdot \gr_y V \del_{y_0} V y_0^{a} dY \\
&= \int_{\BB{B}_R \cap \set{y_ 0 > \eps}} y_0 \zeta_\eps  \del_{y_0}V \di\pr{y_0^{a} \gr V} dY
- a \int_{\BB{B}_R \cap \set{y_ 0 > \eps}} \zeta_\eps \pr{\del_{y_0} V}^2 y_0^{a} dY
+ \int_{\BB{B}_R \cap \set{y_ 0 > \eps}} y_0 \gr_y \zeta_\eps \cdot \gr_y V \del_{y_0} V y_0^{a} dY.
\end{align*}
Simplifying and using the bounds on $\zeta_\eps$ as above gives
\begin{align*}
\int_{\BB{B}_r \cap \set{y_ 0 = \eps}} \brac{\abs{\gr V}^2 - 2 \pr{\del_{y_0} V}^2} y_0^{1+a} dy
&\lesssim \int_{\BB{B}_R \cap \set{\eps < y_ 0 < 3 \eps}} \pr{\abs{\gr V}^2 +y_0^2 \brac{y_0^{-a}\di\pr{y_0^{a} \gr V}}^2}y_0^{a} dY.
\end{align*}
Arguing as in the first case leads to the second conclusion.
\end{pf}

We define a preliminary space for our parabolic functions.

\begin{defn}[Moderate $a$-growth at infinity]
\label{aGrowth}
Let $U: \R^{d+1}_+ \times \pr{0, T} \to \R$ be a  function with weak first order derivatives. We say that $U$ has \textbf{moderate $a$-growth at infinity} if for every $t \in (0, T)$, the functionals $\mathcal{H}\pr{\cdot \ ; U}$, $\mathcal{D}\pr{\cdot \ ; U}$, and $\mathcal{T}\pr{\cdot \ ; U}$ from Definition \ref{functionalDefs} all belong to $L^1((0,t))$.
\end{defn}

The transformations send functions from the parabolic space to the appropriate elliptic space, defined above.

\begin{lem}[Inheritance of function classes]
\label{functionClassLemma}
    If $U$ defined on $\R^{d+1}_+ \times (0, T)$ has moderate $a$-growth at infinity, then each $V_n : \BB{D}^{dn+1}_{\sqrt{MT}} \to \R$ defined by 
    $V_n \pr{Y} = U\pr{F_n\pr{Y}}$
    belongs to $H^1\pr{\BB{D}^{dn+1}_{\sqrt{M t}}, y_0^a}$ for every $t \in (0, T)$.
\end{lem}

\begin{proof}
Fix $t \in (0, T)$.
With the notation from \cite{Kil94}, we show that $V_n \in W^{1,2}\pr{\BB{D}^{dn+1}_{\sqrt{M t}}, y_0^a}$ by checking that $\disp \int_{\BB{D}_{\sqrt{M t}} } \pr{\abs{V_n(Y)}^2 + \abs{\gr V_n(Y)}^2} y_0^a dY < \iny$.
 
 By Lemma \ref{PFT}, \eqref{pointwiseGBound} in Lemma \ref{GaussianLimit}, and Definition \ref{functionalDefs},
\begin{align*}
\int_{\BB{D}_{\sqrt{M t}} } \abs{V_n(Y)}^2 y_0^a dY
&= \frac{\bar{C} \sqrt M}{2} \int_0^t \int_{\R^{d+1}_+} \abs{U\pr{X, \tau}}^2 \tau^{\frac{dn-1-a}{2}} \dGnta d\tau
\le \frac{C_0 \bar{C} \sqrt M}{2} \int_0^t \tau^{\frac{dn-1-a}{2}} \mathcal{H}(\tau) d\tau \\
&\le \frac{C_0 \bar{C} \sqrt M}{2} t^{\frac{dn-1-a}{2}} \norm{\mathcal{H}}_{L^1([0,t])}
< \iny.
\end{align*}
From \eqref{grVnBound1} in Lemma \ref{bridgeLemmaResults} and \eqref{pointwiseGBound} again, we see that
\begin{align*}
\int_{\BB{D}_{\sqrt{Mt}}} \abs{\gr V_n\pr{Y}}^2 y_0^a \, dY
&\le n C_0 \bar{C}\sqrt M t^{\frac{dn-1+a}{2}} \norm{\mathcal{D}}_{L^1([0,t])}
+ \frac{4 C_0 \bar{C}}{\sqrt M} t^{\frac{dn+1+a}{2}} \norm{\mathcal{T}}_{L^1([0,t])}
< \iny.
\end{align*}
As shown in \cite{Kil94}, $W^{1,2}\pr{\BB{D}^{dn+1}_{\sqrt{M t}}, y_0^a} = H^1\pr{\BB{D}^{dn+1}_{\sqrt{M t}}, y_0^a}$, so the conclusion follows.
\end{proof}

\section{Almgren Frequency Functions}
\label{S:almgren}

In this section, we prove an Almgren-type monotonicity formula for solutions to nonhomogeneous degenerate elliptic equations. 
We then use our elliptic monotonicity formula to prove a parabolic monotonicity formula for solutions to degenerate parabolic equations. 
More precisely, we use the elliptic result to show that the Almgren frequency
function associated with the degenerate parabolic operator, $x_0^{a}\partial_tU+\text{div}(x_0^{a}\nabla U)$, is monotonically non-decreasing. 

\subsection{Almgren frequency functions in the elliptic setting}
\label{SS:EAlmgren}

We start by adapting the monotonicity formula given in \cite{CS07} to the nonhomogeneous setting.

Recall that $Y = \pr{y_0, y} \in \R^{N+1}_+$.
We work in upper half-balls and upper half-spheres, denoted by $\BB{D}_R^{N+1}$ and $\BB{H}_R^N$, respectively.
At times, when the dimension is understood from the context, we write $\BB{D}_R$ and $\BB{H}_R$ in place of $\BB{D}_R^{N+1}$ and $\BB{H}^N_R$. 
Note that the boundary of $\BB{D}_R$, denoted by $\del \BB{D}_R$, is the disjoint union of $\BB{H}_R$ and the set $\BB{B}_R \cap \set{y_0 = 0}$, where $\BB{B}_R$ denotes the ball of radius $R$.
Our elliptic Almgren monotonicity result is the following:

\begin{thm}[Nonhomogeneous Elliptic Almgren, cf. Theorem 6.1 in \cite{CS07}]
\label{monoThm}
For $V \in H^1\pr{\BB{D}_R, y_0^a} \cap C^2\pr{\BB{D}_R}$, define
\begin{equation}
\label{HDLDefn}
\begin{aligned}
& H\pr{r; V} = \int_{\BB{H}_r} \abs{V\pr{Y}}^2 y_0^{a} \, d{\eta} \\ 
& D\pr{r; V} = \int_{\BB{D}_r} \abs{\gr V\pr{Y}}^2 y_0^{a} \, d{Y} \\
& L\pr{r; V} = \frac{r \, D\pr{r; V}}{ H\pr{r; V}}.
\end{aligned}    
\end{equation}
If $V$ satisfies the elliptic boundary conditions as in Definition \ref{eBC} and is a solution to 
\begin{equation}
\LP_y V + \frac{a}{y_0} \del_{y_0} V + \del^2_{y_0} V = K \quad \text{ in } \; \BB{D}_{R},
\label{VEqn}
\end{equation}
where $K \in L^2\pr{\BB{D}_R, y_0^a}$, then for all $r \in (0, R)$ at which $H(r; V) > 0$, it holds that
\begin{align*}
L'\pr{r} \pr{\int_{\BB{H}_r} \abs{V(Y)}^2 y_0^{a} \, d\eta }^2
&\ge 2 \pr{\int_{\BB{H}_r} V(Y) \, \pr{Y \cdot \gr V(Y)} \, y_0^{a} \, d\eta} \pr{ \int_{\BB{D}_r} V(Y) \, K(Y) \, y_0^{a} \, d{Y}} \\
&- 2 \pr{\int_{\BB{H}_r} \abs{V(Y)}^2 y_0^{a} \, d \eta} \pr{\int_{\BB{D}_r} \pr{ Y \cdot \gr V(Y)} \,K(Y) \,  y_0^{a} \, d{Y} } .
\end{align*}
\end{thm}

To prove this theorem, we need the following nonhomogeneous version of the lemma that is used in the proof of \cite{CS07}*{Theorem 6.1}.

\begin{lem}[Nonhomogeneous version of Lemma 6.2 in \cite{CS07}]
\label{Lem6.2}
For $V \in H^1\pr{\BB{D}_R, y_0^a} \cap C^2\pr{\BB{D}_R}$, assume that there exists $\eps_0 > 0$ so that $\disp y_0^{a} \del_{y_0} V(y_0, y)$ is bounded whenever $(y_0, y) \in \mathbb{D}_R \cap \set{y_0 \in(0, \eps_0)}$ and either $\disp \del^a_{\nu} V(0,y) = 0$ or $\disp \lim_{y_0 \to 0^+} y \cdot \gr_y V\pr{y_0, y} = 0$.
If $V$ is a solution to \eqref{VEqn} in $\BB{D}_{R}^{N+1}$, where $K \in L^2\pr{\BB{D}_R, y_0^a}$, then for any $r \in (0, R)$, it holds that
\begin{align*}
r \int_{\BB{H}_r} \brac{ \abs{\gr_{\tan} V(Y)}^2 - \pr{\frac{\del V(Y)}{\del \nu}}^2} y_0^{a} \, d\eta
&= \int_{\BB{D}_r} \brac{\pr{N - 1+a} \abs{\gr V(Y)}^2 - 2 \pr{Y \cdot \gr V(Y)} K(Y)} y_0^{a}  \, d{Y},
\end{align*}
where $\gr_{\tan}$ denotes for the tangential gradient of $V$, and $\frac{\del}{\del \nu}$ denotes the normal derivative.
\end{lem}

\begin{rem}
    We do not need the full power of the elliptic boundary condition for this result, so we have made a slightly weaker assumption.
\end{rem}

\begin{proof}
By \eqref{VEqn}, $\di \pr{y_0^{a} \, \gr V} = y_0^{a} \, K$, so a computation shows that
\begin{align*}
\di \pr{\frac{\abs{\gr V}^2}{2} Y \, y_0^{a}  - \pr{Y \cdot \gr V} \gr V \, y_0^{a} } 
=& \pr{\frac{N - 1+a}{2} \abs{\gr V}^2 -  Y \cdot \gr V K }\, y_0^{a}.
\end{align*}
For any $\eps \in (0, \eps_0)$, applying the divergence theorem on the set $\BB{B}_r \cap \set{y_0 > \eps}$ shows that
\begin{align*}
&r \int_{\del \BB{B}_r \cap \set{y_0 > \eps}} \brac{ \frac{\abs{\gr V}^2}{2} - \pr{\frac{\del V}{\del \nu}}^2} y_0^{a} \, d \eta
- \int_{\BB{B}_r \cap \set{y_0 > \eps}} \pr{\frac{N - 1+a}{2} \abs{\gr V}^2 - \pr{Y \cdot \gr V} K } y_0^{a} \, d{Y}\\
=& \int_{\mathbb{B}_r \cap \set{y_0 = \eps}} \pr{\frac{\abs{\gr V\pr{y_0, y}}^2}{2} 
- \abs{\del_{y_0} V\pr{y_0, y}}^2} y_0^{1+a} dy
- \int_{\mathbb{B}_r \cap \set{y_0 = \eps}} y \cdot \gr_y V\pr{y_0, y} y_0^{a}  \del_{y_0} V\pr{y_0, y} dy.
\end{align*}
Since $K \in L^2\pr{\BB{D}_R, y_0^a}$, then by \eqref{trace2} in Lemma \ref{traceLemma}, we deduce that the first term on the right vanishes in the limit as $\eps \to 0^+$. 
The boundary assumptions with the dominated convergence theorem ensure that the second term on the right also vanishes in the limit.
It follows that
\begin{align*}
r \int_{\BB{H}_r} \brac{ \frac{\abs{\gr V}^2}{2} - \pr{\frac{\del V}{\del \nu}}^2} y_0^{a} \, d\eta 
&= \int_{\BB{D}_r} \pr{\frac{N - 1+a}{2} \abs{\gr V}^2 - \pr{Y \cdot \gr V} K} y_0^{a}  \,  d{Y},
\end{align*}
which leads to the conclusion of the lemma since $\disp \abs{\gr V}^2 = \abs{\gr_{\tan} V}^2 + \pr{\frac{\del V}{\del \nu}}^2$.
\end{proof}

Now we use Lemma \ref{Lem6.2} to prove Theorem \ref{monoThm}.

\begin{proof}[Proof of Theorem \ref{monoThm}]
Observe that
\begin{align*}
D'\pr{r; V}
&= \frac{d}{dr}\brac{ r^{N+1} \int_{\BB{D}_1} \abs{\gr V\pr{rY}}^2 \pr{r y_0}^{a} } \\
&= r^{N+1} \brac{\frac{N+1}{r} \int_{\BB{D}_1} \abs{\gr V\pr{r Y}}^2 \pr{r y_0}^{a}
+ \int_{\BB{D}_1} \gr\pr{\pr{r y_0}^{a} \abs{\gr V\pr{r Y}}^2} \cdot Y} \\
&= \frac 1 r \int_{\BB{D}_r} \di \pr{y_0^{a} \abs{\gr V\pr{Y}}^2 Y}
= \lim_{\eps \to 0} \frac 1 r \int_{\BB{B}_r \cap \set{y_0 > \eps}} \di \pr{y_0^{a} \abs{\gr V\pr{Y}}^2 Y}  \\
&=  \lim_{\eps \to 0}  \frac 1 r \int_{\del \BB{B}_r \cap \set{\eps > 0}} \abs{\gr V\pr{Y}}^2 \pr{Y \cdot \nu} \, y_0^{a} 
+  \lim_{\eps \to 0}  \frac 1 r \int_{\BB{B}_r \cap \set{y_0 = \eps}} \abs{\gr V\pr{Y}}^2 \pr{Y \cdot \nu} \, y_0^{a}   \\
&=  \int_{\BB{H}_r} \abs{\gr V\pr{Y}}^2 \, y_0^{a} 
+  \lim_{\eps \to 0}  \frac 1 r \int_{\BB{B}_r \cap \set{y_0 = \eps}} \abs{\gr V\pr{Y}}^2 \, y_0^{1+a} . 
\end{align*}
Since $V$ satisfies the elliptic boundary conditions, then $y_0^a \del_{y_0} V$ is bounded near $\set{y_0 = 0}$ and we see that 
\begin{align*}
\lim_{\eps \to 0^+} \int_{\BB{B}_r \cap \set{y_0 = \eps}} \pr{\del_{y_0} V}^2 \, y_0^{1+a}
&= \lim_{\eps \to 0^+} \int_{\BB{B}_r \cap \set{y_0 = \eps}} \pr{y_0^a \del_{y_0} V}^2 \, y_0^{1-a}
= 0,
\end{align*}
since $1 - a > 0$.
In combination with Lemma \ref{traceLemma}, we may conclude that $\disp \lim_{\eps \to 0} \int_{\BB{B}_r \cap \set{y_0 = \eps}} \abs{\gr V\pr{Y}}^2 \, y_0^{1+a} = 0$ and therefore,
\begin{equation}
\label{D'0}
\begin{aligned}
D'\pr{r; V} &= \int_{\BB{H}_r} \abs{\gr V\pr{Y}}^2 y_0^{a}.
\end{aligned}
\end{equation}
An application of Lemma \ref{Lem6.2} shows that
\begin{equation}
\label{gradsphere}
\begin{aligned}
D'\pr{r; V}
&= \int_{\BB{H}_r} \abs{\gr V}^2 y_0^{a}
= \int_{\BB{H}_r} \brac{ \abs{\gr_{\tan} V}^2 - \pr{\frac{\del V}{\del \nu}}^2} y_0^{a}
+ 2 \int_{\BB{H}_r} \pr{\frac{\del V}{\del \nu}}^2 y_0^{a} \\
&= \frac {N - 1-a} r D\pr{r; V}
- \frac 2 r \int_{\BB{D}_r} \pr{Y \cdot \gr V} K y_0^{a}
+ \frac{2}{r^2} \int_{\BB{H}_r} \pr{Y \cdot \gr V}^2 y_0^{a}.
\end{aligned}
\end{equation}
Since $\abs{\gr V}^2 y_0^{a} = \di\pr{V y_0^{a} \gr V} - V \di\pr{y_0^{a} \gr V}$, then the divergence theorem shows that
\begin{equation}
\label{divExp}
\begin{aligned}
D\pr{r; V} 
&=\int_{\BB{D}_r} \abs{\gr V}^2 y_0^{a}
= \lim_{\eps \to 0^+} \int_{\BB{B}_r \cap \set{y_0 > \eps}}  \di\pr{V y_0^{a} \gr V}
- \int_{\BB{D}_r} V \di\pr{y_0^{a} \gr V} \\
&= \int_{\BB{H}_r} V \frac{\del V}{\del \nu} y_0^a
- \lim_{\eps \to 0^+} \int_{\BB{B}_r \cap \set{y_0 = \eps}} V y_0^{a} \del_{y_0} V
- \int_{\BB{D}_r} V K y_0^{a} \\
&= \frac 1 r \int_{\BB{H}_r} V \pr{Y \cdot \gr V} y_0^{a}
- \int_{\BB{D}_r} V K y_0^{a},
\end{aligned}
\end{equation}
where the elliptic boundary condition has been used to eliminate the limit term. Another computation shows that
\begin{equation}
\label{H'}
\begin{aligned}
H'\pr{r; V}
&= \frac{d}{dr}\brac{ r^{N+a} \int_{\BB{H}_1} \abs{V\pr{rY}}^2 y_0^{a}} 
= r^{N+a} \brac{\frac{N+a}{r}  \int_{\BB{H}_1}  \abs{V\pr{rY}}^2 y_0^{a}
+ 2 \int_{\BB{H}_1} V\pr{r Y} \gr V\pr{r Y} \cdot Y \, y_0^{a}} \\
&= \frac{N+a}{r} H\pr{r; V}
+ \frac 2 r \int_{\BB{H}_r} V \pr{Y \cdot \gr V} \, y_0^{a}.
\end{aligned}
\end{equation}
Assuming that $H(r; V) > 0$, taking the derivative and plugging in \eqref{gradsphere} and \eqref{H'} shows that
\begin{align*}
L'\pr{r; V} \pr{H\pr{r; V}}^2
&=  D\pr{r; V}  H\pr{r; V}
- r D\pr{r; V} H'\pr{r; V} 
+  r D'\pr{r; V}  H\pr{r; V} \\
&= D\pr{r; V}  H\pr{r; V} 
- r D(r; V) \brac{\frac{N+a}{r} H\pr{r; V}
+ \frac 2 r \int_{\BB{H}_r} V \pr{Y \cdot \gr V} y_0^{a}} \\
&+  r \brac{\frac {N - 1- a} r D\pr{r; V}
- \frac 2 r \int_{\BB{D}_r} \pr{Y \cdot \gr V} K y_0^{a} 
+ \frac{2}{r^2} \int_{\BB{H}_r} \pr{Y \cdot \gr V}^2 y_0^{a}} H\pr{r; V} \\
&= \brac{1 - (N + a) + N - 1- a } D\pr{r; V}  H\pr{r; V} \\
&- 2 \brac{\frac 1 r \int_{\BB{H}_r} V \pr{Y \cdot \gr V} y_0^{a}
- \int_{\BB{D}_r} V K y_0^{a}} \brac{ \int_{\BB{H}_r} V \pr{Y \cdot \gr V} y_0^{a}} \\
&+  2 \brac{- \int_{\BB{D}_r} \pr{Y \cdot \gr V} K y_0^{a} 
+ \frac{1}{r} \int_{\BB{H}_r} \pr{Y \cdot \gr V}^2 y_0^{a}} \int_{\BB{H}_r} \abs{V}^2 y_0^{a}  \\
&= \frac 2 r \brac{ \pr{\int_{\BB{H}_r} \pr{Y \cdot \gr V}^2 y_0^{a}} \pr{\int_{\BB{H}_r} \abs{V}^2 y_0^{a}}
- \pr{\int_{\BB{H}_r} V \pr{Y \cdot \gr V} y_0^{a} \,  }^2 } \\
&+ 2 \brac{ \pr{\int_{\BB{H}_r} V \pr{Y \cdot \gr V} y_0^{a} } \pr{\int_{\BB{D}_r} V K y_0^{a}}
- \pr{\int_{\BB{H}_r} \abs{V}^2 y_0^{a}} \pr{\int_{\BB{D}_r} \pr{Y \cdot \gr V} K y_0^{a}}},
\end{align*}
where we have applied \eqref{divExp} then simplified to reach the last line.
The Cauchy-Schwartz inequality shows that the first line is non-negative and  the conclusion of the theorem follows.
\end{proof}

\begin{rem}
\label{noBCRemark}
    We point out that the boundary assumptions are not used to establish the equality described by \eqref{H'}.
\end{rem}

\subsection{Almgren frequency functions in the parabolic setting}
\label{SS:PAlmgren}

We use Theorem \ref{monoThm} to establish its parabolic counterpart. 
To begin, we define the class of functions that we work with.

\begin{defn}[Almgren function class]
We say that a function $U = U(X,t)$ defined on $\R^{d+1}_+ \times (0, T)$ belongs to the function class $\mathfrak{A}\pr{\R^{d+1}_+ \times \pr{0, T}}$ if 
\begin{itemize}
    \item $U$ has moderate $a$-growth at infinity (see Definition \ref{aGrowth}),
    \item for every $t_0 \in \pr{0, T}$, there exists $p > 1$, $\al \in \R$ so that
\begin{equation}
\label{AlmgrenClass2}   
t^\al \mathcal{H}, t^\al \mathcal{I}, t^\al \mathcal{J} \in L^{p}\pr{\pr{0, t_0}},
\end{equation}
    \item for every $t_0 \in \pr{0, T}$, there exists $\eps \in \pr{0, t_0}$ so that
\begin{equation}
\label{AlmgrenClass1}
\mathcal{I} \in L^\iny\pr{\brac{t_0 - \eps, t_0}}.
\end{equation}
\end{itemize}
Refer to Definition \ref{functionalDefs} for these functionals.    
\end{defn}

Monotonicity of the Almgren frequency function for solutions to the degenerate parabolic equation is established in \cite{ST17}.
To suit our setting, we reverse the time direction here.

\begin{thm}[Parabolic Almgren Monotonicity, Theorem 1.15 in \cite{ST17}]
\label{pMono}
For $U = U\pr{X,t} \in \mathfrak{A}\pr{\R^{d+1}_+ \times \pr{0, T}}$, 
define
\begin{align}
\label{pAlmFreq}
\mathcal{L}\pr{t; U} 
&= t \frac{\mathcal{D}\pr{t; U}}{\mathcal{H}\pr{t; U}}
= t \frac{\int_{\R^{d+1}_+} \abs{\gr U(X,t)}^2 \dGt}{ \int_{\R^{d+1}_+} \abs{U(X,t)}^2 \dGt}.
\end{align}
If $U$ satisfies the parabolic boundary conditions as in Definition \ref{pBC} and is a solution to 
\begin{equation}
\label{pAlmPDE}
\LP_x U + \frac{a}{x_0} \del_{x_0} U + \del^2_{x_0} U + \del_t U = 0 \text{ in } \R^{d+1}_+ \times \pr{0, T},
\end{equation}
then for all $t \in (0, T)$ for which $\mathcal{H}(t; U) > 0$, it holds that $\mathcal{L}\pr{t; U}$ is monotonically non-decreasing in $t$.
\end{thm}

\begin{rem}
If $t_0 \in Z := \set{t \in (0, T) : \mathcal{H}(t; U) = 0}$, 
then $U(\cdot, t_0) \equiv 0$ so that $\mathcal{D}(t_0; U) = 0$ and then $\mathcal{L}(t_0; U)$ is ill-defined.
Therefore, it is not meaningful to talk about the monotonicity of $\mathcal{L}(t; U)$ at such times. 
\end{rem}

The proof of this theorem is given in \cite{ST17}, for example.
Here we use the high-dimensional limiting technique to prove this theorem.
That is, we use the nonhomogeneous elliptic theorem stated in Theorem \ref{monoThm} in combination with the tools developed in earlier sections to take limits and give a new proof of Theorem \ref{pMono}.

\begin{proof}
With $U$ as given, for each $n \in \N$, define $V_n : \BB{D}^{dn+1}_{\sqrt{MT}} \to \R$ so that
$$V_n \pr{Y} = U\pr{F_{n}\pr{Y}} = U\pr{X, t}.$$
Since $U$ has moderate $a$-growth at infinity, then Lemma \ref{functionClassLemma} shows that each $V_n \in H^1(\mathbb{D}^{dn+1}_{\sqrt{Mt}}, y_0^a)$ for every $t \in (0, T)$.
Since $U$ is a solution to a parabolic equation, then regularity theory ensures that each $V_n$ is $C^2$.
As shown in Lemma \ref{ChainRuleLem}, since $U$ satisfies \eqref{pAlmPDE}, then each $V_n$ satisfies \eqref{VnKnEllipEqn} with $K_n$ as defined in \eqref{KnDef}.

Since $U$ satisfies the parabolic boundary conditions, then Lemma \ref{BCLemma} shows that each $V_n$ satisfies the elliptic boundary conditions.

Choose $t \in P := \set{t \in (0, T) : \mathcal{H}\pr{t; U} > 0}$.
Since $\mathcal{H}_n\pr{t; U} \to \mathcal{H}\pr{t; U} > 0$, then there exists $N \in \N$ so that $\mathcal{H}_n\pr{t; U} > 0$ whenever $n \ge N$.
An application of Lemma \ref{bridgeLemmaResults} shows that $H\pr{\sqrt{M t}; V_n}$ is given by \eqref{HnExpr}.
In particular, $H\pr{\sqrt{M t}; V_n} > 0$ whenever $n \ge N$.
Therefore, for any $n \ge N$, we may apply Theorem \ref{monoThm} to each $V_n$ with $K = K_n$ and $r = \sqrt{Mt}$.

Since $U$ satisfies the parabolic boundary conditions, then so too does $V_n$, so by the expression in \eqref{divExp}, 
\begin{equation}
\label{DnExp}
\begin{aligned}
&\sqrt{Mt} D\pr{\sqrt{Mt}; V_n}
= \int_{\BB{H}^{dn}_{\sqrt{Mt}}} V_n \pr{Y \cdot \gr V_n} y_0^{a}
- \sqrt{Mt} \int_{\BB{D}^{dn+1}_{\sqrt{Mt}}} V_n K_n y_0^{a} \\
&= \bar{C} t^{\frac{dn+a}{2}} \brac{\int_{\R^{d+1}_+} U(X,t) \brac{(X,2t) \cdot \gr_{(X,t)} U(X,t)} \dGnt
- 2 \int_0^t \pr{\frac \tau t}^{\frac{dn-1+a}2}\mathcal{M}_n(\tau) d\tau},
\end{aligned}    
\end{equation}
where we have used \eqref{sphereMixedTerm} and \eqref{ballKnVn} from Lemma \ref{bridgeLemmaResults}.
H\"older's inequality and the bound \eqref{pointwiseGBound} show that
\begin{equation}
\label{secondTermBound}
\abs{\int_0^t \pr{\frac \tau t}^{\frac{dn-1+a}2} \mathcal{M}_n(\tau) d\tau}
\le C_0 \pr{\int_0^t \mathcal{H}(\tau) d\tau}^{\frac 1 2} \pr{\int_0^t \pr{\frac \tau t}^{dn-1+a} \mathcal{J}(\tau) d\tau}^{\frac 1 2}.
\end{equation}
Since $\mathcal{H} \in L^1((0, t))$, then the first term is bounded.
As $\tau^\al \mathcal{J} \in L^p((0, t))$, then $\tau^\al \mathcal{J} \in L^1((0, t))$ as well and the dominated convergence theorem shows that the second term in the product goes to zero as $n \to \iny$. Since $\gr \mathcal{G}(X, t) = - \frac{X}{2t} \mathcal{G}(X, t)$ and $U$ is a solution to \eqref{pAlmPDE}, then
\begin{equation}
\label{pIntbyPs}
\begin{aligned}
& \int_{\R^{d+1}_+} U(X,t) \brac{\frac{X}{2t} \cdot \gr U(X,t) + \del_t U(X,t)} \dGt \\
&= - \int_{\R^{d+1}_+} U(X,t) \brac{\gr \mathcal{G}(X,t) \cdot \gr U(X,t) x_0^a  + \di\pr{x_0^a \gr U} \mathcal{G}(X,t)} dX \\
&= \int_{\R^{d+1}_+} \abs{\gr U(X,t)}^2 \dGt,   
\end{aligned}
\end{equation}
where the last line follows from an integration by parts and uses that $U$ satisfies the parabolic boundary conditions.
From Lemma \ref{GaussianLimit}, we see that
\begin{equation}
\label{othergradExp}
\begin{aligned}
&\lim_{n \to \iny} \brac{\int_{\R^{d+1}_+} U(X,t) \brac{(X,2t) \cdot \gr_{(X,t)} U(X,t)} \dGnt
- 2 \int_0^t \pr{\frac \tau t}^{\frac{dn-1+a}2}\mathcal{M}_n(\tau) d\tau} \\
&= 2t \int_{\R^{d+1}_+} \abs{\gr U(X,t)}^2 \dGt.  
\end{aligned}
\end{equation}
Therefore, with $\disp L_n(t):= \frac 1 2 L\pr{\sqrt{Mt}; V_n}$, we have 
\begin{align*}
L_n(t)
&= \frac{ \int_{\R^{d+1}_+} U(X,t) \brac{(X,2t) \cdot \gr_{(X,t)} U(X,t)} \dGnt}{2  \int_{\R^{d+1}_+}  \abs{U(X,t)}^2 \dGnt} 
- \frac{\int_0^t \pr{\frac \tau t}^{\frac{dn-1+a}2} \mathcal{M}_n(\tau) d\tau }{ \int_{\R^{d+1}_+}  \abs{U(X,t)}^2 \dGnt}
\end{align*}
and comparing with \eqref{pAlmFreq}, we see that
\begin{align*}
\lim_{n \to \iny} L_n(t)
&= \frac{t \int_{\R^{d+1}_+} \abs{\gr U(X,t)}^2 \dGt}{\int_{\R^{d+1}_+}  \abs{U(X,t)}^2 \dGt} 
= \mathcal{L}(t; U).
\end{align*}

An application of Theorem \ref{monoThm} shows that 
\begin{align*}
L'&\pr{\sqrt{M t}; V_n}
\ge 2  \frac{\pr{\int_{\BB{D}_{\sqrt{Mt}}} V_n \, K_n \,  y_0^a} \pr{\int_{\BB{H}_{\sqrt{Mt}}} \, V_n \pr{Y \cdot \gr V_n} y_0^a }}{\pr{\int_{\BB{H}_{\sqrt{Mt}}} \, \abs{V_n}^2 y_0^a }^2}
- 2 \frac{\int_{\BB{D}_{\sqrt{Mt}}} \pr{Y \cdot \gr V_n} K_n y_0^a}{\int_{\BB{H}_{\sqrt{Mt}}} \, \abs{V_n}^2 y_0^a} \\
&\ge -4 \frac{\brac{\int_0^t \tau^{\frac{dn-1+a}2} \mathcal{J}_n(\tau) d\tau}^{\frac 1 2 }}{\sqrt {Mt} \mathcal{H}_n(t)} \set{ \brac{ \frac{\mathcal{I}_n(t)}{\mathcal{H}_n(t)} \int_0^t \pr{\frac \tau t}^{\frac{dn-1+a}2} \mathcal{H}_n(\tau) d\tau}^{\frac 1 2 } + \brac{\int_0^t \pr{\frac \tau t}^{\frac{dn-1+a}2} \mathcal{I}_n(\tau) d\tau}^{\frac 1 2 }  },
\end{align*}
where we have used 
\eqref{HnExpr}, \eqref{sphereMixedTerm}, \eqref{ballKnVn}, and \eqref{ballKngradBound} from Lemma \ref{bridgeLemmaResults} to reach the second line. 
By the chain rule, $\disp \frac{d}{dt}{L}_n(t) = \frac{\sqrt M}{4 \sqrt t} L'\pr{\sqrt{Mt}; V_n}$, so using the bound  \eqref{pointwiseGBound} from Lemma \ref{GaussianLimit}, we deduce that for any $t \in P$,
\begin{equation}
\label{FnBound}
\begin{aligned}
\frac{d}{dt}{L}_n(t)
&\ge - \frac{C_0 \brac{\int_0^t \pr{\frac \tau t}^{\frac{dn -1+a}2} \mathcal{J}(\tau) d\tau }^{\frac 1 2}}{t \mathcal{H}_n\pr{t}} \set{
\brac{\int_0^t \pr{\frac \tau t}^{\frac{dn -1+a}2} \mathcal{I}(\tau) d\tau }^{\frac 1 2}
+ \brac{\frac{C_0 \mathcal{I}(t)}{ \mathcal{H}_n\pr{t}} \int_0^t \pr{\frac \tau t}^{\frac{dn -1+a}2} \mathcal{H}(\tau)d\tau }^{\frac 1 2}  } 
  \\
&\ge - \frac{C_0 }{t \mathcal{H}_n\pr{t}} \brac{\frac{C_0 \mathcal{I}(t)}{ 2\mathcal{H}_n\pr{t}} \int_0^t \pr{\frac \tau t}^{\frac{dn -1+a}2} \mathcal{H}(\tau)d\tau  
+ \frac 1 2 \int_0^t \pr{\frac \tau t}^{\frac{dn -1+a}2} \mathcal{I}(\tau) d\tau 
+ \int_0^t \pr{\frac \tau t}^{\frac{dn -1+a}2} \mathcal{J}(\tau) d\tau } \\
&=: - E_n(t).
\end{aligned}
\end{equation}

To show that $\mathcal{L}$ is monotone non-decreasing, it suffices to show that given any $t_0 \in P$, there exists $\de \in \pr{0, \frac {t_0} 2}$ so that $E_n$ converges uniformly to $0$ on $\brac{t_0 - \de, t_0}$.
Indeed, since $\frac{d}{dt}L_n(t) \ge - E_n(t)$, then for any $t \in \brac{t_0 - \de, t_0}$, it holds that
\[
L_n(t_0)-L_n(t)
\ge -\int_{t}^{t_0}E_n(\tau) d\tau.
\]
Since $L_n(t)$ converges pointwise to $\mathcal{L}(t)$, then
\[
\mathcal{L}(t_0) - \mathcal{L}(t)
= \lim_{n \to \iny} \brac{L_n(t_0)-L_n(t)} 
\ge - \lim_{n \to \iny} \int_{t}^{t_0}E_n(\tau)d\tau.
\]
Assuming the local uniform convergence of $E_n$ to 0 on $\brac{t_0 - \de, t_0} \supset \brac{t, t_0}$, we see that
$$\lim_{n \to \iny} \int_{t}^{t_0}E_n(\tau) d\tau 
=  \int_{t}^{t_0} \lim_{n \to \iny} E_n(\tau)d\tau = 0$$
and we may conclude that $\mathcal{L}(t_0)-\mathcal{L}(t) \ge 0$, as desired.

It remains to justify the local uniform convergence of $E_n(t)$ to $0$, as described above. For $t_0 \in P$, with $\disp H = \frac 1 8 \int_{\R^{d+1}_+} \abs{U\pr{X, t_0}}^2 x_0^a \, dX> 0$, an application of Lemma \ref{HnLowerBoundLemma} shows that there exists $N \in \N$ and $\de \in (0, \frac{t_0} 2]$ so that whenever $n \ge N$, \eqref{HnLowerBound} holds. 

Since $U \in \mathfrak{A}\pr{\R^{d+1}_+ \times \pr{0, T}}$, then there exists $p > 1$ so that \eqref{AlmgrenClass2} holds.
In particular, an application of Lemma \ref{weightedLpBound} shows that there exists $C_1 = C_1(a, d, \al, p, N)$ so that
\begin{equation*}
\sup_{t \in \brac{t_0 - \de, t_0}} \int_0^t \pr{\frac \tau t}^{\frac{dn -1+a}2} \mathcal{H}(\tau) d\tau 
\le 2^\al C_1 t_0^{-\al} \pr{ \frac{t_0}{n}}^{1 - \frac 1 p} \norm{t^\al \mathcal{H}}_{L^p\pr{[0, t_0]}}.
\end{equation*}
Analogous statements hold with $\mathcal{I}$ and $\mathcal{J}$ in place of $\mathcal{H}$.
Returning to the expression \eqref{FnBound}, assuming that $\de \le \eps$ from \eqref{AlmgrenClass1}, we see that whenever $n \ge N$, it holds that
\begin{align*}
\sup_{t \in \brac{t_0 - \de, t_0}} E_n(t)
&\le \frac{2^{1+\al} C_0 C_1 e^{\frac{d N \ln 2}{3} }}{ \mathcal{C} H  } t_0^{\frac{d-1+a-2\al} 2} \pr{ \frac{t_0}{n}}^{1 - \frac 1 p} \times \\
&\times \brac{\frac{C_0 e^{\frac{d N \ln 2}{3} }}{ 2\mathcal{C} H  } t_0^{\frac{d+1+a} 2} \norm{\mathcal{I}}_{L^\iny\pr{[t_0 - \de, t_0]}} \norm{t^\al \mathcal{H}}_{L^p\pr{[0, t_0]}}  
+ \frac 1 2 \norm{t^\al \mathcal{I}}_{L^p\pr{[0, t_0]}} 
+ \norm{t^\al \mathcal{J}}_{L^p\pr{[0, t_0]}} }.
\end{align*}
The required version of uniform convergence follows from this bound and completes the proof.
\end{proof}

\section{Weiss Frequency Functions}
\label{S:Weiss}

The central results of this section are a parabolic Weiss monotonicity formula (Theorem \ref{T:parabolicweiss}) and a new epiperimetric inequality for weakly $a$-harmonic functions (Theorem \ref{T:epi}). 
Both of these results have been inspired by those originally obtained by Weiss in \cite{Wei98} and \cite{Wei99} in the context of the classical obstacle problem. 

\subsection{Elliptic results}
\label{SS:ellipticWeiss}

The goal of this section is to establish the elliptic tools that will be used to prove the corresponding parabolic monotonicity result.
In most other applications of our elliptic-to-parabolic technique, we prove nonhomogeneous versions of the elliptic result, then carefully take limits to arrive at the corresponding parabolic theorem.
In this setting, the approach is more subtle and we need a number of new ideas and results, each outlined below.

\subsubsection{Weiss frequency functions in the elliptic setting}
\label{SSS:weisselliptic}

Here we state and prove a monotonicity formula for the elliptic Weiss frequency function.
Notice that this result holds for any function in the appropriate function class, not just those functions that satisfy a certain elliptic PDE.

\begin{thm}[Elliptic Weiss]
\label{monoWeissThm}
For $V \in H^1_{loc}\pr{\BB{D}^{N+1}_{R}, y_0^a}$ and $h \in (0, \iny)$, define the Weiss frequency functional as
\begin{equation}
\label{ellipWeiss}
\begin{aligned}
W_h(r;V)
&=\frac{1}{r^A} \int_{\BB{D}_r}  |\nabla V(Y)|^2 y_0^{a} \, dY
-\frac{h}{r^{A+1} } \int_{\BB{H}_r}  \abs{V(Y)}^2 y_0^{a} \, d\eta
= \frac{1}{r^A} D(r; V)
-\frac{h}{r^{A+1} } H(r; V),
\end{aligned}
\end{equation}
where we set $A = N + a + 2h - 1$.
Let $\bar V$ denote the $h$-homogeneous extension of $V$ on $\BB{H}^N_r$; that is, 
\begin{equation}
\label{hhomogExt}
\bar V(Y) 
= \pr{\frac{\abs{Y}}r}^h V\pr{\frac{r Y}{\abs{Y}}}.  
\end{equation}
For a.e. $r \in (0, R)$, it holds that
\begin{equation}
\label{Weisselliptic}
\begin{aligned}
\frac{d}{dr} W_h(r; V)
&= \frac{A}{r} \brac{W_h(r;\bar V) - W_h(r;V)}
+ \frac{1}{r^{A+2}} \int_{\BB{H}_r}  \brac{Y \cdot \gr V(Y) - h V(Y)}^2 y_0^{a} \, d\eta.
\end{aligned}
\end{equation}
Moreover, if $V$ satisfies elliptic boundary conditions, then \eqref{Weisselliptic} holds for every $r \in (0, R)$.
\end{thm}

\begin{proof}
Observe that
$$W_h(r; V) =\frac{1}{r^{A}}D(r; V)
-\frac{h}{r^{A+1}}H(r; V),$$
where $D\pr{r; V}$ and $H\pr{r; V}$ are as defined in \eqref{HDLDefn}. 
An application of the coarea formula and the Lebesgue differentiation theorem shows that for a.e. $r \in (0, R)$, $\disp D'(r) = \int_{\BB{H}_r} \abs{\gr V}^2 y_0^a$.

Using \eqref{H'}, see Remark \ref{noBCRemark}, it follows that for a.e. $r \in (0, R)$,
\begin{equation}
\label{WeissDeriv}
\begin{aligned}
W_h'(r;V)
&=-\frac{A}{r^{A+1}} D(r;V)
+\frac{D'(r;V)}{r^A}
+(A+1) \frac{h}{r^{A+2}} H(r;V)
-\frac{h}{r^{A+1}} H'(r;V)\\
&=-\frac{A}{r} W_h(r;V)
+\frac{h}{r^{A+2}} H(r;V)
+\frac{D'(r;V)}{r^A}
-\frac{h}{r^{A+1}} H'(r;V)\\
&=-\frac{A}{r} W_h(r;V)
+\frac{h}{r^{A+2}} H(r;V)
+\frac{1}{r^A} \int_{\BB{H}_r}  |\nabla V|^2 y_0^{a}  
-\frac{h(N+a)}{r^{A+2}} H(r;V)
-\frac{2h}{r^{A+2}} \int_{\BB{H}_r}  V \pr{Y \cdot \gr V} y_0^{a}  
\\
&= - \frac{A}{r} W_h(r;V)
+ \frac{1}{r^{A+2}} \int_{\BB{H}_r} \brac{- h(N-1+a)V^2  -2h V \pr{Y \cdot \gr V}  + r^2 \abs{\gr_{\tan} V}^2 + \pr{Y \cdot \gr V}^2} y_0^{a} 
\\
&= - \frac{A}{r} W_h(r;V)
+ \frac{1}{r^{A+2}} \int_{\BB{H}_r}  \pr{Y \cdot \gr V - h V}^2 y_0^{a}
 \\ 
&+ \frac{1}{r^{A+2}} \int_{\BB{H}_r}  \brac{r^2 \abs{\gr_{\tan} V}^2 - h(N - 1+a + h) \abs{V}^2} y_0^{a}, 
\end{aligned}
\end{equation}
where we have used that $\abs{\gr }^2 = \abs{\gr_{\tan} }^2 + \pr{\frac{\del }{\del \nu}}^2$, then completed the square.
With $V^{(r)}(Y) := \frac{V(rY)}{r^h}$, it holds that $\nabla V^{(r)}(Y) = \frac{\nabla V(rY)}{r^{h-1}}$, and then with $r Z = Y$ we get
\begin{align*}
& \frac{1}{r^{A+2}} \int_{\BB{H}_r}  \pr{r^2 \abs{\gr_{\tan} V(Y)}^2 - h(N -1+a + h) \abs{V(Y)}^2} \, y_0^{a} \, d\eta(Y) \\
&= \frac{1}{r^{A+2}} \int_{\BB{H}_1}  \pr{r^2 \abs{\gr_{\tan}V(rZ )}^2 - h(N -1+a + h) \abs{V(rZ)}^2} \pr{r z_0}^a d\eta(r Z) \\
&= \frac{1}{r} \int_{\BB{H}_1}  \pr{ \abs{\gr_{\tan} V^{(r)}(Z)}^2 - h(N -1+a + h) \abs{V^{(r)}(Z)}^2} \, z_0^a \, d\eta(Z).
\end{align*}
If $\bar V$ denotes the $h$-homogeneous extension of $V^{(r)}$ on $\BB{H}_1$, then $\disp \bar V(Y) = \abs{Y}^h V^{(r)}\pr{\frac{Y}{\abs{Y}}} = \pr{\frac{\abs{Y}}r}^h V\pr{\frac{r Y}{\abs{Y}}}$, showing that $\bar V$ is also the $h$-homogeneous extension of $V$ on $\BB{H}_r$.
It follows that
\begin{align*}
\int_{\BB{H}_1}  \pr{\abs{\gr_{\tan} V^{(r)}}^2 - h\pr{N-1+a+h} \abs{V^{(r)}}^2} z_0^a 
&= \int_{\BB{H}_1}  \pr{\abs{\gr_{\tan}\bar V}^2 - h\pr{N-1+a+h} \abs{\bar V}^2}z_0^a 
\\
&= \int_{\BB{H}_1} \pr{|\gr \bar V|^2 - |Z \cdot \gr \bar V|^2 - h\pr{N-1+a+h} \abs{\bar V}^2} z_0^a 
\\
&= \int_{\BB{H}_1}  |\gr \bar V|^2z_0^a 
- A h \int_{\BB{H}_1} \abs{\bar V}^2z_0^a,  
\end{align*}
since the $h$-homogeneity of $\bar V$ implies that $Z \cdot \nabla \bar V(Z) = h \bar V(Z)$ for any $Z \in \BB{H}_1$. 
Another application of $h$-homogeneity shows that
\begin{align*}
\int_{\BB{D}_r}  |\nabla \bar V(Y)|^2 \, y_0^a \, dY
&= \int_0^r \int_{\BB{H}_1} |\nabla \bar V(t Z)|^2 \pr{t z_0}^a  d\eta(tZ) dt
= \int_0^r \int_{\BB{H}_1}  |t^{h-1} \nabla \bar V(Z)|^2 t^N t^a z_0^a \, d\eta(Z) dt
\\
&= \int_0^r t^{N+a + 2h - 2} dt \int_{\BB{H}_1}  |\nabla \bar V(Z)|^2 z_0^a d\eta(Z) 
= \frac{r^A}{A} \int_{\BB{H}_1}  |\nabla \bar V|^2 z_0^a. 
\end{align*}
From the definition of $\bar V$, we see that 
\begin{align*}
\int_{\BB{H}_1}   \abs{\bar V(Z)}^2 z_0^a \, d\eta(Z) 
&= \int_{\BB{H}_1}  \abs{\pr{\frac{\abs{Z}}r}^h V\pr{\frac{r Z}{\abs{Z}}}}^2  z_0^a \, d\eta(Z) 
= r^{-\pr{a + 2h}} \int_{\BB{H}_1}   \abs{ V\pr{r Z}}^2 \pr{r z_0}^a \, d\eta(Z) \\
&= r^{-\pr{N + a + 2h}} \int_{\BB{H}_r}   \abs{V\pr{Y}}^2 y_0^a \, d\eta(Y) 
= r^{-\pr{A+1}} \int_{\BB{H}_r}  \abs{\bar V}^2 y_0^a.
\end{align*}
Combining the previous four equations shows that
\begin{align*}
\frac{1}{r^{A+2}} \int_{\BB{H}_r}   \pr{r^2 \abs{\gr_{\tan} V}^2 - h(N -1+a + h) \abs{V}^2} y_0^a 
&= \frac 1 r \int_{\BB{H}_1}  |\nabla \bar V|^2 z_0^a 
- \frac{A h}{r} \int_{\BB{H}_1}  \abs{\bar V}^2 z_0^a 
= \frac A r W_h\pr{r; \bar V}.
\end{align*}
Substituting this expression into \eqref{WeissDeriv} leads to the conclusion.
If $V$ satisfies the elliptic boundary conditions, then \eqref{D'0} holds for every $r \in (0, R)$ and we may improve the conclusion accordingly.
\end{proof}

\subsubsection{Tools from elliptic theory}
\label{SSS:EllipTools}

In this subsection, we introduce some general theory that we use below.
We need precise information about the constant in the weighted Poincar\'{e} inequality on balls, so we prove the estimate here.

\begin{lem}[Weighted Poincar\'{e} inequality]
\label{L:poincare}
If $V \in C^1_0(\BB{B}_R^{N+1})$, then
\begin{align*}
\int_{\BB{B}_R}  \abs{V(Y)}^2 \abs{y_0}^a dY \le \frac{R^2}{N + a + 1} \int_{\BB{B}_R}  \abs{\gr V(Y)}^2 \abs{y_0}^a dY.
\end{align*}
\end{lem}

\begin{proof}
First, we change to polar coordinates, $Y = r \om$, to get
\begin{align*}
\int_{\BB{B}_R}  \abs{V(Y)}^2  \abs{y_0}^a dY 
&= \int_{\BB{S}_1} \int_0^R  \abs{V(r \om)}^2 \pr{r\abs{\om_0}}^a r^{N} dr d\si
= \int_{\BB{S}_1} \int_0^R \abs{V(r \om)}^2 r^{N+a} dr \abs{\om_0}^a d\si,
\end{align*}
where $d \si$ denotes the canonical surface measure on $\BB{S}_1$.
Since $V \in C^1_0(\mathbb{B}_R)$, then for any $Y = r \om \in \mathbb{B}_R$, it holds that 
\begin{align*}
V(r \om) 
&= V(r \om) - V\pr{R\om}
= - \int_r^R \gr V\pr{t \om} \cdot \om \, dt
\end{align*}
and then, via H\"older's inequality, we get
\begin{align*}
\abs{V(r \om) }^2
&= \abs{\int_r^R \gr V\pr{t \om} \cdot \om \, dt}^2
\le \pr{\int_r^R \abs{\gr V\pr{t \om}} dt}^2
\le R \pr{\int_r^R \abs{\gr V\pr{t \om}}^2 dt}.
\end{align*}
Returning to the expression above, this bound combined with an application of Fubini's theorem shows that
\begin{align*}
\int_0^R \abs{V(r \om)}^2 r^{N+a} dr
&\le \int_0^R R \pr{\int_r^R \abs{\gr V\pr{t \om}}^2 dt} r^{N+a} dr
= R \int_0^R \int_r^R \abs{\gr V\pr{t \om}}^2 t^{N+a} \pr{\frac r t}^{N+a} dt dr \\
&= R \int_0^R \abs{\gr V\pr{t \om}}^2 t^{N+a} \pr{\int_0^t \pr{\frac r t}^{N+a} dr} dt
= \frac{R}{N+a+1} \int_0^R \abs{\gr V\pr{t \om}}^2 t^{N+1+a} dt \\
&\le \frac{R^2}{N+a+1} \int_0^R \abs{\gr V\pr{t \om}}^2 t^{N+a} dt.
\end{align*}
Integrating over $\BB{S}_1$ and changing back from polar coordinates leads to the conclusion.
\end{proof}

\begin{defn}[Weak solutions]
    Let $\Om \su \R^{N+1}$ be open and bounded.
    Given $K \in L^2\pr{\Om, \abs{y_0}^a}$, we say that $V \in H^1\pr{\Om, \abs{y_0}^a}$ is a \textbf{weak solution} to $- \di \pr{\abs{y_0}^a \gr V} = K \abs{y_0}^a$ if for every $\vp \in C^\iny_0\pr{\Om}$, it holds that
    \begin{equation}
    \label{weakSolEqn}
        \int_{\Om} \gr V(Y) \cdot \gr \vp(Y) \abs{y_0}^a dY = \int_{\Om} K(Y) \vp\pr{Y} \abs{y_0}^a dY.
    \end{equation}
    If $K = 0$, then we say that $V$ is \textbf{weakly $a$-harmonic}.
    Given $W \in H^1\pr{\Om, \abs{y_0}^a}$, we say that $V \in H^1\pr{\Om, \abs{y_0}^a}$ is a \textbf{weak solution to the boundary value problem}
    \begin{equation*}
    \left\{
    \begin{array}{rl} 
    - \di\pr{\abs{y_0}^a \gr V} = K \abs{y_0}^a & \text{ in } \Om \\
    V = W \qquad & \text{ on } \del\Om
    \end{array}
    \right.
    \end{equation*}
    if \eqref{weakSolEqn} holds and $V - W \in H^1_0\pr{\Om, \abs{y_0}^a}$.
\end{defn}

\begin{rem}
    We point out that if $\Om \su \R^{N+1}_+$, then weak solutions are in fact classical solutions.
\end{rem}

\begin{lem}[Existence of weak solutions]
\label{weakSolExLemm}
    For every $W \in H^1\pr{\Om, \abs{y_0}^a}$, there exists a unique $V \in H^1\pr{\Om, \abs{y_0}^a}$ that weakly solves
    \begin{equation}
    \label{EBVP}
    \left\{
    \begin{array}{rl} 
    - \di\pr{\abs{y_0}^a \gr V} = 0 & \text{ in } \Om \\
    V = W & \text{ on } \del\Om.
    \end{array}
    \right.
    \end{equation}
\end{lem}

A more general version of this result appears in \cite{KS80}*{Theorem 5.4}, see also \cite{FKS82}*{Theorem 2.2}.

\begin{lem}[Even solutions lemma]
\label{evenSolLemma}
For $W \in H^1\pr{\BB{B}_r, \abs{y_0}^a}$, let $V \in H^1\pr{\BB{B}_r, \abs{y_0}^a}$ be the weak solution to \eqref{EBVP}.
If $W$ is even with respect to $y_0$, then $V$ is even with respect to $y_0$ on $\BB{B}_r$.
\end{lem}

\begin{proof}
For any function $F: \BB{B}_r \to \R$, define its reflection $\overline{F}: \BB{B}_r \to \R$ as $\overline{F}\pr{y_0, y} = F\pr{- y_0, y}$ and note that $\gr \overline{F}\pr{y_0, y} = \pr{-\del_{0} F\pr{-y_0, y}, \gr_y F \pr{-y_0, y}}$.

We first show that $\overline{V}$ is a weak solution to $-\di\pr{\abs{y_0}^a \gr V} = 0$ in $\BB{B}_r$.
It suffices to show that for arbitrary $\zeta \in C^\iny_0\pr{\BB{B}_r}$, we have
$$\int_{\mathbb{B}_R} \gr \overline{V}(Y) \cdot \gr \zeta(Y) \abs{y_0}^a dY = 0.$$
Observe that
\begin{align*}
\int_{\BB{B}_r} \gr \overline{V}(Y) \cdot \gr \zeta(Y) \abs{y_0}^a dY
&= \int_{\BB{B}_r} \brac{-\del_{0} V\pr{-y_0, y}, \gr_y V \pr{-y_0, y}} \cdot \pr{\del_{0} \zeta\pr{y_0, y}, \gr_y \zeta \pr{y_0, y}} \abs{y_0}^a dy_0 dy \\
&= - \int_{\BB{B}_r} \brac{\del_{0} V\pr{z_0, y}, \gr_y V \pr{z_0, y}} \cdot \pr{-\del_{0} \zeta\pr{-z_0, y}, \gr_y \zeta \pr{-z_0, y}} \abs{z_0}^a dz_0 dy \\
&= - \int_{\BB{B}_r} \gr V(Y) \cdot \gr \overline{\zeta}(Y) \abs{y_0}^a dY.
\end{align*}
Since $\zeta \in C^\iny_0\pr{\BB{B}_r}$ if and only if $\overline{\zeta} \in C^\iny_0\pr{\BB{B}_r}$, then we deduce that $\disp \int_{\BB{B}_r} \gr \overline{V}(Y) \cdot \gr \zeta(Y) \abs{y_0}^a dY= 0$ as well, showing that $\overline{V}$ is a weak solution.

As $W$ is even with respect to $y_0$, then because $W = V$ on $\del\BB{B}_r$ in the sense of trace, then we deduce that $W = \overline{V}$ on $\del \BB{B}_r$ as well.
However, we know by Lemma \ref{weakSolExLemm} that weak solutions are unique, so we may conclude that $V = \overline{V}$.
That is, $V$ is itself an even function.
\end{proof}

Given a function $V$ defined on $\BB{D}_r$, we let $V^e$ denote its even extension to $\BB{B}_r$.
Unless we have some kind of continuity assumption, $V^e$ may not be defined on $\BB{B}_r \cap \set{y_0 = 0}$. 
It turns out that the even extension of a function in the Sobolev space belongs to the extended Sobolev space.

\begin{lem}[Even extension spaces lemma]
\label{evenExtLemma}
    If $V \in H^1\pr{\mathbb{D}_r, y_0^a}$, then $V^e \in H^1\pr{\mathbb{B}_r, \abs{y_0}^a}$.
\end{lem}

\begin{proof}
Since $V^e \in H^1\pr{\mathbb{B}_r \setminus \set{y_0 = 0}, \abs{y_0}^a}$ and $\set{y_0 = 0}$ is relatively closed and of dimension $N$, then an application of \cite[Theorem 2.6]{Kil94} shows that $V^e \in H^1\pr{\mathbb{B}_r, \abs{y_0}^a}$.
\end{proof}

If we impose boundary conditions of type I, then the even extension of a solution is a weak solution.

\begin{lem}[Even extension solutions, cf. Lemma 4.1 in \cite{CS07}]
\label{evenLem}
    Let $K \in L^2\pr{\BB{D}_r, y_0^a}$.
    If $V \in H^1\pr{\BB{D}_r, y_0^a}$ satisfies the elliptic boundary conditions of type I and is a solution to 
    \begin{equation*}
        - \di\pr{y_0^a \gr V} = K y_0^a   \quad \text{ in } \; \BB{D}_r,
    \end{equation*}
    then $V^e \in H^1\pr{\BB{B}_r, \abs{y_0}^a}$ is a weak solution to 
    \begin{equation*}
        - \di\pr{\abs{y_0}^a \gr V^e} = K^e \abs{y_0}^a  \quad \text{ in } \; \BB{B}_r,
    \end{equation*}
    where $K^e$ denotes the even extension of $K$ to $\Om$.
\end{lem}

\begin{proof}
Since $- \di\pr{y_0^a \gr V} = K y_0^a$ in $\BB{D}_r$, then $- \di\pr{\abs{y_0}^a \gr V^e} = K^e \abs{y_0}^a$ in $\BB{B}_r \cap \set{y_0 < 0}$ as well, in the classical sense.
Given $\zeta \in C^\iny_0\pr{\BB{B}_r}$, for any $\eps > 0$, we get
\begin{align*}
    \int_{\BB{B}_r} \gr V^e(Y) \cdot \gr \zeta(Y) \abs{y_0}^a dY
    &= \int_{\BB{B}_r \setminus \set{\abs{y_0} < \eps}} \gr V^e(Y) \cdot \gr \zeta(Y) \abs{y_0}^a dY
     + \int_{\BB{B}_r \cap \set{\abs{y_0} < \eps}} \gr V^e(Y) \cdot \gr \zeta(Y) \abs{y_0}^a dY \\
    &= \int_{\BB{B}_r \setminus \set{\abs{y_0} < \eps}} K^e(Y) \zeta(Y) \abs{y_0}^a dY
    + \int_{\BB{B}_r \cap \set{\abs{y_0} = \eps}} \abs{y_0}^a \gr V^e(Y) \cdot \nu \zeta(Y) dy  \\
    &+ \int_{\BB{B}_r \cap \set{\abs{y_0} < \eps}} \gr V^e(Y) \cdot \gr \zeta(Y) \abs{y_0}^a dY.
\end{align*}
Since $K \in L^{2}\pr{\BB{D}_r, y_0^a dY}$ and $\zeta$ is bounded, then
\begin{align*}
    \lim_{\eps \to 0^+} \int_{\BB{B}_r \setminus \set{\abs{y_0} < \eps}} V^e(Y) K^e(Y) \zeta(Y) \abs{y_0}^a dY
    &= \int_{\BB{B}_r} V^e(Y) K^e(Y) \zeta(Y) \abs{y_0}^a dY.
\end{align*}
Because $V$ satisfies the elliptic boundary conditions of type I, then by the dominated convergence theorem,
\begin{align*}
    \lim_{\eps \to 0^+} \int_{\BB{B}_r \cap \set{\abs{y_0} = \eps}} \abs{y_0}^a \gr V^e(Y) \cdot \nu \zeta(Y)  dy
    = -\lim_{\eps \to 0^+} \int_{\BB{B}_r \cap \set{y_0 = \eps}} \eps^a \del_{y_0} V^e(\eps, y) \brac{\zeta(\eps, y) + \zeta(-\eps, y)}  dy
    = 0.
\end{align*}
Since  
\begin{align*}
    & \int_{\BB{B}_r \cap \set{\abs{y_0} < \eps}} \gr V^e(Y) \cdot \gr \zeta(Y) \abs{y_0}^a dY 
    \le \norm{\gr \zeta}_{L^\iny\pr{\BB{B}_R}} \pr{\int_{\BB{B}_r} \abs{y_0}^a dY}^{\frac 1 2} \pr{\int_{\BB{B}_r \cap \set{\abs{y_0} < \eps}} \abs{\gr V^e(Y)}^2 \abs{y_0}^a dY}^{\frac 1 2},
\end{align*}
and $V^e \in H^1\pr{\BB{B}_r, \abs{y_0}^a}$, then the last term also vanishes in the limit and we reach the conclusion.
\end{proof}

\subsubsection{Epiperimetric inequalities for weakly $a$-harmonic functions}
\label{SSS:WeakaEpi}

If $V$ is a function defined on a ball $\BB{B}_R$, then for any $r \in (0, R)$, we write $W_h^e\pr{r; V}$ to denote the associated Weiss functional, as defined in \eqref{ellipWeiss}, except we now replace $y_0^a$, $d\eta$, $\BB{D}_r^{N+1}$, and $\BB{H}^N_r$ with $\abs{y_0}^a$, $d \si$, $\BB{B}^{N+1}_r$, and $\BB{S}^N_r$, respectively.
That is,
\begin{equation}
\label{ellipWeissFull}
\begin{aligned}
W^e_h(r;V)
&=\frac{1}{r^{N+a+2h-1}} \int_{\BB{B}_r}  |\nabla V(Y)|^2 \abs{y_0}^{a} \, dY
-\frac{h}{r^{N+a+2h} } \int_{\BB{S}_r}  \abs{V(Y)}^2 \abs{y_0}^{a} \, d\si(Y).
\end{aligned}
\end{equation}

To establish a ``full-ball" version of the epiperimetric inequality, we use polar coordinates and an orthonormal system for $L^2(\BB{S}_1, \abs{\om_0}^a)$.

We introduce polar coordinates on $\R^{N+1}$ by setting $Y = r \om$ where $\om \in \BB{S}^N$ and $r > 0$.
That is, for each $i = 0, \ldots, N$, $y_i = r \om_i$.
The vector fields $\Om_i$ act on $\BB{S}^N$ and satisfy $\disp \sum_{i=0}^N \om_i \Om_i = 0$ and $\disp \sum_{i=0}^N \Om_i \om_i = N.$
For each $i = 0, \ldots, N$, $\disp \frac{\del }{\del y_i} = \om_i \del_r + \frac 1 r \Om_i$, so it follows that
\begin{align}
\label{polarGradient}
    \gr V \cdot \gr W
    &= \pr{\del_r V} \pr{\del_r W }
+ \frac 1 {r^2} \sum_{i=0}^N \pr{\Om_i V} \pr{\Om_i W}.
\end{align}

Let $\mathcal{P}$ denote the set of all weakly $a$-harmonic homogeneous polynomials on $\BB{B}_1$ that are even with respect to $y_0$.
If $P \in \mathcal{P}$, then $P = P_j$ is a weakly $a$-harmonic polynomial of order $h_j \in \N$ that is even with respect to $y_0$ and we may write $P_j(Y) = r^{h_j} \phi_j\pr{\om}$.
Moreover, as shown in \cite{JP20}, $\set{\phi_j}_{j=1}^\iny$ forms an orthonormal system in $L^2\pr{\BB{S}_1, \abs{\om_0}^a}$.
The following result allows us to prove our epiperimetric inequality below.

\begin{lem}[Properties of the orthonormal system]
For $\set{\phi_j}_{j=1}^\iny$ as introduced above, it holds that
\begin{equation}
\label{sphericalIntegral}
\sum_{i=0}^N \int_{\BB{S}_1} \Om_i \phi_j  \, \Om_i \phi_k \abs{\om_0}^a d\si
= \begin{cases}
    h_j \pr{N + a + h_j - 1} & \text{ if } j = k \\
    0 & \text{ otherwise }.
\end{cases}
\end{equation}
\end{lem}

\begin{proof}
Choose $j, k \in \N$ and $h \in \pr{0, \iny} \setminus \set{h_j, h_k}$. 
If $P_j, P_k \in \mathcal{P}$, then $P_j(r, \om) = r^{h_j}\phi_j\pr{\om}$, $P_k(r, \om) = r^{h_k}\phi_k\pr{\om}$, and $\overline{P}_k = r^h \phi_k\pr{\om}$ is the $h$-homogeneous extension of $P_k$ from $\BB{S}_1$ to $\BB{B}_1$.
From \eqref{polarGradient}, we see that
\begin{align*}
\gr P_j \cdot \gr P_k
&= r^{h_j + h_k - 2}\pr{ h_j h_k \phi_j \phi_k + \sum_{i=0}^N \Om_i \phi_j \Om_i \phi_k} \\
\gr P_j \cdot \gr \overline{P}_k
&= r^{h_j + h - 2}\pr{ h_j h \phi_j \phi_k + \sum_{i=0}^N \Om_i \phi_j \Om_i \phi_k} .
\end{align*}
Therefore,
\begin{align*}
\int_{\BB{B}_1} \gr P_j \cdot \gr P_k \abs{y_0}^a dY
&= \int_{\BB{S}_1} \int_0^1 r^{h_j + h_k - 2}\pr{ h_j h_k \phi_j \phi_k + \sum_{i=0}^N \Om_i \phi_j \Om_i \phi_k} \abs{r \om_0}^a r^N dr d\si \\
&= \pr{\int_0^1 r^{N + a + h_j + h_k -2} dr} \brac{\int_{\BB{S}_1} \pr{ h_j h_k \phi_j \phi_k + \sum_{i=0}^N \Om_i \phi_j  \, \Om_i \phi_k} \abs{\om_0}^a  d\si} \\
&= \frac 1 {N + a + h_j + h_k - 1} \pr{ h_j h_k \de_{jk} 
+ \sum_{i=0}^N \int_{\BB{S}_1}  \Om_i \phi_j  \, \Om_i \phi_k \abs{\om_0}^a d\si} \\
\int_{\BB{B}_1} \gr P_j(Y) \cdot \gr \overline{P}_k(Y) \abs{y_0}^a dY
&= \frac 1 {N + a + h_j + h -1} \pr{ h_j h \de_{jk}
+ \sum_{i=0}^N \int_{\BB{S}_1}  \Om_i \phi_j  \, \Om_i \phi_k \abs{\om_0}^a d\si},
\end{align*}
where we have used that $\set{\phi_j}_{j = 1}^\iny$ is orthonormal in $L^2\pr{\BB{S}_1, \abs{\om_0}^a}$.
Since $P_j$ is weakly $a$-harmonic and $P_k - \overline{P}_k \in H^1_0\pr{\BB{B}_1, \abs{y_0}^a}$, then
\begin{align*}
0
&= \int_{\BB{B}_1}\gr P_j \cdot \pr{\gr P_k - \gr \bar{P}_k} \abs{y_0}^a dY \\
&=  \frac {h_k - h} {\pr{N + a + h_j + h_k - 1}\pr{N + a + h_j + h -1}} \brac{h_j \de_{jk}\pr{N + a + h_j  -1} - \sum_{i=0}^N \int_{\BB{S}_1}  \Om_i \phi_j  \, \Om_i \phi_k \abs{\om_0}^a d\si},
\end{align*}
where we have substituted the previous pair of equations and simplified.
Since $h \ne h_k$, then the bracketed term must vanish and the conclusion follows.
\end{proof}

We now prove an epiperimetric inequality for weakly $a$-harmonic functions 
that are even with respect to $y_0$. 
Our proof is inspired by that of the epiperimetric inequality for harmonic functions in \cite{BET20}.
We rely on some of the properties of weakly $a$-harmonic functions that are shown in \cite{JP20}.

\begin{prop}[Epiperimetric inequalities for weakly $a$-harmonic functions]
\label{T:epi}
Let $V \in H^1\pr{\BB{B}_R, \abs{y_0}^a}$ be a weakly $a$-harmonic function that is even with respect to $y_0$. 
For any $r \in (0, R)$ and any $h \in \pr{0, \iny}$, let $\bar V$ denote the $h$-homogeneous extension of $V$ from $\BB{S}_r^N$ to $\BB{B}_r^{N+1}$, as defined in \eqref{hhomogExt}.
Define 
\begin{equation}
\label{kappa0Def}
\kappa_0 =\frac{1 + \lfloor h \rfloor - h}{N + a +h + \lfloor h \rfloor}.
\end{equation}
For any $\kappa \in \pb{0, \kappa_0}$, it holds that
\begin{align}
&W_h^e\pr{r; V} \le \pr{1 - \kappa} W_h^e\pr{r; \bar V} 
\label{epiIneqFull} \\
&W_h\pr{r; V} \le \pr{1 - \kappa} W_h\pr{r; \bar V},
\label{epiIneqweakHalf}
\end{align}
where $W_h^e$ and $W_h$ are defined in \eqref{ellipWeissFull} and \eqref{ellipWeiss}, respectively.
\end{prop}

\begin{proof}
Without loss of generality, assume that $r = 1$.
Since $V \in H^1\pr{\BB{B}_R, \abs{y_0}^2}$ is weakly $a$-harmonic and even with respect to $y_0$, then an application of \cite{JP20}*{Theorem B.1} shows that we may write 
$$V(Y) = \sum_{j=1}^\iny c_j P_j\pr{Y},$$
where each $P_j \in \mathcal{P}$ is an $a$-harmonic homogeneous polynomial of degree $h_j$ that is even in $y_0$.
Moreover, the convergence of the series is locally uniform so $V$ is real analytic in $\BB{B}_1$.
In polar coordinates, $P_j(Y) = r^{h_j} \phi_j(\om)$ where $\set{\phi_j}_{j=1}^\iny$ forms an orthonormal system in $L^2\pr{\BB{S}_1, \abs{\om_0}^a}$.
Therefore, in $\overline{\BB{B}}_1$, it holds that 
$$V(r, \om) = \sum_{j=1}^\iny c_j r^{h_j} \phi_j\pr{\om}.$$

Using \eqref{polarGradient}, we see that 
\begin{align*}
\abs{\gr V}^2
&= \pr{\del_r V}^2 
+ \frac 1 {r^2} \sum_{i=0}^N \pr{\Om_i V}^2
= \pr{\sum_{j=1}^\iny h_j c_j r^{h_j-1} \phi_j}^2 
+ \frac 1 {r^2} \sum_{i=0}^N \pr{ \sum_{j=1}^\iny c_j r^{h_j} \Om_i \phi_j}^2 \\
&= \sum_{j, k=1}^\iny c_j c_k  r^{h_j + h_k -2} \pr{ h_j h_k \phi_j \phi_k  
+ \sum_{i=0}^N \Om_i \phi_j  \, \Om_i \phi_k}.
\end{align*}
Therefore,
\begin{align*}
\int_{\BB{B}_1} \abs{\gr V(Y)}^2 &\abs{y_0}^a dY
= \int_{\BB{S}_1} \int_0^1  \sum_{j, k=1}^\iny c_j c_k  r^{h_j + h_k -2} \pr{ h_j h_k \phi_j \phi_k  
+ \sum_{i=0}^N \Om_i \phi_j  \, \Om_i \phi_k} \abs{r \om_0}^a r^N dr d\si \\
&= \sum_{j, k=1}^\iny c_j c_k \pr{\int_0^1   r^{N+a+h_j + h_k -2} dr } \brac{ h_j h_k \int_{\BB{S}_1}  \phi_j \phi_k \abs{\om_0}^a d\si
+ \sum_{i=0}^N \int_{\BB{S}_1} \Om_i \phi_j  \, \Om_i \phi_k \abs{\om_0}^a d\si} \\
&= \sum_{j=1}^\iny \frac{c_j^2 }{N + a + 2h_j - 1} \brac{h_j^2 + h_j \pr{N + a + h_j - 1}}
= \sum_{j=1}^\iny c_j^2 h_j,
\end{align*}
where we have used that $\set{\phi_j}_{j=1}^\iny$ is orthonormal system in $L^2\pr{\BB{S}_1, \abs{\om_0}^a}$ along with \eqref{sphericalIntegral}.
As
\begin{align}
\label{H1V}
\int_{\BB{S}_1} \abs{V}^2 \abs{y_0}^a  
&= \sum_{j=1}^\iny c_j^2  \int_{\BB{S}_1}  \phi_j^2 \, \abs{y_0}^a 
= \sum_{j=1}^\iny c_j^2
\end{align}
then
\begin{align}
\label{aharmonWeiss}
W_h^e(V, 1)
&= \int_{\BB{B}_1}  \abs{\gr V}^2 \abs{y_0}^a 
- h \int_{\BB{S}_1}  \abs{V}^2 \abs{y_0}^a
= \sum_{j=1}^\iny c_j^2 \pr{h_j - h}  .
\end{align}

Analogous arguments show that
\begin{align*}
\int_{\BB{B}_1} \abs{\gr \bar{V}(Y)}^2 \abs{y_0}^a dY
&= \int_{\BB{S}_1} \int_0^1  \sum_{j, k=1}^\iny c_j c_k  r^{2h -2} \pr{ h^2 \phi_j \phi_k   + \sum_{i=0}^N \Om_i \phi_j  \, \Om_i \phi_k} \abs{r \om_0}^a r^N dr d\si \\
&= \sum_{j=1}^\iny \frac{c_j^2 }{N + a + 2h - 1} \brac{h^2 + h_j \pr{N + a + h_j - 1}}.
\end{align*}
Since $V = \bar V$ on $\BB{S}_1$, we may combine these computations with the expression \eqref{H1V} to get
\begin{align}
W_h^e(\bar V, 1)
&= \int_{\BB{B}_1}  \abs{\gr \bar V}^2 \abs{y_0}^a dY
- h \int_{\BB{S}_1}  \abs{\bar V}^2 \abs{y_0}^a d\si
= \sum_{j=1}^\iny c_j^2 \pr{h_j - h} \brac{\frac{N + a +h_j + h -1}{N + a + 2h - 1}}.
\label{homogWeiss}
\end{align}
From \eqref{aharmonWeiss} and \eqref{homogWeiss}, we have
\begin{align*}
W_h^e(V, 1) \le \pr{1 - \kappa}W_h^e(\bar V, 1)
& \iff
\sum_{j=1}^\iny c_j^2 \pr{h_j - h}
\le \pr{1 - \kappa} \sum_{j=1}^\iny c_j^2 \pr{h_j - h} \brac{\frac{N + a +h_j + h -1}{N + a + 2h - 1}} \\
& \iff
\kappa \sum_{j=1}^\iny c_j^2 \pr{h_j - h} \pr{N + a +h_j + h -1}
\le \sum_{j=1}^\iny c_j^2 \pr{h_j - h}^2
\end{align*}
which holds if for all $j \in \N$,
\begin{equation}
\label{kappaCond}
\kappa \pr{j - h} \pr{N + a +j + h -1}
\le \pr{j - h}^2.
\end{equation}
If $j \le h$, then \eqref{kappaCond} holds for any $\kappa \ge 0$.
On the other hand, if $j > h$, then $j \ge \lfloor h \rfloor +1$, and \eqref{kappaCond} holds if and only if $\kappa \le \kappa_0$, as given in \eqref{kappa0Def} and the inequality \eqref{epiIneqFull} follows.

Since $V$ is even with respect to $y_0$, then $W_h^e\pr{r; V} = 2 W_h\pr{r; V}$ and $W_h^e\pr{r; \bar{V}} = 2 W_h\pr{r; \bar{V}}$, so \eqref{epiIneqweakHalf} follows from \eqref{epiIneqFull}.
\end{proof}

\subsection{Weiss frequency functions in the parabolic setting}
\label{SS:paraWeiss}

We use the results from Subsection \ref{SS:ellipticWeiss} to prove a Weiss monotonicity formula for solutions to degenerate parabolic equations, described below in Theorem \ref{T:parabolicweiss}.
In our other proofs of parabolic monotonicity formulas, we follow a common scheme: 
\begin{enumerate}
\item Use the parabolic solution function $U$ to generate a sequence of elliptic solution functions $\disp (V_n)_{n \in \N}$.
\item Apply the nonhomogeneous elliptic theorems to get statements about each $V_n$.
\item Use the bridge lemmas to transform these statements about $V_n$ to $n$-dependent statements about $U$.
\item Take the limit as $n \to \iny$.
\end{enumerate}
In this setting, we have an elliptic Weiss monotonicity result that applies to any function (not just solutions) and we have an epiperimetric inequality that applies to $a$-harmonic functions, i.e. solutions to \textit{homogeneous} degenerate elliptic equations.
Accordingly, we need to introduce auxiliary $a$-harmonic functions into our argument.
To prove Theorem \ref{T:parabolicweiss}, our new scheme is as follows:
\begin{enumerate}
\item Use the parabolic solution function $U$ to generate a sequence of elliptic solution functions $\disp (V_n)_{n \in \N}$.
\item Apply the elliptic Weiss monotonicity formula, Theorem \ref{monoWeissThm}, to get statements about each $V_n$.
\item For each $n \in \N$, use results from \S \ref{SSS:EllipTools} to define an $a$-harmonic functions $Q_n$ that is ``close'' to $V_n$.
\item Rewrite the elliptic Weiss monotonicity formula in terms of $V_n$, $Q_n$, and terms from the epiperimetric inequality, then apply the epiperimetric inequality, Proposition \ref{T:epi}.
\item Use more results from \S \ref{SSS:EllipTools} to carefully estimate the error term involving $V_n - Q_n$.
\item Use the bridge lemmas to transform these statements to $n$-dependent statements about $U$.
\item Take the limit as $n \to \iny$.
\end{enumerate}

We need a few a priori conditions on our solutions, described by the following function class.

\begin{defn}[Weiss function class]
We say that a function $U = U(X, t)$ defined on $\R^{d+1}_+ \times \pr{0, T}$ belongs to the function class $\mathfrak{W}\pr{\R^{d+1}_+ \times \pr{0, T}}$ if 
\begin{itemize}
    \item $U$ has moderate $a$-growth at infinity (see Definition \ref{aGrowth}),
\item for every $t_0 \in \pr{0, T}$, there exists $p > 1$, $\al \in \R$ so that
\begin{equation}
\label{WeissClass2}   
t^\al \mathcal{J} \in L^{p}\pr{\pr{0, t_0}},
\end{equation}
    \item for every $t_0 \in \pr{0, T}$, there exists $\eps \in \pr{0, t_0}$ so that
\begin{equation}
\label{WeissClass1}
\mathcal{H}, 
\mathcal{I} \in L^\iny\pr{\brac{t_0 - \eps, t_0}}.
\end{equation}
\end{itemize}
See \eqref{12functionals} in Definition \ref{functionalDefs} for the definitions of these functionals.    
\end{defn}

\begin{thm}[Parabolic Weiss Monotonicity]
\label{T:parabolicweiss}
For $U = U\pr{X,t} \in \mathfrak{W}\pr{\R^{d+1}_+ \times \pr{0, T}}$ and $h \in (0, \iny)$, define
\[
\mathcal{W}_{h}(t; U):=
\frac{2}{t^{h-1}}\int_{\R^{d+1}_+} |\nabla U(X, t)|^2\dGt
-\frac{h}{t^h}\int_{\R^{d+1}_+} U(X, t)^2 \dGt
= \frac{2}{t^{h-1}} \mathcal{D}(t)
-\frac{h}{t^h}\mathcal{H}(t).
\]
If $U$ satisfies the parabolic boundary conditions of type I as in Definition \ref{pBC} and is a solution to 
\begin{equation}
\label{pWeissPDE}
\LP_x U + \frac{a}{x_0} \del_{x_0} U + \del^2_{x_0} U + \del_t U = 0 \text{ in } \R^{d+1}_+ \times \pr{0, T},
\end{equation}
then for every $t \in \pr{0, T}$, it holds that
\begin{equation*}
\begin{aligned}
\frac{d}{dt}\mathcal{W}_{h}\pr{t; U}
&\ge \frac{1 + \lfloor h \rfloor - h}{2t}\mathcal{W}_{h}(t;U)
+\frac{1}{2t^{h+1}} \int_{\R^{d+1}_+} \abs{ X\cdot\nabla U(X, t) + 2t\del_t U(X,t) - h \, U(X,t)}^2 \dGt.
\end{aligned}
\end{equation*}
\end{thm}

\begin{proof}
With $U$ as given, for each $n \in \N$, define $V_n : \BB{D}^{dn+1}_{\sqrt{MT}} \to \R$ so that
$$V_n \pr{Y} = U\pr{F_n\pr{Y}} = U\pr{X, t}.$$
Since $U$ has moderate $a$-growth at infinity, then Lemma \ref{functionClassLemma} shows that each $V_n \in H^1(\mathbb{D}^{dn+1}_{\sqrt{Mt}}, y_0^a)$ for every $t \in (0, T)$.
In particular, we may apply Theorem \ref{monoWeissThm} to each $V_n$ on $\mathbb{D}^{dn+1}_{\sqrt{Mt}}$. 

We compute the elliptic Weiss functional from \eqref{ellipWeiss} for $V = V_n$ at $r = \sqrt{Mt}$ with $N=dn$ and $A = dn + a + 2h - 1$.
An application of Lemma \ref{bridgeLemmaResults} shows that $H\pr{\sqrt{M t}; V_n}$ is given by \eqref{HnExpr}.
Since $U$ satisfies the parabolic boundary conditions, then the computation in \eqref{DnExp} is applicable and we see that
\begin{align*}
W_h(\sqrt{Mt};V_n)
&= \frac{1}{\pr{Mt}^{\frac A 2}} D(\sqrt{Mt}; V_n)
-\frac{h}{\pr{Mt}^{\frac{A+1}2} } H(\sqrt{Mt}; V_n) \\
&= \frac{\bar{C}t^{-h}}{M^{\frac{A+1}{2}}} \brac{\int_{\R^{d+1}_+} U(X,t) \brac{(X,2t) \cdot \gr_{(X,t)} U(X,t)} \dGnt - 2 \int_0^t \pr{\frac \tau t}^{\frac{dn-1+a}2} \mathcal{M}_n(\tau) d\tau} \\
&- \frac{h \bar{C}t^{-h}}{M^{\frac{A+1}{2}}} \int_{\R^{d+1}_+}  \abs{U(X,t)}^2 \dGnt
 .
\end{align*} 
Since $U$ is a solution to \eqref{pWeissPDE}, we can repeat the arguments in the proof of Theorem \ref{pMono} to get that \eqref{othergradExp} holds.
Thus, with the rescaling
\begin{equation}
\label{rescaledW}
W_{h,n}(t;U)= \frac{M^{\frac{A+1}{2}}}{\bar{C}}  W_h(\sqrt{Mt};V_n),
\end{equation}
we see that
\begin{align*}
\lim_{n \to \iny} W_{h,n}(t;U)
&= \mathcal{W}_{h}(t; U).
\end{align*}

Now we establish bounds for the derivatives of $W_{h,n}(t;U)$.
Since $U$ satisfies parabolic boundary conditions of type I, then Lemma \ref{BCLemma} shows that each $V_n$ satisfies elliptic boundary conditions of type I.
Let $\overline{V}_n$ denote the $h$-homogeneous extension of $V_n|_{\BB{H}^{dn}_{\sqrt{Mt}}}$ to $\BB{D}^{dn+1}_{\sqrt{Mt}}$.
An application of \eqref{Weisselliptic} from Theorem \ref{monoWeissThm} to $W_h(\sqrt{Mt};V_n)$, followed by the chain rule, shows that for every $t \in (0, T)$,
\begin{align*}
\frac{d}{dt} W_h(\sqrt{Mt};V_n)
&= \frac{1}{2} \sqrt{\frac{M}{t}}\set{
\frac{A}{\sqrt{Mt}} \brac{W_h\pr{\sqrt{Mt};\overline{V}_n} -W_h\pr{\sqrt{Mt};V_n} }
+ \frac{1}{(Mt)^{\frac{A+2}{2}}} \int_{\BB{H}^{dn}_{\sqrt{Mt}}}\pr{Y \cdot \nabla V_n -  h V_n}^2 \, y_0^a } \\
&= \frac{A}{2t} \brac{W_h\pr{\sqrt{Mt};\overline{V}_n} -W_h\pr{\sqrt{Mt};V_n} }
+ \frac{1}{2t (Mt)^{\frac{A+1}{2}}} \int_{\BB{H}^{dn}_{\sqrt{Mt}}}\pr{Y \cdot \nabla V_n -  h V_n}^2 \, y_0^a .
\end{align*}

Since $U$ satisfies \eqref{pWeissPDE}, then Lemma \ref{ChainRuleLem} shows that each $V_n$ is a solution to \eqref{VnKnEllipEqn} in $\mathbb{D}^{dn+1}_{\sqrt{MT}}$ with $K_n$ as in \eqref{KnDef}.
Extend $V_n$ evenly from the half-ball $\BB{D}^{dn+1}_{\sqrt{Mt}}$ to the full ball $\BB{B}^{dn+1}_{\sqrt{Mt}}$ and denote the extended function by $V_n^e$. 
According to Lemma \ref{evenExtLemma}, $V_n^e \in H^1\pr{\BB{B}^{dn+1}_{\sqrt{Mt}}, \abs{y_0}^a}$, while Lemma \ref{evenLem} shows that $\di\pr{\abs{y_0}^a \gr V_n^e} = K_n^e$ in $\BB{B}^{dn+1}_{\sqrt{Mt}}$ in the weak sense, where $K_n^e$ denotes the even extension of $K_n$.

An application of Lemma \ref{weakSolExLemm} shows that there exists a unique function $Q_n \in H^1\pr{\BB{B}^{dn+1}_{\sqrt{Mt}}, \abs{y_0}^a}$ that weakly solves
\begin{equation*}
\left\{
\begin{array}{rl} 
\di\pr{\abs{y_0}^a \gr Q_n} = 0 & \text{ in } \BB{B}^{dn+1}_{\sqrt{Mt}} \\
Q_n = V_n^e & \text{ on } \BB{S}^{dn+1}_{\sqrt{Mt}} .
\end{array}
\right.
\end{equation*}
Since $Q_n = V_n$ on $\BB{H}_{\sqrt{Mt}}$, then $\overline{Q}_n = \overline{V}_n$ on $\BB{D}_{\sqrt{Mt}}$ so that $W_h\pr{\sqrt{Mt};\overline{Q}_n} = W_h\pr{\sqrt{Mt};\overline{V}_n}$ and then 
\begin{equation*}
\begin{aligned}
\frac{d}{dt} W_h\pr{\sqrt{Mt};V_n}
&= \frac{A}{2t} \brac{ W_h\pr{\sqrt{Mt};\overline{Q}_n} -W_h\pr{\sqrt{Mt};Q_n} } 
+ \frac{1}{2t (Mt)^{\frac{A+1}{2}}} \int_{\BB{H}^{dn}_{\sqrt{Mt}}}\pr{Y \cdot \nabla V_n -  h V_n}^2 \, y_0^a \\
&+ \frac{A}{2t} \brac{ W_h\pr{\sqrt{Mt}; Q_n} -W_h\pr{\sqrt{Mt};V_n} }.
\end{aligned}
\end{equation*}
Because $V_n^e$ is even with respect to $y_0$, then Lemma \ref{evenSolLemma} shows that $Q_n$ is also even with respect to $y_0$.
Therefore, $Q_n$ is weakly $a$-harmonic and even with respect to $y_0$, so we may apply \eqref{epiIneqweakHalf} from Proposition \ref{T:epi} to get
\begin{align*}
W_h\pr{\sqrt{Mt}; \overline{Q}_n} - W_h\pr{\sqrt{Mt}; Q_n}
&\ge \kappa_0 W_h\pr{\sqrt{Mt}; \overline{Q}_n} 
\ge \frac{\kappa_0}{1 - \kappa_0} W_h\pr{\sqrt{Mt}; Q_n}.
\end{align*}
From \eqref{kappa0Def}, we see that $\disp \frac{\kappa_0}{1 - \kappa_0} 
= \frac{1 + \lfloor h \rfloor - h}{dn + a + 2h - 1} = \frac{\de_h}{A}$, where $\de_h := 1 + \lfloor h \rfloor - h \in (0, 1]$.
Combining the previous two observations shows that
\begin{equation}
\label{aharmest}
\begin{aligned}
\frac{d}{dt} W_h\pr{\sqrt{Mt};V_n}
&\ge \frac{\de_h}{2t} W_h\pr{\sqrt{Mt}; Q_n}
+ \frac{1}{2t (Mt)^{\frac{A+1}{2}}} \int_{\BB{H}^{dn}_{\sqrt{Mt}}}\pr{Y \cdot \nabla V_n -  h V_n}^2 \, y_0^a \\
&+ \frac{A}{2t} \brac{ W_h\pr{\sqrt{Mt}; Q_n} -W_h\pr{\sqrt{Mt};V_n} } \\
&= \frac{\de_h}{2t} W_h\pr{\sqrt{Mt}; V_n} 
+ \frac{1}{2t (Mt)^{\frac{A+1}{2}}} \int_{\BB{H}^{dn}_{\sqrt{Mt}}}\pr{Y \cdot \nabla V_n -  h V_n}^2 \, y_0^a \\
&+ \frac{dn + a + h + \lfloor h \rfloor}{2t} \brac{ W_h\pr{\sqrt{Mt}; Q_n} -W_h\pr{\sqrt{Mt};V_n} }.
\end{aligned}
\end{equation}

Observe that
\begin{align*}
\int_{\BB{B}_{\sqrt{Mt}}} \left(|\nabla Q_n|^2-|\nabla V_n^e|^2\right) \abs{y_0}^a
&= \int_{\BB{B}_{\sqrt{Mt}}} \nabla\pr{Q_n + V_n^e} \cdot \nabla(Q_n - V_n^e) \abs{y_0}^a \\
&= 2 \int_{\BB{B}_{\sqrt{Mt}}} \nabla Q_n \cdot \nabla(Q_n - V_n^e) \abs{y_0}^a
+ \int_{\BB{B}_{\sqrt{Mt}}} \nabla\pr{V_n^e - Q_n} \cdot \nabla(Q_n - V_n^e) \abs{y_0}^a \\
&= - \int_{\BB{B}_{\sqrt{Mt}}}\abs{\nabla\pr{Q_n-V_n^e}}^2 \abs{y_0}^a,
\end{align*}
since $Q_n$ is weakly $a$-harmonic and $Q_n - V_n^e \in H^1_0\pr{\BB{B}_{\sqrt{Mt}}, \abs{y_0}^a}$.
On the other hand, since $V_n^e$ weakly solves $\di\pr{\abs{y_0}^a \gr V_n^e} = K_n^e$, then
\begin{align*}
\int_{\BB{B}_{\sqrt{Mt}}} \left(|\nabla Q_n|^2-|\nabla V_n^e|^2\right) \abs{y_0}^a
&= \int_{\BB{B}_{\sqrt{Mt}}} \nabla V_n^e \cdot \nabla\pr{Q_n-V_n^e} \abs{y_0}^a
= - \int_{\BB{B}_{\sqrt{Mt}}} K_n^e \pr{Q_n-V_n^e} \abs{y_0}^a.
\end{align*}
Comparing these equations shows that
\begin{align*}
\int_{\BB{B}_{\sqrt{Mt}}}\abs{\nabla\pr{Q_n-V_n^e}}^2 \abs{y_0}^a
&= \int_{\BB{B}_{\sqrt{Mt}}} K_n^e \pr{Q_n-V_n^e} \abs{y_0}^a
\le \pr{\int_{\BB{B}_{\sqrt{Mt}}} \abs{K_n^e}^2 \abs{y_0}^a}^{\frac 1 2} \pr{\int_{\BB{B}_{\sqrt{Mt}}} \abs{Q_n-V_n^e}^2 \abs{y_0}^a}^{\frac 1 2} \\
&\le \sqrt{\frac{Mt}{dn + a + 1}}\pr{\int_{\BB{B}_{\sqrt{Mt}}} \abs{K_n^e}^2 \abs{y_0}^a}^{\frac 1 2} \pr{\int_{\BB{B}_{\sqrt{Mt}}} \abs{\gr\pr{Q_n-V_n^e}}^2 \abs{y_0}^a}^{\frac 1 2},
\end{align*}
where we have applied H\"older's inequality followed by the Poincar\'e inequality in Lemma \ref{L:poincare}.
That is,
\begin{align*}
\int_{\BB{B}_{\sqrt{Mt}}} \left(|\nabla Q_n|^2-|\nabla V_n^e|^2\right) \abs{y_0}^a
&= - \int_{\BB{B}_{\sqrt{Mt}}}\abs{\nabla\pr{Q_n-V_n^e}}^2 \abs{y_0}^a
\ge - \frac{Mt}{dn + a + 1} \int_{\BB{B}_{\sqrt{Mt}}} \abs{K_n^e}^2 \abs{y_0}^a,
\end{align*}
or, equivalently, 
\begin{align*}
\int_{\BB{D}_{\sqrt{Mt}}} \left(|\nabla Q_n|^2-|\nabla V_n|^2\right) y_0^a
&\ge - \frac{Mt}{dn + a + 1} \int_{\BB{D}_{\sqrt{Mt}}} \abs{K_n}^2 y_0^a.
\end{align*}
Substituting this bound into \eqref{aharmest} and recalling the definition of the Weiss functional in \eqref{ellipWeiss}, we have
\begin{equation*}
\begin{aligned}
\frac{d}{dt} W_h\pr{\sqrt{Mt};V_n}
&\ge \frac{\de_h}{2t} W_h\pr{\sqrt{Mt}; V_n} 
+ \frac{1}{2t (Mt)^{\frac{A+1}{2}}} \int_{\BB{H}^{dn}_{\sqrt{Mt}}}\pr{Y \cdot \nabla V_n  -  h V_n}^2 \, y_0^a
-  \frac {\frac{dn + a + h + \lfloor h \rfloor}{dn + a + 1}}{2t\pr{Mt}^{\frac{A-2} 2}} \int_{\BB{D}^{dn+1}_{\sqrt{Mt}}} \abs{K_n}^2 y_0^a \\
&= \frac{\de_h}{2t} W_h\pr{\sqrt{Mt}; V_n} 
+ \frac{\bar{C} }{M^{\frac{A+1}{2}}} \set{\frac 1 {2t^{h+1}} \int_{\R^{d+1}_+} \brac{(X, 2t) \cdot \gr_{(X,t)} U - h U}^2 \dGnt - E_n(t)},
\end{aligned}
\end{equation*}
where we have applied Lemmas \ref{PFTS} and \ref{PFT} with \eqref{YdotGrExp} and \eqref{KnDef}, then introduced
$$E_n(t) = \frac{4 \frac{dn + a + h + \lfloor h \rfloor}{dn + a + 1}}{t^h} \int_0^t \pr{\frac \tau t}^{\frac{dn-1+a}{2}} \mathcal{J}_n\pr{\tau} d\tau.$$
Let 
\[
\mathcal{P}(t; U) := \int_{\R^{d+1}_+} \brac{(X, 2t) \cdot \gr_{(X,t)} U - h U}^2 \dGt
\]
and define $\mathcal{P}_n(t; U)$ analogously with $\mathcal{G}_n$ in place of $\mathcal{G}$.

Recalling \eqref{rescaledW}, we have shown that for every $t \in (0, T)$,
\begin{equation}
\label{WeissMonoCombined}
\frac{d}{dt} \brac{t^{- \frac{\de_h}{2}} W_{h,n}(t;U)}
\ge t^{- \frac{\de_h}{2}} \brac{\frac{1}{2t^{h+1}} \mathcal{P}_n(t) 
-E_n(t)}.
\end{equation}
To show the desired estimate for $\disp \frac{d}{dt}\mathcal{W}_h(t;U)$, it suffices to show two uniform convergence statements: 
given any $t_0 \in (0, T]$, there exists $\eps_0 \in \pr{0, t_0}$ so that $E_n(t)$ converges uniformly to $0$ and $\mathcal{P}_n$ converges uniformly to $\mathcal{P}$; both on $\brac{t_0 - \eps_0, t_0}$.

Indeed, from \eqref{WeissMonoCombined}, for any $t \in \brac{t_0 - \eps_0, t_0}$, it holds that
\[
t_0^{- \frac{\de_h}{2}}  W_{h,n}(t_0;U) - t^{- \frac{\de_h}{2}}  W_{h,n}(t;U)
\ge \int_{t}^{t_0} \tau^{- \frac{\de_h}{2}} \brac{\frac{1}{2\tau^{h+1}} \mathcal{P}_n(\tau) - E_n(\tau)} d\tau.
\]
Since $W_{h,n}(t;U)$ converges pointwise to $\mathcal{W}_h(t;U)$, then
\[
t_0^{- \frac{\de_h}{2}} \mathcal{W}_h(t_0;U) - t^{- \frac{\de_h}{2}} \mathcal{W}_h(t;U)
\ge \lim_{n \to \iny} \int_{t}^{t_0} \tau^{- \frac{\de_h}{2}} \brac{\frac{1}{2\tau^{h+1}} \mathcal{P}_n(\tau) - E_n(\tau)} d\tau.
\]
Assuming the two local uniform convergences described above on $\brac{t_0 - \eps_0, t_0} \supseteq \brac{t, t_0}$, we see that
\begin{align*}
\lim_{n \to \iny} \int_{t}^{t_0} \tau^{- \frac{\de_h}{2}} \brac{\frac{1}{2\tau^{h+1}} \mathcal{P}_n(\tau)- E_n(\tau)} d\tau 
&= \int_{t}^{t_0} \lim_{n \to \iny} \tau^{- \frac{\de_h}{2}} \brac{ \frac{1}{2\tau^{h+1}} \mathcal{P}_n(\tau) - E_n(\tau)} d\tau \\
&= \int_{t}^{t_0}  \brac{\frac{\tau^{- \frac{\de_h}{2}}}{2\tau^{h+1}}\int_{\R^{d+1}_+}\left(X\cdot\nabla U+2\tau \partial_\tau U -h \, U\right)^2x_0^a\mathcal{G}(X,\tau)dX}d\tau.
\end{align*}
By the Lebesgue differentiation theorem, since the bracketed term is integrable and $\mathcal{W}_h(t,U)$ is differentiable, we may conclude that for almost every $t_0\in(0,T],$
\[
\frac{d}{dt}\mathcal{W}_h(t_0,U)
\ge \frac{\de_h}{2t_0} \mathcal{W}_h(t_0;U) + \frac{1}{2t_0^{h+1}}\int_{\R^{d+1}_+}\left(X\cdot \nabla U+2t_0 \partial_t U -h\, U \right)^2x_0^a\mathcal{G}(X,t_0)dX.
\] 
By continuity, the statement holds for all $t_0 \in (0, T]$. 

It remains to justify the two local uniform convergence statements from above. 
For $t_0 \in (0, T]$, since $U\in \mathfrak{W}\pr{\R^{d+1}_+ \times \pr{0, T}}$, then let $\eps \in(0,t_0)$ be from \eqref{WeissClass1} and let $p > 1, \al \in \R$ be from \eqref{WeissClass2}. We first show that $E_n$ converges uniformly to $0$ in $[t_0-\eps,t_0]$. 
Since $\frac{dn+a+h+ \lfloor h \rfloor}{dn+a+1}\rightarrow 1$, then for $n \gg 1$, Lemma \ref{GaussianLimit} shows that
$$E_n(t) \le \frac{5 C_0}{t^h} \int_0^t \pr{\frac \tau t}^{\frac{dn-1+a}{2}} \mathcal{J}\pr{\tau} d\tau.$$
By \eqref{WeissClass2}, Lemma \ref{weightedLpBound} is applicable and shows that whenever $n \gg 1$, we have
\[
\sup_{t \in [t_0-\eps,t_0]} \int_0^{t} \left(\frac{\tau}{t}\right)^{\frac{dn-1+a}{2}}\mathcal{J}(\tau)d\tau
\le C_1 (t_0-\eps)^{-\al}\left(\frac{t_0}{n}\right)^{1-\frac{1}{p}}||t^\al\mathcal{J}||_{L^p([0,t_0])}.
\]
Combining these bounds shows that $E_n$ converges to zero uniformly in $[t_0-\eps,t_0]$. 

Next we show that $\mathcal{P}_n$ converges uniformly to $\mathcal{P}$ in $[t_0-\eps,t_0]$. 
Since $\mathcal{P}(t) \le 2 \pr{\mathcal{I}(t) + h^2 \mathcal{H}(t)}$, then \eqref{WeissClass1} shows that $\mathcal{P} \in L^\iny\pr{\brac{t_0 - \eps, t_0}}$ as well.
An application of Lemma \ref{uniformConvLemma} implies that for some $\eps_0 \in (0, \eps]$, $\mathcal{P}_n$ converges uniformly to $\mathcal{P}$ on $\brac{t_0 - \eps_0, t_0}$.

As we have shown the two desired uniform convergence properties, the proof is complete.
\end{proof}

\section{Alt-Caffarelli-Friedman Frequency Functions}
\label{S:ACF}

The main result of this section is a parabolic counterpart to the degenerate Alt-Caffarelli-Friedman (ACF) monotonicity formulas of \cite{TVZ14}, \cite{TVZ2}.
This result is given below in Theorem \ref{T:parabolicACF}. 
To prove this degenerate parabolic
result via the high-dimensional limiting technique, we first need to prove a result that applies to elliptic subsolutions, as described by \eqref{strangecondition} below.

\subsection{Alt-Caffarelli-Friedman frequency functions in the elliptic setting}
\label{SS:EllipACF}

Here we prove nonhomogeneous versions of the degenerate ACF monotonicity formulas that appear in \cite{TVZ14} and \cite{Zil14} for general $a$, and in \cite{TVZ2} for the square root of the Laplacian.

\begin{thm}[Nonhomogeneous Elliptic ACF]
\label{C:ACFellipticnonzero}
For $R > 0$, assume that $V_1, V_2 \in C^0\pr{\overline{\BB{D}}_R} \cap H^1\pr{\BB{D}_R, y_0^a}$ are non-negative functions with $V_1\pr{0} = 0 = V_2\pr{0}$ and $V_1 V_2 \equiv 0$ on $\BB{B}_R^{N+1}\cap\{y_0=0\}$. 
For each $i$, assume that $K_i$ is chosen so that $K_i V_i \in L^1_{loc}(\BB{D}_R,y_0^a)$, 
and for every $r \in (0, R)$, $\disp \int_{\BB{D}_r} K_i^- V_i |Y|^{-(N-1+a)} y_0^a \, dY < \infty$.
Assume that each $V_i$ is a subsolution in the sense that for all $\Psi\in C_0^{2}(\BB{B}_R)$ with $\Psi\ge 0$, we have
\begin{equation}
\label{strangecondition}
\int_{\R^{N+1}_+} \brac{\nabla V_i \cdot \nabla \pr{V_i\Psi} + K_i(V_i\Psi) } y_0^a \, dY
\le 0.
\end{equation}
For a.e. $r \in (0, R)$, the function $\phi\pr{r} = \phi\pr{r; V_1, V_2}$ given by 
\begin{equation}
\label{ACF}
\phi(r)
= \frac{1}{r^{1-a}} \pr{\int_{\BB{D}_r} |\nabla V_1(Y)|^2 |Y|^{-(N-1+a)} y_0^a \, dY } \pr{\int_{\BB{D}_r} |\nabla V_2(Y)|^2 |Y|^{-(N-1+a)} y_0^a \, dY }
\end{equation}
is well-defined and bounded.
Moreover, for each such $r$ for which $\disp \pr{\int_{\BB{H}_r} V_{1}^2 y_0^a \, d\eta} \pr{\int_{\BB{H}_r} V_{2}^2 y_0^a \, d\eta} \neq 0$, we have that
\begin{equation}
\label{phiTDer}
\begin{aligned}
\tfrac{N - 1+ a}2 \phi'\pr{r}
&\ge - r^a \frac{ \pr{
\int_{\BB{D}_r} K_1^- V_1 |Y|^{-(N-1+a)} y_0^a \, dY}
\pr{\int_{\BB{D}_r} |\nabla V_2|^2 |Y|^{-(N-1+a)} y_0^a \, dY} \pr{\int_{\BB{H}_r} |\nabla V_1|^2 y_0^a \, d\eta} }{\pr{\int_{\BB{H}_r} V_1^2 y_0^a \, d\eta}} \\
&- r^a \frac{\pr{\int_{\BB{D}_r} |\nabla V_{1}|^2 |Y|^{-(N-1+a)} y_0^a \, dY} \pr{\int_{\BB{D}_r} K_{2}^- V_{2} |Y|^{-(N-1+a)} y_0^a \, dY} \pr{\int_{\BB{H}_r} |\nabla V_{2}|^2 y_0^a \, d\eta}} {\pr{\int_{\BB{H}_r} V_2^2 y_0^a \, d\eta}} .
\end{aligned}
\end{equation}
\end{thm}

\begin{rem} If, instead of \eqref{strangecondition}, $V_i$ satisfies $\div(y_0^a\nabla V_i)\ge K_iy_0^a$ pointwise on $\D_R$, and $V_i\partial_{\nu}^a V_i(0,y) = 0$ on $\BB{B}_R \cap \{y_0=0\}$, then \eqref{strangecondition} holds.

Indeed, given $\Psi\in C_0^{2}(\BB{B}_R)$ with $\Psi\ge 0$, one has
\begin{align*}
\int_{\R^{N+1}_+} \nabla V_i \cdot \nabla \pr{V_i\Psi} y_0^a
&= \int_{\BB{B}_R \cap \{y_0=0\}} (V_i\Psi) \partial_{\nu}^a V_i(0,y)
- \int_{\R^{N+1}_+} \div(y_0^a\nabla V_i) V_i\Psi 
\le-\int_{\R^{N+1}_+}K_i (V_i\Psi) y_0^a.
\end{align*}

\end{rem}

\begin{rem}
\label{radialTestFunc}
An inspection of the proof of Theorem \ref{C:ACFellipticnonzero} shows that our arguments also hold in the setting where the test functions $\Psi \in C^2_0\pr{\BB{B}_R^{N+1}}$ are assumed to be radial.
\end{rem}

Our proof differs from the original of Alt-Caffarelli-Friedman \cite{ACF84} in a few ways. 
First, as in \cite{TVZ14}, \cite{TVZ2} and \cite{Zil14}, we do not work with distributional solutions, but instead consider the type of weak subsolution described by \eqref{strangecondition}.
Next, we need to deal with a technical obstruction: mollifying our subsolutions in the usual way changes the operator.
To overcome this, we carefully employ smooth cutoff functions and replace our singular weight, $\abs{Y}^{-(N-1_+a)}$, with a bounded, $C^2$ approximation of itself. 

\begin{proof}
To simplify notation, let $\Ga(Y) = |Y|^{-(N-1+a)}$.
We first need to show that 
the mapping
\begin{equation}
\label{acf1}
r\mapsto \int_{\BB{D}_r} |\nabla V(Y)|^2 \Ga(Y) \, y_0^a \, dY
\end{equation}
is well-defined and bounded a.e. in $(0,R)$ for each $V = V_i$. 

If $G(Y) = f(\abs{Y})$ is a radial function, then with $\al = N+a$ and $L_a = \di\pr{y_0^a \gr}$, it holds that 
$$L_a G = y_0^a\pr{f'' + \frac{\al}{\abs{Y}} f'}.$$
In particular, $L_a \Ga = 0$ away from $Y = 0$.

We now construct $C^2$ approximations to $\Ga$ that are bounded near the origin.
Fix $\eps > 0$ and define 

$\disp g_0 : \brac{\frac \eps 2, \frac{3 \eps} 2} \to \R$ by
$$g_0(\rho) = \eps^{1-\al}\brac{1 
+ c_3 \pr{\frac \rho \eps - \frac 1 2}^3 
+ c_4 \pr{\frac \rho \eps - \frac 1 2}^4 
+ c_5 \pr{\frac \rho \eps - \frac 1 2}^5},$$
where
\begin{align*}
& c_3 = 
10\brac{\pr{\frac {3}2}^{1 - \al} -1} 
- 4\pr{1 - \al} \pr{\frac{3} 2}^{-\al} 
+ \frac 1 2\al \pr{\al - 1} \pr{\frac {3}2}^{-\pr{\al + 1}} \\
&
c_4 = 
- 15 \brac{\pr{\frac {3}2}^{1 - \al} -1} 
+ 7\pr{1 - \al} \pr{\frac{3} 2}^{-\al} 
- \al \pr{\al - 1} \pr{\frac {3}2}^{-\pr{\al + 1}} \\
&
c_5 = 
6 \brac{\pr{\frac {3}2}^{1 - \al} -1} 
- 3 \pr{1 - \al} \pr{\frac{3} 2}^{-\al} 
+ \frac 1 2 \al \pr{\al - 1} \pr{\frac {3}2}^{-\pr{\al + 1}} .
\end{align*}
Since
\begin{align*}
&g_0'(\rho) = \eps^{-\al}\brac{
3 c_3 \pr{\frac \rho \eps - \frac 1 2}^2
+ 4c_4 \pr{\frac \rho \eps - \frac 1 2}^3 
+ 5c_5 \pr{\frac \rho \eps - \frac 1 2}^4} \\
&g_0''(\rho) = \eps^{-\pr{\al+1}}\brac{
6 c_3 \pr{\frac \rho \eps - \frac 1 2}
+ 12 c_4 \pr{\frac \rho \eps - \frac 1 2}^2
+ 20 c_5 \pr{\frac \rho \eps - \frac 1 2}^3 },
\end{align*}
then there exist constants $c_1, c_2 > 0$ so that for all $\disp \rho \in \brac{\frac \eps 2, \frac{3 \eps} 2}$, $\disp \abs{g_0'(\rho)} \le \frac{c_1}{\eps^\al}$ and $\disp \abs{g_0''(\rho)} \le \frac{c_2}{\eps^{\al+1}}$.
Moreover, it can be seen that $\disp g_0\pr{\frac \eps 2} = \eps^{1 - \al}$, $\disp g_0'\pr{\frac \eps 2} = 0$, and $\disp g_0''\pr{\frac \eps 2} = 0$, while further computations show that
\begin{align*}
&g_0\pr{\frac {3\eps} 2} 
= \eps^{1-\al}\pr{1 + c_3 + c_4 + c_5 }
= \pr{\frac {3 \eps}2}^{1 - \al} \\
&g_0'\pr{\frac {3\eps} 2}
= \eps^{-\al}\pr{ 3 c_3 + 4 c_4 + 5 c_5 }
= \pr{1 - \al} \pr{\frac {3\eps} 2}^{-\al} \\
&g_0''\pr{\frac {3\eps} 2} 
= \eps^{-\al-1 }\pr{6 c_3 + 12 c_4 + 20 c_5}
= \al \pr{\al - 1} \pr{\frac {3\eps}2}^{-\pr{\al + 1}}.
\end{align*}
In particular, with $g : [0, \iny) \to \R$ given by
$$g(\rho) := \begin{cases}
    \eps^{1 - \al} & \rho \le \frac{\eps} 2 \\
    g_0(\rho) & \frac{\eps} 2 < \rho < \frac{3 \eps} 2 \\
    \rho^{1 - \al} & \rho \ge \frac{3 \eps} 2
\end{cases},$$
we deduce that $g \in C^2((0, \iny))$ and is bounded. 

We define $\Ga_\eps \in C^2(\R^{N+1}_+)$ by $\Ga_\eps(Y) = g(\abs{Y})$ so that $\Ga_\eps$ is bounded and agrees with $\Ga$ on $\R^{N+1}_+ \setminus \mathbb{D}_{\frac {3\eps} 2 }$. 
In particular, $\gr \Ga_\eps = -\pr{N-1+a} \abs{Y}^{-\pr{N+1+a}} Y$ and $L_a \Ga_\eps(Y) = 0$ on $\R^{N+1}_+ \setminus \mathbb{D}_{\frac {3\eps} 2 }$.
Moreover, all derivatives of $\Ga_\eps$ vanish on $\mathbb{D}_{\frac \eps 2}$, while on $\mathbb{D}_{\frac {3\eps} 2 } \setminus \mathbb{D}_{\frac \eps 2}$, $\abs{\gr \Ga_\eps} \le C \eps^{-(N+a)}$ and $\abs{L_a \Ga_\eps(Y)} \le C \eps^{-(N + 1 + a)} y_0^a$.

For $r \in (0, R)$, we next construct smooth approximations to $\chi_{\mathbb{D}_r}$.
Let $\de \in (0, r)$ be small enough so that $r + \de < R$.
Choose $\mu = \mu(\de) > 0$ so that $\de = \mu + 2 \mu^2$. With $I_1 = [r + \mu^2, r+ \mu + \mu^2]$, define the ramp function $\si : [0, \iny) \to \R$ by 
\[
\si(\rho) = \begin{cases}
    1 & \text{ if } \rho \le r + \mu^2 \\
    - \frac 1 {\mu}\pr{\rho - r - \mu - \mu^2} & \text{ if } \rho \in I_1 \\
    0 & \text{ if } \rho \ge r + \mu + \mu^2.
\end{cases}
\]
Set $\be_\mu = \eta_{\mu^2} * \si$, where $\eta$ denotes the standard mollifier.
It holds that $\be_\mu \in C^{\infty}([0, r+\de])$ is a cut-off function with $0 \le \be_\mu \le 1$,  $\be_\mu = 1$ in $[0,r]$, and  $\be_\mu(r + \de) = 0$.
Therefore, with $\xi_\de(Y) = \beta_\mu(\abs{Y})$, we have that $\xi_\de\in C^{\infty}_0\pr{\BB{B}_{r+\delta}}$ is a radial cut-off function with $0 \le \xi_\de \le 1$ and $\xi_\de = 1$ in $\BB{B}_r$.

For now, to simplify notation, we write $V=V_i$, $K=K_i$. 
Let $\Psi = \Psi_{\eps, \de} := \xi_\de \Ga_\eps$ and note that $\Psi \in C^{2}_0\pr{\BB{B}_{r+\delta}}$ is bounded, non-negative, and radial.
An application of \eqref{strangecondition} with $\Psi$ gives
\begin{equation*}
\begin{aligned}
\int_{\R^{N+1}_+} \nabla V \cdot \nabla (V \xi_\de \Ga_\eps) \, y_0^a 
\le -\int_{\R^{N+1}_+} K(V \xi_\de \Ga_\eps) y_0^a
\end{aligned}
\end{equation*}
from which it follows that
\begin{equation}
\label{another}
\begin{aligned}
\int_{\R^{N+1}_+} \brac{|\nabla V|^2\Ga_\eps + \tfrac{1}{2} \nabla (V^2) \cdot \nabla\Ga_\eps} \xi_\de \, y_0^a 
&\le - \tfrac{1}{2} \int_{\R^{N+1}_+} \nabla (V^2) \cdot \nabla\xi_\de \Ga_\eps y_0^a
+ \int_{\R^{N+1}_+} K^- V \xi_\de\Ga_\eps y_0^a. 
\end{aligned}
\end{equation}

We first take $\de \to 0^+$ in the above inequality, then we take $\eps \to 0^+$.

Since $\xi_\de(Y) = \be_\mu(\abs{Y})$, then with $\disp G(\rho) = \int_{\BB{H}_{\rho}} V \pr{\nabla V \cdot \nu} \Ga_\eps \, y_0^a$, it holds that 
$$\tfrac{1}{2} \int_{\R^{N+1}_+} \nabla (V^2) \cdot \nabla \xi_\de \Ga_\eps \, y_0^a = \int_r^{r+\delta} \be_\mu'(\rho) G(\rho) d\rho.$$

We now examine the behavior of $\disp \int_{r}^{r+\de} \be_\mu'(\rho) G(\rho) d\rho$ as $\delta\rightarrow 0^+$. 

Notice that $- \be_\mu' = \frac 1 {\mu} \eta_{\mu^2} * \chi_{I_1}$ is supported on $I_2 = \brac{r, r + \mu + 2\mu^2} = \brac{r, r + \de}$ and $- \be_\mu' - \frac 1 \mu\chi_{I_1}$ is supported on $I_3 \cup I_4$, where $I_3 = \brac{r, r + 2\mu^2}$ and $I_4 = \brac{r +\mu , r + \mu + 2\mu^2}$.
It follows that
\begin{align*}
    - \int_{I_2} \be_\mu'(\rho) G(\rho) d\rho - G(r)
    &= \frac 1 {\mu} \int_{I_2} \brac{\eta_{\mu^2} * \chi_{I_1}(\rho) - \chi_{I_1}(\rho)} G(\rho) d\rho 
    + \fint_{I_2} \pr{G(\rho)- G(r)} d\rho \\
    &+ \frac 1 {\mu} \int_{I_2} \chi_{I_1}(\rho) G(\rho) d\rho 
    - \fint_{I_2} G(\rho) d\rho \\
    &= \frac 1 {\mu} \int_{I_3 \cup I_4} \pr{\eta_{\mu^2} * \chi_{I_1} - \chi_{I_1}} G(\rho) d\rho 
    + \fint_{I_2} \pr{G(\rho)- G(r)} d\rho \\
    &+ \frac {2 \mu}{1 + 2 \mu} \fint_{I_1}  G(\rho) d\rho
    - \frac {\mu}{1 + 2 \mu} \fint_{r}^{r + \mu^2} G(\rho) d\rho
    - \frac {\mu}{1 + 2 \mu} \fint_{r +\mu + \mu^2}^{r + \mu + 2 \mu^2} G(\rho) d\rho.
\end{align*}
For $i = 3, 4$, 
\begin{align*}
   \abs{\frac 1 {\mu} \int_{I_i} \pr{\eta_{\mu^2} * \chi_{I_1} - \chi_{I_1}} G(\rho) d\rho  }
   &\le 2 \mu \norm{\eta_{\mu^2} * \chi_{I_1} - \chi_{I_1}}_{L^\iny(I_i)} \fint_{I_i} \abs{G(\rho)} d\rho
   \le 2 \mu \fint_{I_i} \abs{G(\rho)} d\rho,
\end{align*}
so then
\begin{align*}
    \abs{- \int_{I_2} \be_\mu'(\rho) G(\rho) d\rho - G(r)}
    &\le \abs{\fint_{I_2} G(\rho)- G(r) d\rho} 
    + 2 \mu \brac{\fint_{I_3} \abs{G(\rho)} d\rho
    + \fint_{I_4} \abs{G(\rho)} d\rho} \\
    &+ \frac {\mu}{1 + 2 \mu}  \brac{2 \fint_{I_1}  \abs{G(\rho)} d\rho
    + \fint_{r}^{r + \mu^2} \abs{G(\rho)} d\rho
    + \fint_{r + \mu + \mu^2}^{r + \mu + 2 \mu^2} \abs{G(\rho)} d\rho}.
\end{align*}
By the Lebesgue differentiation theorem, for a.e. $r \in (0, R)$, all of these terms converge to zero as $\mu \to 0^+$.
Since $\mu \to 0^+$ if and only if $\de \to 0^+$, then for a.e. $r \in (0, R)$,
\begin{equation}
\label{limit1}
\lim_{\de \to 0^+} \frac{1}{2} \int_{\R^{N+1}_+}  \nabla (V^2) \cdot \nabla\xi_\de \Ga_\eps y_0^a 
= \lim_{\de \to 0^+} \int_r^{r+\delta} \be_\mu'(\rho) G(\rho) d\rho
= - G(r)
= - \int_{\BB{H}_{r}} V \pr{\nabla V \cdot \nu} \Ga_\eps \, y_0^a.
\end{equation}

Since $\abs{K^- V \xi_\de \Ga_\eps y_0^a} \le \abs{K^- V \Ga_\eps y_0^a} \le \abs{K^- V \Ga \, y_0^a}$, which is assumed to be integrable, the dominated convergence theorem shows that 
\begin{equation}
\label{limit2}
\lim_{\de \to 0^+}\int_{\R^{N+1}_+} K^- V \xi_\de \Ga_\eps y_0^a = \int_{\BB{D}_{r}} K^- V \Ga_\eps y_0^a.  
\end{equation}
Since $V \in H^{1}\pr{\mathbb{D}_R, y_0^a}$ and $\Ga_\eps$ is $C^2$ and bounded, another application of dominated convergence theorem gives
\begin{equation}
\label{limit3}
\lim_{\de \to 0^+}\int_{\R^{N+1}_+} \brac{|\nabla V|^2\Ga_\eps + \tfrac{1}{2} \nabla (V^2) \cdot \nabla\Ga_\eps} \xi_\de \, y_0^a 
= \int_{\BB{D}_{r}} \brac{|\nabla V|^2\Ga_\eps + \tfrac{1}{2} \nabla (V^2) \cdot \nabla\Ga_\eps} y_0^a .
\end{equation}
In particular, taking the limit as $\de \to 0^+$ in \eqref{another} and using \eqref{limit1} -- \eqref{limit3} shows that for a.e. $r \in (0, R)$,
\begin{equation}
\label{delta0}
\begin{aligned}
\int_{\BB{D}_{r}} \brac{|\nabla V|^2\Ga_\eps + \tfrac{1}{2} \nabla (V^2) \cdot \nabla\Ga_\eps} y_0^a
&\le \int_{\BB{H}_{r}} V \pr{\nabla V \cdot \nu} \Ga_\eps y_0^a
+ \int_{\BB{D}_{r}} K^- V \Ga_\eps y_0^a. 
\end{aligned}
\end{equation}
We now carefully take $\eps \to 0^+$.

Because $V \in H^{1}\pr{\mathbb{D}_R, y_0^a}$ and $\Ga_\eps$ is $C^2$ and bounded, then
\begin{equation*}
\begin{aligned}
       \int_{\BB{D}_r}\nabla (V^2) \cdot \nabla \Ga_\eps \,  y_0^a
    &= \lim_{\de \to 0^+} \int_{\BB{B}_r \cap \set{y_0 > \de}} \nabla (V^2) \cdot \nabla \Ga_\eps \, y_0^a.
   \end{aligned}
\end{equation*}
From here, an integration by parts gives
\begin{align*}
\int_{\BB{B}_r \cap \set{y_0 > \de}} \brac{\nabla (V^2) \cdot \nabla \Ga_\eps \, y_0^a
 + V^2 L_a(\Ga_\eps) }
&= \int_{\BB{H}_r \cap \set{y_0 > \de}} V^2 \pr{\nabla \Ga_\eps \cdot \nu} y_0^a
+ \int_{\BB{B}_r \cap \set{y_0 = \de}} V^2 \pr{\nabla \Ga_\eps \cdot \nu} y_0^a.
\end{align*}
Since $\abs{L_a \Ga_\eps(Y)} \le C \eps^{-(N + 1 + a)} y_0^a$ and $V \in H^{1}\pr{\mathbb{D}_R, y_0^a}$, then 
\begin{align*}
\lim_{\de \to 0^+} \int_{\BB{B}_r \cap \set{y_0 > \de}} V^2 L_a(\Ga_\eps) 
&= \int_{\BB{D}_r} V^2 L_a(\Ga_\eps).
\end{align*}
Assuming that $r > \frac{3\eps}{2}$, $\Ga_\eps = \Ga$ on $\BB{H}_r$ and then
\begin{align*}
\lim_{\de \to 0^+} \int_{\BB{H}_r \cap \set{y_0 > \de}}  V^2 \pr{\nabla \Ga_\eps \cdot \nu} y_0^a    
&= - \pr{N-1+a} r^{-(N+a)}\int_{\BB{H}_r} V^2 y_0^a.
\end{align*}
Since $\abs{\nabla \Ga_\eps \cdot \nu} \le C \eps^{-(N+1+a)} y_0$ on $\BB{B}_r^{N+1} \cap \set{y_0 = \de}$, an application of \eqref{trace1} in Lemma \ref{traceLemma} shows that
\begin{align*}
\lim_{\de \to 0^+} \abs{\int_{\BB{B}_r \cap \set{y_0 = \de}} V^2 \pr{\nabla \Ga_\eps \cdot \nu} y_0^a}
&\lesssim \eps^{-(N+1+a)} \lim_{\de \to 0^+} \int_{\BB{B}_r \cap \set{y_0 = \de}}V^2 y_0^{1+a}
= 0.
\end{align*}
Therefore,
\begin{equation}
\label{noThinBoundary}
\begin{aligned}
    \int_{\BB{D}_r} \nabla (V^2) \cdot \nabla \Ga_\eps \, y_0^a
    &= - \int_{\BB{D}_r} V^2 L_a(\Ga_\eps)
    - \pr{N-1+a} r^{-(N+a)}\int_{\BB{H}_r} V^2 y_0^a.
\end{aligned}
\end{equation}
Substituting \eqref{noThinBoundary} into \eqref{delta0} shows that
\begin{equation}
\label{almostThere}
\begin{aligned}
\int_{\BB{D}_{r}} |\nabla V|^2\Ga_\eps y_0^a
&\le r^{-(N-1+a)}\int_{\BB{H}_{r}} V \pr{ \nabla V \cdot \nu} y_0^a
+ \frac{N-1+a}{2} r^{-(N+a)}\int_{\BB{H}_r} V^2 y_0^a
+ \int_{\BB{D}_{r}} K^- V \Ga_\eps y_0^a \\
&+ \frac{1}{2} \int_{\BB{D}_r} V^2 L_a(\Ga_\eps). 
\end{aligned}
\end{equation}
Since $\abs{K^- V \Ga_\eps y_0^a} \le \abs{K^- V \Ga \, y_0^a}$, which is assumed to be integrable, the dominated convergence theorem shows that 
\begin{equation*}
\lim_{\eps \to 0^+} \int_{\BB{D}_{r}} K^- V \Ga_\eps y_0^a 
= \int_{\BB{D}_{r}} K^- V \Ga \, y_0^a .  
\end{equation*}
Because $\abs{L_a \Ga_\eps(Y)} \le C \eps^{-(N + 1 + a)} y_0^a$ on $\mathbb{D}_{\frac {3\eps} 2 } \setminus \mathbb{D}_{\frac \eps 2}$ and ${L_a \Ga_\eps(Y)} = 0$  otherwise, then
\begin{align*}
\abs{\int_{\BB{D}_r} V^2 L_a(\Ga_\eps) }
&\lesssim \eps^{-(N + 1 + a)} \int_{\mathbb{D}_{2\eps}} V^2 y_0^a
\lesssim \norm{V}^2_{L^{\iny}\pr{\mathbb{D}_{2\eps}}}.
\end{align*}

As $V$ is continuous with $V(0) = 0$, then $\disp \lim_{\eps \to 0^+} \norm{V}_{L^\iny\pr{\mathbb{D}_{2\eps}}} = 0$, 
so we may conclude that the right-hand side of \eqref{almostThere} is bounded independent of $\eps> 0$.
Taking the limit $\eps \to 0^+$ in \eqref{almostThere} shows that for a.e. $r \in (0, R)$,
\begin{equation}
\label{usenegativepart}
\begin{aligned}
\int_{\BB{D}_{r}} |\nabla V|^2\Ga \, y_0^a
&\le r^{-(N-1+a)}\int_{\BB{H}_{r}} V\pr{ \nabla V \cdot \nu} y_0^a
+ \frac{N-1+a}{2} r^{-(N+a)}\int_{\BB{H}_r} V^2 y_0^a
+ \int_{\BB{D}_{r}} K^- V \Ga \, y_0^a.
\end{aligned}
\end{equation}
\noindent Therefore, the function in \eqref{acf1} is indeed well-defined and bounded for a.e. $r \in (0,R)$. 

With $\phi(r)$ as defined in \eqref{ACF}, the coarea formula and the Lebesgue differentiation theorem imply that for a.e. $r \in (0, R)$,

\begin{equation}
\label{ACF'}
\begin{aligned}
\phi'(r)
&= - \frac{1-a}{r^{2-a}} \pr{\int_{\BB{D}_r} |\nabla V_1|^2 \Ga \,y_0^a }
\pr{\int_{\BB{D}_r}  |\nabla V_2|^2 \Ga \,y_0^a} \\
&+ \frac{1}{r^{N}} \brac{\pr{\int_{\BB{H}_r} |\nabla V_1|^2 y_0^a} 
\pr{\int_{\BB{D}_r} |\nabla V_2|^2 \Ga \,y_0^a}
+ \pr{\int_{\BB{D}_r} |\nabla V_1|^2 \Ga \,y_0^a} \pr{\int_{\BB{H}_r} |\nabla V_2|^2 y_0^a}}.
\end{aligned}
\end{equation}
Assume for now that $\phi(r)$ is well-defined, bounded, and satisfies \eqref{ACF'} when $r=1$.
Assume also that $\disp \pr{\int_{\BB{H}_1} V_{1}^2 y_0^a} \pr{\int_{\BB{H}_1} V_{2}^2 y_0^a} \neq 0$. Let $\gr_\te \omega$ denote the gradient of the function $\omega$ on $\BB{H}_1$, the upper half unit sphere in $\R^{N+1}$.

Let $S_i$ denote the support of $V_i$ on $\mathbb{S}^{N-1} = \del \BB{H}_1$ for $i = 1, 2$. 
Define
\begin{align*}
\frac{1}{\al_i} = \inf \set{ \frac{\int_{\BB{H}_1} \abs{\gr_\te \om}^2 y_0^a}{\int_{\BB{H}_1} \om^2 y_0^a} : \om \in W^{1,2}\pr{\BB{H}_1, y_0^a}, \om \equiv 0 \text{ on } \mathbb{S}^{N-1}\setminus S_i}.
\end{align*}
Note that since $V_1 V_2 \equiv 0$, then $\pr{\mathbb{S}^{N-1}\setminus S_1} \cap \pr{\mathbb{S}^{N-1}\setminus S_2} = \emptyset$.
Observe that for any $\be_i \in \pr{0,1}$,
\begin{align*}
\frac{1 - \be_i^2}{\al_i} \int_{\BB{H}_1}  V_i^2 y_0^a 
&\le \pr{1 - \be_i^2} \int_{\BB{H}_1} \abs{\gr _\te V_i}^2 y_0^a  
\end{align*}
and 
\begin{align*}
\frac{2 \be_i}{\sqrt{\al_i}} \int_{\BB{H}_1} V_i \pr{ \gr V_i \cdot y} y_0^a 
&\le 2 \pr{ \frac{\be_i^2}{ \al_i} \int_{\BB{H}_1} V_i^2 y_0^a }^{\frac 1 2} \pr{ \int_{\BB{H}_1} \abs{ \gr V_i \cdot y}^2 y_0^a }^{\frac 1 2}
\le \frac{\be_i^2}{\al_i} \int_{\BB{H}_1} V_i^2 y_0^a
+ \int_{\BB{H}_1} \pr{ \gr V_i \cdot \nu}^2 y_0^a \\
&\le \int_{\BB{H}_1} \brac{\be_i^2\abs{\gr _\te V_i}^2 + \pr{ \gr V_i \cdot \nu}^2} y_0^a.
\end{align*}
If for $i = 1,2$ we choose $\disp \beta_i = \frac{\sqrt{\alpha_i}}{2}\left\{\brac{(N - 1+ a)^2+\frac{4}{\alpha_i}}^{\frac{1}{2}}-(N - 1+ a)\right\}$, then
\begin{equation*}
\frac{1 - \be_i^2}{\al_i} =  \frac{\be_i \pr{N - 1+ a}}{\sqrt{\al_i}},
\end{equation*}
and it follows from \eqref{usenegativepart} that
\begin{align*}
\int_{\BB{D}_1} |\nabla V_i|^2 \Ga \,y_0^a
- \int_{\BB{D}_1}K_i^- V_i \Ga \, y_0^a
&\le \frac{N - 1+a} {2} \int_{\BB{H}_1} V_i^2 y_0^a 
+ \int_{\BB{H}_1} V_i \pr{\gr V_i \cdot \nu} y_0^a
\le \frac{\sqrt{\alpha_i}}{2\beta_i}\int_{\BB{H}_1}|\nabla V_i|^2 y_0^a.
\end{align*}
We substitute this inequality into \eqref{ACF'} with $r=1$ to obtain 
\begin{align*}
\phi'(1)
&\ge  \pr{\frac{2 \be_1}{\sqrt{\al_1}} + \frac{2 \be_2}{\sqrt{\al_2}} - 1+a }\pr{\int_{\BB{D}_1} |\nabla V_1|^2 \Ga \,y_0^a} 
\pr{\int_{\BB{D}_1}  |\nabla V_2|^2 \Ga \,y_0^a} \\
&- \frac{2 \be_1}{\sqrt{\al_1}} \pr{
\int_{\BB{D}_1} K_1^- V_1 \Ga \, y_0^a}
\pr{\int_{\BB{D}_1} |\nabla V_2|^2 \Ga \,y_0^a } 
- \frac{2 \be_2}{\sqrt{\al_2}} \pr{ \int_{\BB{D}_1} |\nabla V_1|^2 \Ga \,y_0^a } \pr{\int_{\BB{D}_1} K_2^- V_2 \Ga \, y_0^a}.
\end{align*}

It is known that for every choice of $S_1$ and $S_2$, $\disp \frac{\beta_1}{\sqrt{\alpha_1}}+\frac{\beta_2}{\sqrt{\alpha_2}} \ge \max\set{\frac{1-a}{2}, \frac{1-2a}{2}}$ when $N=2$, and for $N\ge 3$, $\disp \frac{\beta_1}{\sqrt{\alpha_1}} + \frac{\beta_2}{\sqrt{\alpha_2}} \ge \frac{1-a}{2}$; see \cite{TTS18}, \cite{TVZ14}, \cite{TVZ2}, \cite{TV18}. 
Therefore, $\disp \frac{2 \be_1}{\sqrt{\al_1}} + \frac{2 \be_2}{\sqrt{\al_2}} \ge 1 - a$ and we get that
\begin{align*}
\phi'(1)
&\ge - \frac{2 \be_1}{\sqrt{\al_1}} \pr{
\int_{\BB{D}_1} K_1^- V_1 \Ga \, y_0^a}
\pr{\int_{\BB{D}_1} |\nabla V_2|^2 \Ga \, y_0^a }
- \frac{2 \be_2}{\sqrt{\al_2}} \pr{ \int_{\BB{D}_1} |\nabla V_1|^2 \Ga \, y_0^a } \pr{\int_{\BB{D}_1} K_2^- V_2 \Ga \, y_0^a}.
\end{align*}
Since
\begin{align*}
\frac{\beta_i}{\sqrt{\alpha_i}}
&= \frac{N - 1+a}{2}\left\{\brac{1+\frac{4}{\alpha_i(N - 1+ a)^2}}^{\frac{1}{2}}-1\right\}
\le \frac{N - 1+ a}{2}\frac{2}{\alpha_i(N - 1+ a)^2}
\le \frac{1}{N - 1+a} \frac{\int_{\BB{H}_1} |\nabla V_i|^2 y_0^a}{\int_{\BB{H}_1} V_i^2 y_0^a},
\end{align*}
then
\begin{align*}
\phi'(1)
&\ge - \frac{2}{N - 1+ a} \frac{\pr{
\int_{\BB{D}_1} K_1^- V_1 \Ga \, y_0^a}
\pr{\int_{\BB{D}_1} |\nabla V_2|^2 \Ga \, y_0^a} \pr{\int_{\BB{H}_1} |\nabla V_1|^2 y_0^a}}{ \pr{\int_{\BB{H}_1} V_1^2 y_0^a}} \\
&- \frac{2}{N - 1+a} \frac{\pr{\int_{\BB{D}_1} |\nabla V_1|^2 \Ga \, y_0^a} \pr{\int_{\BB{D}_1} K_2^- V_2 \Ga \, y_0^a} \pr{\int_{\BB{H}_1} |\nabla V_2|^2 y_0^a}}{ \pr{\int_{\BB{H}_1} V_2^2 y_0^a}}.
\end{align*}

Recall that $\phi(r; V_1, V_2)$ is well-defined and bounded and that \eqref{ACF'} holds for a.e. $r$.
Choose such an $r\neq 1$ and assume additionally that $\disp \pr{\int_{\BB{H}_r} V_{1}^2 y_0^a} \pr{\int_{\BB{H}_r} V_{2}^2 y_0^a} \neq 0$.
Define $V_{i,r}(Y)=r^{-\frac{1-a}{4}} V_i(rY)$. 
For $0< \rho,r<1$, observe that
\[
\phi(\rho;V_{1,r},V_{2,r})
=\frac{1}{\rho^{1-a}} \pr{\int_{\BB{D}_\rho}  |\nabla V_{1,r}|^2 \Ga \, y_0^a} \pr{\int_{\BB{D}_\rho} |\nabla V_{2,r}|^2 \Ga \, y_0^a}
= \phi(r\rho;V_1,V_2).
\]
Thus
\[
\frac{d}{d\rho} \phi(\rho, V_{1,r}, V_{2,r}) = r \phi'(r\rho; V_1, V_2)
\]
and taking $\rho=1$ gives
\[
\phi'(1,V_{1,r},V_{2,r})=r\phi'(r;V_1,V_2).
\]

With $K_{i, r}\pr{Y} =r^{2-\frac{1-a}{4}}K_i(rY)$, we have from \eqref{strangecondition} that for $\Psi\in C_0^{2}(\BB{B}_{R/r})$ with $\Psi\ge 0$,
\begin{align*}
\int_{\R^{N+1}_+} \brac{ \nabla V_{i,r} \cdot \nabla (V_{i,r}\Psi) + K_{i,r}(V_{i,r}\Psi)} y_0^a ,
&\le 0
\end{align*}
which shows that $K_{i, r}$ is a suitable rescaling of $K_i$.
Applying the derivative estimates above to the pair $V_{1, r}$, $V_{2, r}$ shows that 
\begin{align*}
\phi'(r)
&=\frac{1}{r}\phi'(1;V_{1,r},V_{2,r}) \\
&\ge  - \frac{2}{r(N - 1+a)} \frac{\pr{
\int_{\BB{D}_1}K_{1,r}^- V_{1,r} \Ga \, y_0^a}
\pr{\int_{\BB{D}_1} |\nabla V_{2,r}|^2 \Ga \, y_0^a} \pr{\int_{\BB{H}_1} |\nabla V_{1,r}|^2 y_0^a} }{\pr{\int_{\BB{H}_1} V_{1,r}^2 y_0^a}} \\
&- \frac{2}{r(N - 1+a)} \frac{\pr{\int_{\BB{D}_1} |\nabla V_{1,r}|^2 \Ga \, y_0^a} \pr{\int_{\BB{D}_1} K_{2,r}^- V_{2,r} \Ga \, y_0^a} \pr{\int_{\BB{H}_1} |\nabla V_{2,r}|^2 y_0^a} }{\pr{\int_{\BB{H}_1} V_{2,r}^2 y_0^a}} ,
\end{align*}
which, after rescaling, is the result described by \eqref{phiTDer}.

\end{proof}

\subsection{Alt-Caffarelli-Friedman frequency functions in the parabolic setting}
\label{SS:ParaACF}

Here we prove our degenerate parabolic ACF monotonicity formula. 
As a first step, we establish that the transformation maps, $F_n$, relate subsolutions of the degenerate parabolic equation to subsolutions of the related nonhomogeneous degenerate elliptic equations.
In the elliptic setting, the appropriate notion of a subsolution is described by \eqref{strangecondition}.
For the parabolic setting, our notion of a subsolution is given below in \eqref{superweakParab}.

Throughout this section, we think of $F_{n}$ as being defined in the whole space $\R^{dn+1}$.

\begin{lem}[Subsolutions relationships]
\label{weaktosuperweak}
For $U(X,t)$ defined on $\R^{d+1}_+ \times (0,T)$, let $V_n : \BB{D}^{dn+1}_{\sqrt{MT}} \to \R$ be given by $V_n \pr{Y} = U\pr{F_{n}\pr{Y}}$.
For $\xi \in C^2\pr{\R^{d+1} \times (0, T)}$, 
define $\Psi_n : \R^{dn+1} \to \R$ so that $\Psi_n \pr{Y} = \xi\pr{F_{n}\pr{Y}}$.
Assume that $\xi$ is chosen so that $\Psi_n \in C^2_0\pr{\BB{B}_{\sqrt{MT}}}$.
If for each $t \in (0, T)$, $\mathcal{H}(t)$, $\mathcal{D}(t)$, $\mathcal{T}(t)$, and $\mathcal{M}(t)$ all belong to $L^1\pr{(0, t)}$, then with $K_n$ as in \eqref{KnDef}, it holds that
\begin{align*}
\int_{\R^{dn + 1}_+} \brac{\gr V_n \cdot \gr \pr{V_n \Psi_n} + K_n V_n \Psi_n} y_0^{a} \, dY 
&= \frac{\bar{C} n\sqrt M}{2} \int_0^T \int_{\R^{d+1}_+} \brac{ \gr U \cdot \gr (U \xi\mathcal{G}_n) - \del_t U (U \xi\mathcal{G}_n) } t^{\frac{dn-1+a}{2}} x_0^a dX d{t} .
\end{align*}
\end{lem}

\begin{proof}
From \eqref{VyijDeriv}, we get that
\begin{align*}
\gr_Y V_n \cdot \gr_Y \pr{V_n \Psi_n} 
&= n \gr_X U \cdot \gr_X \pr{U \xi}
+ \frac 2 M \del_t U X \cdot \gr_X \pr{U \xi}
+ \frac{4}{M} t \del_t U \del_t \pr{U \xi}
+ \frac 2 M X \cdot \gr_X U \del_t \pr{U\xi} .
\end{align*}
Since $\supp \Psi_n \su \BB{B}_{\sqrt{MT}}^{dn+1}$, then an application of Lemma \ref{PFT} shows that
\begin{equation}
\label{solutionTransf}
\begin{aligned}
&\int_{\R^{dn + 1}_+} \brac{\gr V_n \cdot \gr \pr{V_n \Psi_n} + K_n V_n \Psi_n} y_0^{a} \, dY 
= \int_{\BB{D}^{dn + 1}_{\sqrt{MT}}} \brac{\gr V_n \cdot \gr \pr{V_n \Psi_n} + K_n V_n \Psi_n} y_0^{a} \, dY \\
&= \frac{\bar{C} \sqrt M}{2}\int_0^T \int_{\R^{d+1}_+} \brac{n \pr{\gr U \cdot \gr \pr{U\xi}}
+ \frac 4 M \pr{\pr{X, t} \cdot \gr_{(X,t)} \del_t U} U \xi
} t^{\frac{dn-1+a}{2}} \dGnt d{t} \\
&+ \frac{\bar{C}}{\sqrt M} \int_0^T \int_{\R^{d+1}_+} \brac{ \del_t U \pr{X \cdot \gr\pr{U \xi}}
+ 2 t \del_t U \del_t \pr{U\xi}
+ \pr{X \cdot \gr U} \del_t \pr{U\xi} } t^{\frac{dn-1+a}{2}} \dGnt dt,
\end{aligned}
\end{equation}
where the assumptions on $\mathcal{H}(t)$, $\mathcal{D}(t)$, $\mathcal{T}(t)$, and $\mathcal{M}(t)$ ensure that the righthand side is finite.
Note that because $\Psi_n \equiv 0$ in a neighborhood of $\del \BB{B}_{\sqrt{MT}}$, then $\xi \equiv 0$ in a neighborhood of $t = T$.

To remove the derivatives from $U \xi$, we integrate each of the terms in the last line of \eqref{solutionTransf} by parts.
Since $\mathcal{G}_n(\cdot, t)$ is compactly supported in $\R^{d+1}$ for each $t \in (0, T)$, then when we integrate by parts with respect to $X$, the only boundary terms come from the spatial boundary $\set{x_0 = 0}$.
Thus, for the first term on the last line of \eqref{solutionTransf}, we get
\begin{equation}
\label{partialXjust}
\begin{aligned}
&\int_0^T \int_{\R^{d+1}_+} \del_t U \pr{X \cdot \gr\pr{U \xi}} t^{\frac{dn-1+a}{2}} \dGnt d{t} \\
&= - \lim_{\eps \to 0^+} \int_0^T \int_{\R^{d+1}_+ \cap \set{x_0 = \eps}} U \del_t U \xi \, t^{\frac{dn-1+a}{2}} x_0^{1+a} \mathcal{G}_n(X, t) dx dt \\
&- \lim_{\eps \to 0^+} \int_0^T t^{\frac{dn-1+a}{2}}\int_{\R^{d+1}\cap \set{x_0 > \eps}} U \brac{\pr{ X \cdot \gr \del_t U} \xi \,  x_0^a \mathcal{G}_n + \del_t U  \xi \di \pr{X x_0^a \mathcal{G}_n} } dX dt.
\end{aligned}
\end{equation}
With the first term on the right, we change the order of integration, integrate by parts with respect to $t$, then change the order of integration back to get
\begin{align*}
&-2 \int_0^T \int_{\R^{d+1}_+ \cap \set{x_0 = \eps}} U \del_t U \xi \, t^{\frac{dn-1+a}{2}} x_0^{1+a} \mathcal{G}_n(X, t) dx dt
= - \int_{\R^{d+1}_+ \cap \set{x_0 = \eps}} \int_0^T \del_t \pr{U^2} \xi \, t^{\frac{dn-1+a}{2}} x_0^{1+a} \mathcal{G}_n(X, t) dt dx \\
&= \int_{\R^{d+1}_+ \cap \set{x_0 = \eps}} \int_0^T U^2 \del_t\pr{ \xi \, t^{\frac{dn-1+a}{2}} x_0^{1+a} \mathcal{G}_n(X, t)} dt dx
= \int_0^T \int_{\R^{d+1}_+ \cap \set{x_0 = \eps}}  U^2 \del_t\pr{ \xi \, t^{\frac{dn-1+a}{2}} x_0^{1+a} \mathcal{G}_n(X, t)} dx dt,
\end{align*}
where we use that $\supp \mathcal{G}_n(\eps, \cdot, 0) = \emptyset$ and $\supp \xi(\cdot, T) = \emptyset$ to eliminate the boundary terms at $t = 0$ and $t = T$, respectively.
Since 
\begin{align}
\label{derivComp1}
\del_t  \pr{t^{\frac{dn-1+a}{2}} \mathcal{G}_n(X,t)}
&= \frac{dn-d-2} {2t} \pr{1 - \frac{\abs{X}^2}{M n t}}^{-1} t^{\frac{dn-1+a}{2}} \mathcal{G}_n(X,t),
\end{align}
then by \eqref{pointwiseGBound} and that $\xi$ is $C^2$ and compactly supported, we have
\begin{align*}
\abs{ \int_{\R^{d+1}_+ \cap \set{x_0 = \eps}} U \del_t U \xi \, x_0^{1+a} \mathcal{G}_n(X, t) dx}    
&\le \int_{\R^{d+1}_+ \cap \set{x_0 = \eps}} U^2 \abs{\del_t \xi + \xi \frac{dn-d-2} {2t}  \pr{1 - \frac{\abs{X}^2}{M n t}}^{-1}} x_0^{1+a} \mathcal{G}_n(X, t) dx \\
&\lesssim \int_{\BB{B}_{\sqrt{Mnt}} \cap \set{x_0 = \eps}} U^2 x_0^{1+a} \mathcal{G}(X, t) dx,
\end{align*}
where we recall the support of $\mathcal{G}_n$.
For $\eps > 0$, set $R = \sqrt{Mnt} + 2\eps$ and let $\zeta_\eps \in C^\iny_0(\BB{B}_{R'})$ be a nonnegative cutoff function that is supported on the $2\eps$-neighborhood of $\BB{B}_{\sqrt{Mnt}} \cap \set{x_ 0 = \eps}$ with $\zeta_\eps \equiv 1$ on the $\eps$-neighborhood of $\BB{B}_{\sqrt{Mnt}} \cap \set{x_ 0 = \eps}$.
Then
\begin{align*}
& \int_{\BB{B}_{\sqrt{Mnt}} \cap \set{x_ 0 = \eps}} U^2 x_0^{1+a} \mathcal{G}(X,t) dx
\le 
\int_{\BB{B}_{R} \cap \set{x_ 0 = \eps}} \zeta_\eps U^2 x_0^{1+a} \mathcal{G}(X,t) dx
= - \int_{\BB{B}_R \cap \set{x_ 0 > \eps}} \del_{x_0} \pr{\zeta_\eps U^2 x_0^{1+a} \mathcal{G}} dX \\
&= - \int_{\BB{B}_R \cap \set{x_ 0 > \eps}} \brac{x_0 \del_{x_0} \zeta_\eps + \zeta_\eps \pr{1 + a - \frac{x_0^2}{2t}}}  U^2 \dGt
- \int_{\BB{B}_R \cap \set{x_ 0 > \eps}}  2 x_0 \zeta_\eps U \del_{x_0} U \dGt \\
&\lesssim \int_{\BB{B}_R \cap \set{\eps < x_ 0 < 3 \eps}} \pr{U^2 + \abs{\gr U}^2} \dGt,
\end{align*}
where we have used Young's inequality and that $x_0 \abs{\gr \zeta_\eps} \lesssim 1$ because $\abs{\gr \zeta_\eps} \lesssim \eps^{-1}$.
Since $\mathcal{H}(t)$ and $\mathcal{D}(t)$ are well-defined, then the right-hand side goes to zero in the limit and we may conclude that
\begin{align*}
\lim_{\eps \to 0^+} \int_0^T \int_{\R^{d+1}_+ \cap \set{x_0 = \eps}} U \del_t U \xi \, t^{\frac{dn-1+a}{2}} x_0^{1+a} \mathcal{G}_n(X, t) dx dt = 0.
\end{align*}
Therefore, from \eqref{partialXjust},
\begin{equation}
\label{intbp1}
\begin{aligned}
&\int_0^T \int_{\R^{d+1}_+} \del_t U \pr{X \cdot \gr\pr{U \xi}} t^{\frac{dn-1+a}{2}} \dGnt d{t} \\
&= - \int_0^T \int_{\R^{d+1}_+} U \brac{\pr{ X \cdot \gr \del_t U} x_0^a \mathcal{G}_n + \del_t U  \di \pr{X x_0^a \mathcal{G}_n} }\xi  t^{\frac{dn-1+a}{2}} dX dt.
\end{aligned}
\end{equation}

We now return to the other terms in the last line of \eqref{solutionTransf}.
For any $\eps > 0$, let $\mu_\eps \in C^\iny(\R^+)$ be supported in $\R_{\ge \eps}$ with $\mu_\eps(t) \equiv 1$ on $t > 2 \eps$ and $\abs{\mu_\eps'(t)} \lesssim \eps^{-1}$ on $\brac{\eps, 2 \eps}$.
By integrability and Fubini, we get
\begin{align*}
\int_0^T \int_{\R^{d+1}_+} t \del_t U \del_t \pr{U\xi}  t^{\frac{dn-1+a}{2}} x_0^a \mathcal{G}_n(X,t) dX dt 
&= \lim_{\eps \to 0^+} \int_0^T \int_{\R^{d+1}_+} t \del_t U  \del_t \pr{U\xi}  t^{\frac{dn-1+a}{2}} \mu_\eps(t) x_0^a \mathcal{G}_n(X,t) dX dt \\
&= \lim_{\eps \to 0^+} \int_{\R^{d+1}_+} \int_0^T t \del_t U \del_t \pr{U\xi}  t^{\frac{dn-1+a}{2}} \mu_\eps(t) x_0^a \mathcal{G}_n(X,t) dt dX.
\end{align*}
An integration by parts applied to the inner integral then shows that
\begin{align*}
\int_0^T t \del_t U \del_t \pr{U\xi}  t^{\frac{dn-1+a}{2}} \mu_\eps(t) x_0^a \mathcal{G}_n(X,t) dt
&= - \int_0^T U\xi  \del_t \pr{\del_t U  t^{1 + \frac{dn-1+a}{2}} \mathcal{G}_n(X,t)} x_0^a \mu_\eps(t)  dt \\
&- \int_0^T \xi \mu_\eps'(t) U \del_t U t^{\frac{dn+1+a}{2}}  x_0^a \mathcal{G}_n(X,t) dt,
\end{align*}
where there are no boundary terms because $\mu_\eps$ vanishes in a neighborhood of $t = 0$, while $\xi$ vanishes in a neighborhood of $t = T$.
For the last term, we apply Fubini followed by \eqref{pointwiseGBound}, the bound on $\mu_\eps'$, and an application of H\"older's inequality to get that
\begin{align*}
&\abs{\int_{\R^{d+1}_+} \int_0^T \xi \mu_\eps'(t) U \del_t U t^{\frac{dn+1+a}{2}}  x_0^a \mathcal{G}_n(X,t) dt dX}
= \abs{\int_{\eps}^{2\eps} \mu_\eps'(t) t^{\frac{dn+1+a}{2}} \int_{\R^{d+1}_+} \xi U \del_t U  x_0^a \mathcal{G}_n(X,t) dX dt } \\
&\lesssim C_0 \eps^{\frac{dn-1+a}{2}} \norm{\xi}_{L^\iny\pr{\BB{K}_{n}(T)}} \pr{\int_{\eps}^{2\eps} \mathcal{H}(t) dt}^{\frac 1 2} \pr{\int_{\eps}^{2\eps} \mathcal{T}(t) dt}^{\frac 1 2},
\end{align*}
where $\BB{K}_{n}(T) := \set{(X, t) \in \R^{d+1}_+ \times (0, T) : \abs{X}^2 \le M n t}$ corresponds to the space-time support of $\mathcal{G}_n$.
Since this integral vanishes in the limit as $\eps \to 0^+$, then we may conclude that
\begin{align}
\label{intbp2}
\int_0^T \int_{\R^{d+1}_+} 2t \del_t U \del_t \pr{U\xi}  t^{\frac{dn-1+a}{2}} x_0^a \mathcal{G}_n(X,t) dX dt 
&= - \int_0^T \int_{\R^{d+1}_+} 2 U \del_t \pr{ \del_t U  t^{\frac{dn+1+a}{2}} \mathcal{G}_n(X,t)}\xi x_0^a dX dt.
\end{align}
Integrating the last term from \eqref{solutionTransf} by parts in a similar way, then putting it together with \eqref{intbp1} and \eqref{intbp2} shows that
\begin{align*}
& \int_0^T \int_{\R^{d+1}_+} \brac{ \del_t U \pr{X \cdot \gr\pr{U \xi}}
+ 2 t \del_t U \del_t \pr{U\xi}
+ \pr{X \cdot \gr U} \del_t \pr{U\xi} } t^{\frac{dn-1+a}{2}} \dGnt d{t} \\
&= - \int_0^T \int_{\R^{d+1}_+} U \brac{\pr{ X \cdot \gr \del_t U} x_0^a \mathcal{G}_n + \del_t U  \di \pr{X x_0^a \mathcal{G}_n} } \xi t^{\frac{dn-1+a}{2}} dX dt \\
&- 2 \int_0^T \int_{\R^{d+1}_+} U t \del_t^2 U \xi t^{\frac{dn-1+a}{2}} x_0^a \mathcal{G}_n dX d{t}
- 2 \int_0^T \int_{\R^{d+1}_+} U \del_t U \xi \del_t \pr{t^{1+\frac{dn-1+a}{2}} \mathcal{G}_n}  x_0^a dX d{t}\\
&- \int_0^T \int_{\R^{d+1}_+} U \pr{ X \cdot \gr \del_t  U} \xi  t^{\frac{dn-1+a}{2}} x_0^a \mathcal{G}_n dX d{t} 
- \int_0^T \int_{\R^{d+1}_+} U \pr{ X \cdot \gr U} \xi  \del_t \pr{t^{\frac{dn-1+a}{2}}\mathcal{G}_n}  x_0^a dX d{t} \\
&= - 2\int_0^T \int_{\R^{d+1}_+} U \pr{\pr{X, t} \cdot \gr_{(X,t)} \del_t U} \xi t^{\frac{dn-1+a}{2}} x_0^a \mathcal{G}_n dX d{t}
+ \frac{Mn}{2} \int_0^T \int_{\R^{d+1}_+} U \gr U \cdot \gr \mathcal{G}_n \xi t^{\frac{dn-1+a}{2}} x_0^a dX d{t} \\
&- \frac{Mn}{2} \int_0^T \int_{\R^{d+1}_+} U \del_t U  \xi t^{\frac{dn-1+a}{2}} x_0^a \mathcal{G}_n dX d{t},
\end{align*}
where we have used \eqref{MDef}, the computations above in \eqref{derivComp1}, and
\begin{align*}
\gr \mathcal{G}_n(X,t)
&= - \frac{2X}{Mn} \frac{dn-d-2}{2t} \pr{1 - \frac{\abs{X}^2}{Mnt}}^{-1} \mathcal{G}_n(X,t) ,
\end{align*}
to get that
\begin{align*}
- X \del_t\pr{t^{\frac{dn-1+a}{2}} \mathcal{G}_n} 
&= \frac{Mn}{2} t^{\frac{dn-1+a}{2}} \gr \mathcal{G}_n
\end{align*}
and
\begin{align*}
t^{\frac{dn-1+a}{2}} \di \pr{X x_0^a \mathcal{G}_n}
+ 2 \del_t  \pr{t^{1+ \frac{dn-1+a}{2}} x_0^a \mathcal{G}_n}
&= \frac{Mn}{2} t^{\frac{dn-1+a}{2}} x_0^a \mathcal{G}_n,
\end{align*}
then simplified. 
Substituting these expressions into \eqref{solutionTransf} gives
\begin{equation*}
\begin{aligned}
&\int_{\R^{dn + 1}_+} \brac{\gr V_n \cdot \gr \pr{V_n \Psi_n} + K_n V_n \Psi_n} y_0^{a} \, dY \\
&= \frac{\bar{C} n \sqrt M}{2} \brac{\int_0^T  \int_{\R^{d+1}_+} \pr{\gr U \cdot \gr \pr{U\xi}} t^{\frac{dn-1+a}{2}} \dGnt d{t} 
+ \int_0^T \int_{\R^{d+1}_+} \gr U \cdot \gr \mathcal{G}_n U \xi t^{\frac{dn-1+a}{2}} x_0^a dX d{t}} \\
&- \frac{\bar{C} n \sqrt M}{2} \int_0^T \int_{\R^{d+1}_+} U \del_t U  \xi t^{\frac{dn-1+a}{2}} x_0^a \mathcal{G}_n dX d{t},
\end{aligned}
\end{equation*}
as required.
\end{proof}

Before stating the main theorem of this section, we define our ACF class of functions.

\begin{defn}[ACF function class]
We say that a function $U = U(X, t)$ defined on $\R^{d+1}_{\ge 0} \times \brac{0, T}$ belongs to the function class $\mathfrak{C}\pr{\R^{d+1}_+ \times \pr{0, T}}$ if 
\begin{itemize}
    \item $U$ is continuous on $\R^{d+1}_{\ge 0}\times [0,T]$,
    \item $U$ has moderate $a$-growth at infinity (see Definition \ref{aGrowth}),
\item for every $t_0 \in \pr{0, T}$
\begin{equation}
\label{ACFClass2}   
\mathcal{M} \in L^{1}\pr{\pr{0, t_0}},
\end{equation}
\item for every $t_0 \in \pr{0, T}$, there exists $\eps \in \pr{0, t_0}$ so that
\begin{equation}
\label{ACFClass1}
\mathcal{D}, \mathcal{T} \in L^\iny\pr{\brac{t_0 - \eps, t_0}}.
\end{equation}
\end{itemize}
Recall that these functionals are introduced in \eqref{12functionals}, see Definition \ref{functionalDefs}.    
\end{defn}

\begin{thm}[Parabolic ACF Monotonicity]
\label{T:parabolicACF}
For $i = 1, 2$, let $U_i = U_i\pr{X,t} \in \mathfrak{C}\pr{\R^{d+1}_+ \times \pr{0, T}}$ be non-negative functions with $U_i(0, 0)=0$ and $U_1U_2=0$ in $\{x_0=0\}\times \R^{d} \times (0,T)$. 
Assume that each $U_i$ is a subsolution in the sense that for all $\Upsilon \in C^2\pr{\R^{d+1} \times (0, T)}$ with each $\Upsilon(\cdot, t)$ compactly supported and $\Upsilon \ge 0$, it holds that
\begin{equation}
\label{superweakParab}
\int_0^T \int_{\R^{d+1}_+} \brac{\nabla U_i \cdot \nabla \pr{U_i\Upsilon} - \del_t U_i \pr{U_i \Upsilon}} x_0^a \, dX dt \le 0.  
\end{equation}
For almost every $t \in \pr{0, T}$, the function $\Phi(t) = \Phi(t; U_1, U_2)$ given by
\[
\Phi(t)
= \frac{1}{t^{\frac{1-a}{2}}} \prod_{i=1}^2\int_0^t \int_{\R^{d+1}_+} |\nabla U_i(X, \tau)|^2 x_0^a \mathcal{G}(X,\tau) dX d\tau=\frac{1}{t^{\frac{1-a}{2}}}\prod_{i=1}^2\int_0^t\mathcal{D}(\tau,U_i)d\tau
\]
is monotonically non-decreasing in $t$.
\end{thm}

\begin{proof}
Let $U_i : \R^{d+1}_+ \times \pr{0, T} \to \R$ be as given in the theorem.
For each $n \in \N$, define $V_{i,n} : \BB{D}^{dn+1}_{\sqrt{MT}} \to \R$ by
$$V_{i,n} \pr{Y} = U_i\pr{F_n\pr{Y}} = U_i\pr{X, t}.$$
Since each $U_i$ is non-negative with $U_i(0,0) = 0$, then each $V_{i,n}$ is non-negative with $V_{i,n}(0) = 0$.
The continuity of $U_i$ on $\R^{d+1}_{\ge 0} \times [0, T]$ implies that each $V_{i,n}$ belongs to $ C^0\pr{\overline{\BB{D}}_R^{dn+1}}$, while the assumption that each $U_i$ has moderate $a$-growth at infinity shows that each $V_{i,n} \in H^1\pr{\BB{D}_R^{dn+1}, y_0^a}$, see Lemma \ref{functionClassLemma}.
Because $U_1 U_2 \equiv 0$ in $\{x_0=0\}\times \R^{d} \times (0,T)$, then for each $n \in \N$, $V_{1,n} V_{2,n} \equiv 0$ in $\BB{B}^{dn+1}_{\sqrt{MT}} \cap \set{y_0 = 0}$ as well.

Set $J_i(X, t) = \pr{X,t} \cdot \gr_{(X,t)} \del_t U_i$ and for each $n \in \N$, define $K_{i,n} : \BB{D}_{\sqrt{MT}}^{dn+1} \to \R$ by
$$K_{i,n} \pr{Y} = \frac 4 M J_i\pr{F_n\pr{Y}}.$$
For each $t \in (0, T)$, applications of \eqref{ballKnVn} and \eqref{ballKnVnWeighted} show that
$$\int_{\BB{D}^{dn+1}_{\sqrt{Mt}}} V_n K_n y_0^a
= \frac{2\bar{C}}{\sqrt M} \int_{0}^{t} \tau^{\frac{dn-1+a}2} \mathcal{M}_n(\tau; U_i) d\tau$$
and
$$\int_{\BB{D}_{\sqrt{Mt}}^{dn+1}}K_{i,n}^- V_{i,n}|Y|^{-(dn-1+a)}y_0^a
= \frac{M \aSp{dn+a}}{4 \ga_{a}} \int_0^t {\mathcal{M}^-_n(\tau; U_i)} d\tau,$$
respectively.
Since $\mathcal{M}(\cdot; U_i) \in L^1\pr{(0, t)}$, then both of the above integrals are finite and each $K_i$ satisfies the preliminary assumptions of Theorem \ref{C:ACFellipticnonzero}.

Let $\Psi_n \in C^2_0\pr{\mathbb{B}^{dn+1}_{\sqrt{MT}}}$ be a non-negative radial function, a test function for each $V_{i,n}$; see Remark \ref{radialTestFunc}.
We abuse notation and write $\Psi_n(Y) = \Psi_n(|Y|)$.
With $\xi(t) := \Psi_n\pr{\sqrt{Mt}}$, it holds that $\xi \in C^2\pr{(0, T)}$ is non-negative and $\Psi_n(Y) = \xi\pr{F_n(Y)}$.
In particular, Lemma \ref{weaktosuperweak} is applicable and shows that
\begin{align*}
\frac{2}{\bar{C} n\sqrt M} \int_{\R^{dn + 1}_+} \brac{\gr V_{i, n} \cdot \gr \pr{V_{i,n} \Psi_n} + K_{i,n} V_{i,n} \Psi_n} y_0^{a}
&= \int_0^T \int_{\R^{d+1}_+} \brac{ \gr U_i \cdot \gr (U _i\xi\mathcal{G}_n) - \del_t U_i (U_i \xi\mathcal{G}_n) } t^{\frac{dn-1+a}{2}} x_0^a dX d{t} \\
&= \int_0^T \int_{\R^{d+1}_+} \brac{ \gr U_i \cdot \gr (U _i \Upsilon_n) - \del_t U_i (U_i \Upsilon_n) } x_0^a dX d{t},
\end{align*}
where we have introduced $\Upsilon_n(X, t) := \xi(t) t^{\frac{dn-1+a}{2}} \mathcal{G}_n(X,t)$.

Since $\Upsilon_n \in C^2\pr{\R^{d+1} \times (0, T)}$, $\Upsilon_n(\cdot, t)$ is compactly supported for each $t \in (0, T)$, and $\Upsilon_n \ge 0$, then \eqref{superweakParab} holds with $\Upsilon = \Upsilon_n$.
That is,
\begin{equation*}
\int_0^T \int_{\R^{d+1}_+} \brac{\nabla U_i\cdot \nabla \pr{U_i \xi \mathcal{G}_n} - \del_t U_i \pr{U_i \xi \mathcal{G}_n}} t^{\frac{dn-1+a}{2}} x_0^a \, dX dt \le 0
\end{equation*}
so that
$$\int_{\R^{dn + 1}_+} \brac{\gr V_{i, n} \cdot \gr \pr{V_{i,n} \Psi_n} + K_{i,n} V_{i,n} \Psi_n} y_0^{a} \le 0.$$

As $\Psi_n$ was an arbitrary test function (according to Remark \ref{radialTestFunc}), then each $V_{i,n}$ is a subsolution in the sense of \eqref{strangecondition}.
In particular, we have shown that Theorem \ref{C:ACFellipticnonzero} is applicable for every $n \in \N$ with $V_{i, n}$ and $K_{i,n}$ in place of $V_i$ and $K_i$, respectively.

With $\phi(r)$ is given in \ref{ACF}, define $\disp \Phi_n(t) = M^{\frac{1-a}{2}} \pr{\frac{4\gamma_a}{ n M|\mathbb{S}^{dn+a}|}}^2\phi\pr{\sqrt{Mt}; V_{1, n}, V_{2, n}}$.
Theorem \ref{C:ACFellipticnonzero} shows that $\Phi_n(t)$ is well-defined and bounded for a.e. $t \in (0, T)$.
For such values of $t$, \eqref{ballgradVnWeighted} shows that
\begin{align*}
\Phi_n(t)
&=\pr{\frac{4\gamma_a}{ n M|\mathbb{S}^{dn+a}|}}^2 \frac{M^{\frac{1-a}{2}}}{(Mt)^{\frac{1-a}{2}}}\prod_{i=1}^2\int_{\BB{D}^{dn+1}_{\sqrt{Mt}}} |\nabla V_{i,n}(Y)|^2 |Y|^{-(dn-1+a)} y_0^a \, dY \\
&= \frac{1}{t^{\frac{1-a}{2}}}\prod_{i=1}^2 \int_0^t \int_{\R^{d+1}_+} \brac{ \abs{\gr U_i}^2 + \frac {4}{Mn} \pr{\tau \abs{\del_\tau U_i}^2 + \pr{X \cdot \gr U_i} \del_\tau U_i}} \dGnta d\tau.
\end{align*}
It follows from Lemma \ref{GaussianLimit} that
\[
\lim\limits_{n\rightarrow\infty}\Phi_n(t)
= \frac{1}{t^{\frac{1-a}{2}}}\prod_{i=1}^2\int_0^t \int_{\R^{d+1}_+} |\nabla U_i|^2 x_0^a \mathcal{G}(X,\tau) dX d\tau
= \Phi(t).
\]

If we assume that $\disp \pr{\int_{\R^{d+1}_+}  \abs{U_1(X,t)}^2 \dGt} \pr{\int_{\R^{d+1}_+}  \abs{U_2(X,t)}^2 \dGt} \ne 0$, then for $n \in \N$ sufficiently large, $\disp \pr{\int_{\R^{d+1}_+}  \abs{U_1(X,t)}^2 \dGnt} \pr{\int_{\R^{d+1}_+}  \abs{U_2(X,t)}^2 \dGnt} \ne 0$ as well.
It then follows from \eqref{HnExpr} that $\disp \pr{\int_{\BB{H}^{dn}_{\sqrt{M t}}} \abs{V_{1,n}\pr{Y}}^2 y_0^a } \pr{\int_{\BB{H}^{dn}_{\sqrt{M t}}} \abs{V_{2,n}\pr{Y}}^2 y_0^a } \ne 0$.
That is, for such a choice of $t \in (0,T)$ and $n \in \N$ sufficiently large, we can apply Theorem \ref{C:ACFellipticnonzero} to each $V_{i,n}$ with $K_i = K_{i,n}$ on the half-ball $\BB{D}_{\sqrt{Mt}}$ to get that 

\begin{align*}
\phi'(\sqrt{Mt})
&\ge -\frac{2(Mt)^{\frac{a}{2}}}{{dn-1+a}} \sum_{\substack{i,j=1 \\ i\neq j}}^2
  \frac{\pr{\int_{\BB{D}^{dn+1}_{\sqrt{Mt}}} {K_{i,n}^-V_{i,n}}{\Ga \,} y_0^a}
\pr{\int_{\BB{D}^{dn+1}_{\sqrt{Mt}}} {|\nabla V_{j,n}|^2}{\Ga \,} y_0^a } \pr{\int_{\BB{H}^{dn}_{\sqrt{Mt}}} |\nabla V_{i,n}|^2 y_0^a}}{\pr{\int_{\BB{H}^{dn}_{\sqrt{Mt}}} V_{i,n}^2 y_0^a}},
\end{align*}
where we recall the notation $\Ga(Y) = |Y|^{-(dn-1+a)}$. 
Putting the expressions \eqref{HnExpr}, \eqref{sphereGradientSquared}, \eqref{ballKnVnWeighted} and \eqref{ballgradVnWeighted} from Lemma \ref{bridgeLemmaResults}  
into the previous inequality shows that
\begin{align*}
\phi'(\sqrt{Mt})
&\ge -\frac{8(Mt)^{\frac{a}{2}}}{{dn-1+a}} \pr{\frac{n M\aSp{dn+a}}{4\ga_{a}}}^2 \sum_{\substack{i,j=1 \\ i\neq j}}^2
  \int_0^t \mathcal{M}^-_n(\tau; U_j) d\tau \\
&\times \pr{\int_0^t \mathcal{D}_n(\tau; U_j) d\tau
+ \frac{4}{nM} \int_0^t \tau \mathcal{T}_n(\tau; U_j) d\tau }
\frac{\mathcal{D}_n(t; U_i) + \frac{4t}{Mn} \mathcal{T}_n(t; U_i)}{\mathcal{H}_n(t; U_i)}.
\end{align*}

By the chain rule,
 \[
\Phi_n'(t) = M^{\frac{1-a}{2}} \pr{\frac{4\gamma_a}{Mn|\mathbb{S}^{dn+a}|}}^2\phi'\pr{\sqrt{Mt}}\frac{\sqrt{M}}{2\sqrt{t}}.
\]
An application of \eqref{MDef} and the bound \eqref{pointwiseGBound} from Lemma \ref{GaussianLimit} shows that $\Phi_n'(t)  \ge - E_n(t)$, where
\begin{align*}
E_n(t) 
&:=  \frac{10 C_0^3}{n t^{\frac{1-a}{2}}} \sum_{\substack{i,j=1 \\ i\neq j}}^2 \frac{\pr{\int_0^t \mathcal{M}^-(\tau; U_i)  d\tau} 
\pr{ \int_0^t \mathcal{D}(\tau; U_j)  d\tau 
+ \frac{4}{nM} \int_0^t \tau \mathcal{T}(\tau; U_j) d\tau}\pr{\mathcal{D}(t; U_i)
+ \frac{4t}{n M} \mathcal{T}(t; U_i)}}{ \mathcal{H}_n(t; U_i) }.
\end{align*}

To show that $\Phi(t)$ is monotone non-decreasing, it suffices to show that given any $t_0 \in (0, T]$, there exists $\de \in \pr{0, t_0}$ so that $E_n$ converges uniformly to $0$ on $\brac{t_0 - \de, t_0}$.
Indeed, since $\frac{d}{dt}\Phi_n(t) \ge - E_n(t)$, then for any $t \in \brac{t_0 - \de, t_0}$, it holds that
\[
\Phi_n(t_0)-\Phi_n(t)\ge -\int_{t}^{t_0} E_n(\tau) d\tau.
\]
Since $\Phi_n(t)$ converges pointwise to $\Phi(t)$, it follows that 
\[
\Phi(t_0) - \Phi(t)
= \lim_{n \to \iny} \brac{\Phi_n(t_0)-\Phi_n(t)} 
\ge - \lim_{n \to \iny} \int_{t}^{t_0} E_n(\tau) d\tau.
\]
Assuming the local uniform convergence of $ E_n$ to 0 on $\brac{t_0 - \de, t_0} \supset \brac{t, t_0}$, we see that
$$\lim_{n \to \iny} \int_{t}^{t_0} E_n(\tau) d\tau 
=  \int_{t}^{t_0} \lim_{n \to \iny} E_n(\tau) d\tau = 0$$
and we may conclude that $\Phi(t_0) \ge \Phi(t)$, as desired.

It remains to justify the local uniform convergence of $E_n(t)$ to $0$, as described above. 
If $t_0 \in (0, T)$ is such that $\disp H_i = \frac 1 8 \int_{\R^{d+1}_+} \abs{U_i\pr{X, t_0}}^2 x_0^a \, dX> 0$, then an application of Lemma \ref{HnLowerBoundLemma} shows that there exists $N_i \in \N$ and $\de \in (0, \frac{t_0} 2]$ so that whenever $n \ge N_i$, \eqref{HnLowerBound} holds with $U_i$, $N_i$, and $H_i$ in place of $U$, $N$, and $H$, respectively.
Since $U_1, U_2 \in \mathfrak{C}\pr{\R^{d+1}_+ \times \pr{0, T}}$, then \eqref{ACFClass2} and \eqref{ACFClass1} hold and $\mathcal{D}(\cdot, U_i)$, $\mathcal{T}(\cdot, U_i)$ are in $L^1\pr{(0, t_0 )}$.
Assuming that $\de \le \eps$, we see that whenever $n \ge \max\set{N_1, N_2}$
\begin{align*}
\sup_{t \in \brac{t_0 - \de, t_0}} E_n(t)
&\le \frac{C_1  t_0^{a + \frac{d} 2}}{n} \sum_{\substack{i,j=1 \\ i\neq j}}^2 \frac{e^{\frac{d N_i \ln 2}{3} }}{H_i} \pr{\norm{\mathcal{M}^-(\cdot; U_i)}_{L^1\pr{(0, t_0)}} }
\pr{ \norm{\mathcal{D}(\cdot; U_j)}_{L^1\pr{(0, t_0)} }
+ \frac{4 T}{nM} \norm{\mathcal{T}(\cdot; U_j)}_{L^1\pr{(0, t_0)}}} \times \\
&\times \pr{ \norm{\mathcal{D}\pr{\cdot; U_i}}_{L^\iny\pr{[t_0 - \de, t_0]}}
+ \frac{4 T}{n M} \norm{\mathcal{T}\pr{\cdot; U_i}}_{L^\iny\pr{[t_0 - \de, t_0]}}}
\end{align*}
and the required version of uniform convergence follows.

Now assume that 
$\disp \pr{\int_{\R^{d+1}_+}  \abs{U_1(X,t)}^2 \dGt} \pr{\int_{\R^{d+1}_+}  \abs{U_2(X,t)}^2 \dGt} = 0$ for some values of $t$; let $\tau$ be the largest such $t$-value.
Without loss of generality, $\disp {\int_{\R^{d+1}_+}  \abs{U_1(X,\tau)}^2 \dGta} = 0$, from which it follows that $\disp {\int_{\R^{d+1}_+}  \abs{U_1(X,\tau)}^2 \dGnta} = 0$ for all $n \in \N$.
By \eqref{HnExpr}, for each $n \in \N$, $\disp {\int_{\BB{H}^{dn}_{\sqrt{M t}}} \abs{V_{1,n}\pr{Y}}^2 y_0^a } = 0$.
The arguments leading to \eqref{usenegativepart} are valid and the first two terms on the right-hand side vanish, so we see that
\begin{equation*}
\begin{aligned}
\int_{\BB{D}_{\sqrt{M\tau}}} |\nabla V_{1,n}|^2\Ga \, y_0^a
&\le \int_{\BB{D}_{\sqrt{M\tau}}} K_{1,n}^- V _{1,n}\Ga \, y_0^a.
\end{aligned}
\end{equation*}
Combining this inequality with \eqref{ballgradVnWeighted} and \eqref{ballKnVnWeighted} shows that
\begin{align*}
&\int_0^\tau \int_{\R^{d+1}_+} \abs{\gr U_1}^2 \dGnt dt \\
&\le \frac 1 n \int_0^\tau \int_{\R^{d+1}_+} U_1(X, t) J_1^-(X, t) \, \dGnt dt
- \frac 4 {Mn} \int_0^\tau \int_{\R^{d+1}_+} \pr{t \abs{\del_t U_1}^2 + \pr{X \cdot \gr U_1} \del_t U_1} \dGnt dt \\
&\le \frac{C_0}{n} \int_0^\tau \mathcal{M}^-\pr{t; U_1} dt 
+ \frac{4C_0}{Mn} \int_0^\tau t \mathcal{T}\pr{t; U_1} dt
+ \frac 1 2 \int_0^\tau \int_{\R^{d+1}_+} \abs{\gr U_1}^2 \dGnt dt,
\end{align*}
where we have used Cauchy-Schwarz, that $\abs{X}^2 \le M n t$ on the support of $\mathcal{G}_n$, and \eqref{pointwiseGBound}.
That is,
\begin{align*}
\int_0^\tau \int_{\R^{d+1}_+} \abs{\gr U_1}^2 \dGt dt
&= \lim_{n \to \iny} \int_0^\tau \int_{\R^{d+1}_+} \abs{\gr U_1}^2 \dGnt dt \\
&\le \lim_{n \to \iny} \frac {2 C_0}{n} \pr{\norm{\mathcal{M}^-\pr{\cdot; U_1}}_{L^1\pr{(0, \tau)}}
+ \frac {4T} M \norm{\mathcal{T}\pr{\cdot; U_1}}_{L^1\pr{(0, \tau)}}}
= 0,
\end{align*}
where we have used the assumption \eqref{ACFClass2} and that $\mathcal{T}(\cdot \ ; U_1) \in L^1\pr{(0, \tau)}$.

We conclude that $\Phi(t)=0$ for all $t\le \tau$, showing that $\Phi(t)$ is monotonically non-decreasing on $(0, \tau]$.
For $t \in (\tau, T)$, the arguments from the first case are valid, and the proof is complete.

\end{proof}

\begin{appendix}

\section{Limit lemmas}
\label{AppA}

Here we present the technical limit lemmas that are used in our elliptic-to-parabolic arguments.

\begin{lem}[Local uniform lower bound on $\mathcal{H}_n$]
\label{HnLowerBoundLemma}
Let $U$ be continuous on $\R^{d+1}_+ \times (0, T)$.
For $t_0 \in (0, T)$, if $\disp \int_{\R^{d+1}_+} \abs{U(X,t_0)}^2 x_0^a dX = 8H > 0$, then there exist $N \in \N$ and $\de > 0$ so that whenever $n \ge N$ and $t \in [t_0 - \de, t_0]$, it holds that
\begin{equation}
\label{HnLowerBound}
\mathcal{H}_n(t) \ge \mathcal{C} H  t_0^{- \frac{d+1+a} 2} e^{- \frac{d N \ln 2}{3} }.
\end{equation}
\end{lem}

\begin{proof}
We introduce
\begin{equation}
\label{HntDefn}
\begin{aligned}
\widetilde{\mathcal{H}}_n\pr{t} 
&:= \int_{\BB{D}^{d+1}_{\sqrt{Mnt}}}  \abs{U\pr{X, t}}^2 x_0^{a} \pr{1 - \frac{\abs{X}^2}{M n t}}^{\frac{dn-d-2} 2}  d{X}
\end{aligned}
\end{equation} 
so that $\disp \mathcal{H}_n(t) = \mathcal{C}_n t^{-\frac{d + 1+a}{2}} \widetilde{\mathcal{H}}_n\pr{t}$.
As shown in the proof of Lemma \ref{GaussianLimit}, $\disp \lim_{n \to \iny} \mathcal{C}_n = \mathcal{C}$, so there exists $N \in \N$ such that whenever $n \ge N$, it holds that $\disp \mathcal{C}_n \ge \frac 1 2 \mathcal{C}$.
Moreover, for all $t \in \brac{t_0 - \de, t_0}$, $t^{\frac{d + 1+a}{2}} \le t_0^{\frac{d + 1+a}{2}}$.
Therefore, it suffices to show that $\widetilde{\mathcal{H}}_n(t) \ge 2 H  e^{- \frac{d N \ln 2}{3} }$.
Observe that for any $n \in \N$, we have that 
\begin{equation}
\label{denomFn}
\begin{aligned}
\widetilde{\mathcal{H}}_n\pr{t} 
&\ge \int_{\BB{D}_{\sqrt{\frac{Mnt}2}}}  \abs{U\pr{X, t}}^2 x_0^{a} \pr{1 - \frac{\abs{X}^2}{M n t}}^{\frac{dn-d-2} 2}  d{X} 
\ge \int_{\BB{D}_{\sqrt{\frac{Mnt}2}}}  \abs{U\pr{X, t}}^2 x_0^{a} \, e^{- \frac{\ln 2\abs{X}^2}{3 t} } d{X},
\end{aligned}
\end{equation}
where have used a Taylor expansion to show that if $\abs{X}^2 \le \frac{M n t}{2}$, then
\begin{align*}
\log \brac{\pr{1 - \frac{\abs{X}^2}{M n t}}^{\frac{d n - d - 2}{2}}}
&\ge - \frac{d n - d - 2}{d n + 1+a} \frac{\abs{X}^2}{2 t} \brac{ 1 + \frac 1 2 \pr{\frac 1 2} + \frac 1 3 \pr{\frac 1 2}^2 + \ldots }
\ge - \frac{\ln 2 \abs{X}^2}{3 t} .
\end{align*}
Set $D_m = \set{X \in \R^{d+1}_+ : \abs{X} \le m, \ x_0 \ge \frac 1 m}$ and note that each $D_m$ is compact, the sets are nested, and $\disp \R^{d+1}_+ = \bigcup _{m \in \N} D_m$.
Thus, if $\disp 0 < 8H = \int_{\R^{d+1}_+} \abs{U\pr{X, t_0}}^2 x_0^a \, dX$, then there exists $m \in \N$ so that 
$$\int_{D_m} \abs{U\pr{X, t_0}}^2 x_0^a \, dX \ge 6 H.$$
Fix some $\mu < \min\set{\frac {t_0} 2, T - t_0}$.
Since $U$ and $x_0^a \,$ are continuous and $D_m \times \brac{t_0 - \mu, t_0 + \mu}$ is compact, then there exists $\de \in (0, \mu]$ so that whenever $x \in D_m$ and $\abs{t - t_0} \le \de$, it holds that
$$\abs{U(X, t) - U(X, t_0)} x_0^{\frac{a}{2}} < \sqrt{\frac{H}{\abs{D_m}}}.$$
In particular, if $t \in \brac{t_0 - \de, t_0}$, then
\begin{align*}
 \int_{D_m} \abs{U\pr{X, t}}^2 x_0^a \, dX
 &\ge \frac 1 2 \int_{D_m} \abs{U\pr{X, t_0}}^2 x_0^a \, dX
 - \int_{D_m} \abs{U\pr{X,t} - U\pr{X,t_0}}^2  x_0^a \, dX
 \ge 2H.
\end{align*}
Recall from \eqref{MDef} that $M \ge 2d$.
Thus, if $N \in \N$ satisfies $N \ge \frac{m^2}{d \pr{t_0 - \de}}$, then $D_m \su \BB{D}_{\sqrt{dn \pr{t_0 - \de}}} \su \BB{D}_{\sqrt{\frac{M N\pr{t_0 - \de}}{2} }}$.
It follows that for any $n \ge N$ and any $t \in \brac{t_0 - \de, t_0}$, we have from \eqref{denomFn} that
\begin{equation*}
 \begin{aligned}
 \widetilde{\mathcal{H}}_n\pr{t} 
&\ge \int_{\BB{D}_{\sqrt{\frac{MNt}2}}}  \abs{U\pr{X, t}}^2 x_0^{a} \, \exp\pr{- \frac{\ln 2\abs{X}^2}{3 t} } d{X} 
\ge e^{- \frac{d N \ln 2}{3} } \int_{D_m} \abs{U\pr{X,t}}^2 x_0^a \,  dX
\ge 2H e^{- \frac{d N \ln 2}{3} }
> 0,
\end{aligned}   
\end{equation*}
which leads to the conclusion.
    
\end{proof}

\begin{lem}[First uniform convergence lemma]
\label{weightedLpBound}
Given $\mathcal{E} : \pb{0, t_0} \to \R$, assume that for some $\al \in \R$ and some $p > 1$, we have $t^\al \mathcal{E}(t) \in L^p\pr{\pr{0, t_0}}$.
Then there exists $C_1\pr{a, \al,p, N}$ so that for any $\de \in \pr{0, t_0}$, it holds that
\begin{equation}
\label{fIntBound}
\sup_{t \in \brac{t_0 - \de, t_0}} \int_0^t \pr{\frac \tau t}^{\frac{N - 1+a}2} \mathcal{E}(\tau) d\tau 
\le C_1 \pr{t_0 - \de}^{-\al} \pr{ \frac{t_0}{N}}^{1 - \frac 1 p} \norm{t^\al \mathcal{E}}_{L^p\pr{[0, t_0]}}.
\end{equation}
Moreover, for any $N \ge N_0$, we can take $C_1 = C_1(a, \al, p, N_0)$.
\end{lem}

\begin{proof}
For any $t \in (0, t_0]$, an application of H\"older's inequality shows that
\begin{align*}
\int_0^t \pr{\frac \tau t}^{\frac{N -1+a}2} \mathcal{E}(\tau) d\tau 
&= t^{-\al} \int_0^t \pr{\frac \tau t}^{\frac{N -1+a - 2\al}2} \tau^\al \mathcal{E}(\tau) d\tau 
\le t^{-\al} \brac{\int_0^t \pr{\frac \tau t}^{\frac{N-1+a-2\al} 2 \frac{p}{p-1}}d \tau }^{1 - \frac 1 p} \norm{\tau^\al \mathcal{E}}_{L^p\pr{[0, t]}} \\
&= t^{1 - \frac{1}{p}- \al}\brac{\int_0^1 x^{\frac{\pr{N-1+a-2\al}p}{2(p-1)}} dx }^{1 - \frac 1 p} \norm{\tau^\al \mathcal{E}}_{L^p\pr{[0, t]}} \\
&= t^{-\al} \brac{ \frac{2t(p-1)}{\pr{N+1+a-2\al}p -2}}^{1 - \frac 1 p} \norm{\tau^\al \mathcal{E}}_{L^p\pr{[0, t]}},
\end{align*}
and the conclusion follows.   
\end{proof}

\begin{lem}[Second uniform convergence lemma]
\label{uniformConvLemma}
For $f : \R^{d+1}_+ \times \pr{0, T} \to \R$, define
\begin{align*}
\mathcal{E}_n(t) &= \int_{\R^{d+1}_+} \abs{f(X, t)} \, \dGnt \\
\mathcal{E}(t) &= \int_{\R^{d+1}_+} \abs{f(X, t)} \, \dGt.
\end{align*}
If there exist $t_0 \in (0, T)$ and $\de_0 \in (0, t_0)$ such that $\mathcal{E} \in L^\iny\pr{\brac{t_0 - \de_0, t_0}}$, then there exists $\de \in (0, \de_0]$ so that the sequence $\set{\mathcal{E}_n}_{n=1}^\iny$ converges uniformly to $\mathcal{E}$ on $\brac{t_0 - \de, t_0}$.
\end{lem}

\begin{proof}
For each $n \in \N$, let $r_n = \sqrt{4 (t_0 - \de) \sqrt{\frac{ n}{\log n}}}$ so that $\set{r_n}_{n=1}^\iny$ is an increasing sequence for which $\disp \lim_{n \to \iny} r_n = \iny$. 
Since
\begin{align*}
\mathcal{E}_n(t) - \mathcal{E}(t)
&= \int_{\R^{d+1}_+} \abs{f(X, t)} \, x_0^a \, \brac{\mathcal{G}_n(X,t) - \mathcal{G}(X,t)} dX \\
&= \int_{\BB{D}_{r_n}} \abs{f(X, t)}  \pr{\frac{\mathcal{G}_n(X,t)}{\mathcal{G}(X,t)} - 1}  \dGt
+ \int_{\R^{d+1}_+ \setminus \BB{D}_{r_n}} \abs{f(X, t)}  \, x_0^a \, \brac{\mathcal{G}_n(X,t) - \mathcal{G}(X,t)} dX, 
\end{align*}
then Lemma \ref{GaussianLimit} shows that
\begin{align}
\label{EfuncDifEst}
\abs{\mathcal{E}_n(t) - \mathcal{E}(t)}
&\le \norm{\frac{\mathcal{G}_n(\cdot, t)}{\mathcal{G}(\cdot, t)} - 1}_{L^\iny(\BB{D}_{r_n})} \mathcal{E}(t)
+ \pr{1 + C_0} \int_{\R^{d+1}_+ \setminus \BB{D}_{r_n}} \abs{f(X, t)} \dGt.
\end{align}

We start by bounding the second term.
Because $\mathcal{E} \in L^\iny\pr{\brac{t_0 - \de_0, t_0}}$, then for each $n \in \N$ we may define
$$\eps_n = \sup_{t \in \brac{t_0 - \de_0, t_0}} \int_{\R^{d+1}_+ \setminus \BB{D}_{r_n}} \abs{f(X, t)} \dGt \ge 0.$$
As $\set{\eps_n}_{n=1}^\iny$ is decreasing and bounded from below, then set $\disp \eps_0 := \lim_{n \to \iny} \eps _n$.

If $\eps_0 = 0$, then for any $\eps > 0$ there exists $N \in \N$ so that for all $t \in \brac{t_0 - \de_0, t_0}$ and all $n \ge N$, we have
\begin{equation}
\label{exteriorIntegralBound}
\int_{\R^{d+1}_+ \setminus \BB{D}_{r_n}} \abs{f(X, t)} \dGt < \eps.
\end{equation}

If $\eps_0 > 0$, fix some $\eps \in \pr{0, \eps_0}$.
Then for every $n \in \N$, there exists a non-empty set $T_n \su \brac{t_0 - \de_0, t_0}$ so that for every $t \in T_n$,
$$\int_{\R^{d+1}_+ \setminus \BB{D}_{r_n}} \abs{f(X, t)} \dGt \ge \eps.$$ 
With $t_n = \sup T_n$, it holds that $\disp \int_{\R^{d+1}_+ \setminus \BB{D}_{r_n}} \abs{f(X, t_n)} x_0^a \, \mathcal{G}(X, t_n) dX \ge \eps$ as well.
Since the sets $\set{T_n}_{n=1}^\iny$ are nested, then $\set{t_n}_{n=1}^\iny$ is decreasing, so there exists $\disp t_* = \lim_{n \to \iny} t_n \in \brac{t_0 - \de_0, t_0}$.
If $t_* = t_0$, then $t_n = t_0$ for all $n \in \N$ and we see that $\disp \int_{\R^{d+1}_+ \setminus \BB{D}_{r_n}} \abs{f(X, t_0)} x_0^a \, \mathcal{G}(X, t_0) dX \ge \eps$ for all $n \in \N$.
However, for each $t \in \brac{t_0 - \de_0, t_0}$, it holds that $\disp \lim_{n \to \iny} \int_{\R^{d+1}_+ \setminus \BB{D}_{r_n}} \abs{f(X, t)} \dGt = 0$, so this is impossible and we may conclude that $t_* < t_0$.
Therefore, there exists $N \in \N$ so that whenever $n \ge N$, we have $t_n < \frac 1 2 \pr{t_0 + t_*}$ which means that $T_n \su [t_0 - \de_0, \frac 1 2 \pr{t_0 + t_*})$.
In particular, for all $t \in \brac{\frac 1 2 \pr{t_0 + t_*}, t_0}$ and all $n \ge N$, \eqref{exteriorIntegralBound} holds.

In either case, we have shown that for any $\eps > 0$, there exists positive $\disp \de = \begin{cases} \de_0 & \eps_0 = 0 \\
    \frac{t_0 - t_*} 2 & \eps_0 > 0
\end{cases}$ and $N \in \N$ so that for all $n \ge N$ and all $t \in \brac{t_0 - \de, t_0}$, \eqref{exteriorIntegralBound} holds.
Next, we estimate the first term in \eqref{EfuncDifEst}, namely $\disp \norm{\frac{\mathcal{G}_n(\cdot, t)}{\mathcal{G}(\cdot, t)} - 1}_{L^\iny(\BB{D}_{r_n})}$.

Since $\BB{D}_{r_n} \su \BB{D}^{d+1}_{\sqrt{Mnt}}$ for all $t \in \brac{t_0 - \de, t_0}$, then from \eqref{GsdnDefn}, on $\BB{D}_{r_n}$ it holds that
\begin{align*}
  & \frac{\mathcal{G}_n(X,t) }{\mathcal{G}(X,t) }
  = \frac{\mathcal{C}_n}{\mathcal{C}} \exp\brac{z + (m-\al) \log\pr{1 - \frac z m}},
\end{align*}
where we have recalled \eqref{MDef} and introduced $m = \frac{dn+1+a}{2}$, $\al = \frac {d+3 + a} 2$, and $z = \frac{\abs{X}^2}{4 t} \in \brac{0, z_n}$ with $z_n = \frac{r_n^2}{4 t}$.
Let $g(z) = z + (m-\al) \log\pr{1 - \frac z m}$ and note that $g$ increases in $\pr{0, \al}$, then decreases in $\pr{\al, z_n}$.
A Taylor expansion shows that
\begin{align*}
    g(z) 
    &\le g(\al)
    = \al + (m-\al) \log\pr{1 - \frac \al m}
    = \al - (m - \al) \brac{\pr{\frac \al m} + \frac 1 2 \pr{\frac \al m}^2 + \frac 1 3 \pr{\frac \al m}^3 + \ldots} \\
    &= \frac{\al^2} m\brac{1 
    - \pr{1 - \frac{\al}{m}} \pr{ \frac 1 2  + \frac 1 3 \frac{\al}{m} + \ldots}}
    \le \frac{\al^2} m
    \le \frac{c_1}{n}
\end{align*}
while $g(z) \ge \min\set{0, g(z_n)}$ and 
\begin{align*}
    g(z_n)
    &= z_n + (m-\al) \log\pr{1 - \frac {z_n} m}
    = z_n - (m - \al) \brac{\pr{\frac {z_n} m} + \frac 1 2 \pr{\frac {z_n} m}^2 + \frac 1 3 \pr{\frac {z_n} m}^3 + \ldots} \\
    &= - \frac {z_n^2} m \brac{\pr{1 - \frac{\al} m} \pr{\frac 1 2  + \frac 1 3 \pr{\frac {z_n} m} + \ldots}
    - \frac{\al}{z_n}  }.
\end{align*}
Since $r_n = \sqrt{4 (t_0 - \de) \sqrt{\frac{ n}{\log n}}} \le \sqrt{M n (t_0 - \de)}$, then $z_n = \frac{r_n^2}{4t} = \frac{4 (t_0 - \de) }{4 t} \sqrt{\frac{ n}{\log n}} \le \sqrt{\frac{ n}{\log n}}$ and we see that $\frac{z_n} m \le \frac{c_2}{\sqrt {n \log n}}$ while $\frac{z_n^2} m \le \frac{c_3}{\log n}$.
Therefore, $g(z_n) \ge - \frac{c_4}{\log n}$ and it follows that
\begin{align*}
  \frac{\mathcal{C}_n}{\mathcal{C}} e^{- \frac{c_4}{\log n}}
  &\le \frac{\mathcal{G}_n(X,t) }{\mathcal{G}(X,t) }
  \le \frac{\mathcal{C}_n}{\mathcal{C}} e^{\frac{c_1}{n}}.
\end{align*}
That is, for any $(X,t) \in \BB{D}_{r_n} \times \brac{t_0 - \de_0, t_0}$, we have
\begin{align*}
    \abs{\frac{\mathcal{G}_n(X,t) }{\mathcal{G}(X,t) } - 1}
    &\le \abs{\frac{\mathcal{C}_n}{\mathcal{C}} e^{- \frac{c_4}{\log n}} - 1}
    = \frac{\abs{ \mathcal{C}_n e^{- \frac{c_4}{\log n}} - \mathcal{C}}}{\mathcal{C}}.
\end{align*}
As $\disp \lim_{n \to \iny} \mathcal{C}_n = \mathcal{C}$ and $\disp \lim_{n \to \iny} e^{- \frac{c_4}{\log n}} = 1$, then for any $\eps > 0$, there exists $N \in \N$ so that for all $n \ge N$ and all $t \in \brac{t_0 - \de, t_0}$, it holds that
\begin{align*}
    \norm{\frac{\mathcal{G}_n(\cdot, t)}{\mathcal{G}(\cdot, t)} - 1}_{L^\iny(\BB{D}_{r_n})}
 < \eps.
\end{align*}
Combining this estimate with \eqref{exteriorIntegralBound} in \eqref{EfuncDifEst} leads to the conclusion.
\end{proof}

\end{appendix}

\bibliography{refs}
\bibliographystyle{alpha}

\end{document}